\documentclass[11pt, twoside,leqno]{article}

\usepackage{amssymb}
\usepackage{amsmath}
\usepackage{amsthm}
\usepackage{yhmath}
\usepackage{mathrsfs}
\usepackage{color}
\usepackage{txfonts}
\usepackage{indentfirst}
\usepackage{booktabs}
\usepackage{graphicx}
\usepackage{subfigure}
\usepackage{tikz}
\usepackage{enumerate}
\usepackage{anysize}
\usepackage{setspace}
\usepackage{float}
\usepackage{color}
\usepackage{chngcntr}

\allowdisplaybreaks

\def\esssup{\mathop\mathrm{\,ess\,sup\,}}

\def\1bf{\mathbf{1}}
\def\0bf{\mathbf{0}}

\def\red{\color{red}}

\newtheorem{theorem}{Theorem}[section]
\newtheorem{lemma}[theorem]{Lemma}

\theoremstyle{definition}
\newtheorem{remark}[theorem]{Remark}

\newtheorem{definition}[theorem]{Definition}
\renewcommand{\appendix}{\par
   \setcounter{section}{0}%
   \setcounter{subsection}{0}%
   \setcounter{subsubsection}{0}%
   \gdef\thesection{\@Alph\c@section}%
   \gdef\thesubsection{\@Alph\c@section.\@arabic\c@subsection}%
   \gdef\theHsection{\@Alph\c@section.}%
   \gdef\theHsubsection{\@Alph\c@section.\@arabic\c@subsection}%
   \csname appendixmore\endcsname
 }

\numberwithin{equation}{section}

\begin{document}

\arraycolsep=1pt

\title{\bf\Large
Refining Image Edge Detection via Linear Canonical Riesz Transforms
\footnotetext{\hspace{-0.35cm} 2020 {\it Mathematics Subject
Classification}.
Primary 42B20; Secondary 42B35, 42A38, 94A08.
\endgraf {\it Key words and phrases.} linear canonical transform,
linear canonical Riesz transform, 
linear canonical multiplier, sharpness, edge detection.
\endgraf This project is partially supported by
the National Key Research and Development Program of China
(Grant No. 2020YFA0712900),
 the National Natural Science Foundation of China
(Grant Nos.\  12431006, 12371093, and 12471090), and the
China Postdoctoral Science Foundation (Grant No.
2024M760238).}}
\author{Shuhui Yang, Zunwei Fu, Dachun Yang\footnote{Corresponding
author, E-mail: \texttt{dcyang@bnu.edu.cn}/{\red \today}/Final version.},\ \ Yan Lin and  Zhen Li}
\date{}
\maketitle

\vspace{-0.8cm}

\begin{center}
\begin{minipage}{13cm}
{\small{\textbf{Abstract}}\quad}
Combining the linear canonical transform and the Riesz
transform, we introduce the  linear canonical Riesz 
transform (for short, LCRT), which is further proved to 
be a linear canonical multiplier.   Using this LCRT multiplier, 
we conduct numerical simulations on images. Notably, 
the LCRT multiplier significantly reduces the complexity of 
the algorithm. Based on these we introduce the new concept of
the sharpness $R^{\rm E}_{\rm sc}$ of the edge strength and continuity of images 
associated with the LCRT  and, using it, we propose a new LCRT 
image edge detection method (for short, LCRT-IED method) 
and provide its mathematical foundation. Our experiments indicate 
that this sharpness $R^{\rm E}_{\rm sc}$ characterizes 
the macroscopic trend of edge  variations of the image under
 consideration, while this new LCRT-IED 
method  not only controls the overall edge strength and continuity 
of the image, but also excels in feature extraction in some local 
regions.  These highlight the fundamental differences between the LCRT and 
the Riesz transform, which are precisely due to the
multiparameter of the former. This new LCRT-IED method might be of significant 
importance  for image feature extraction, image matching, and 
image refinement.
\end{minipage}
\end{center}
\vspace{0.2cm}

\tableofcontents

\vspace{0.2cm}
\section{Introduction}\label{sec1}
As is well known, the Fourier transform, proposed by Joseph
Fourier in 1807, was originally intended to solve the heat
conduction equation. Moreover, over time, the range of
applications of the Fourier transform has rapidly expanded,
evolving into a key technology in scientific research and
engineering. Over the course of more than two centuries,
the Fourier transform has not only occupied a significant
position in mathematics but has also become a cornerstone
of signal processing and analysis, particularly in analyzing
and processing stationary signals (see, for example, Plonka et al. \cite{ppst23} and Sun et al.
\cite{ccs23,ccls23} for some recent applications of the Fourier
transform in graph signal processing). To recall the precise definition
of the Fourier transform, we denote the set of all Schwartz
functions equipped with the well-known topology determined
by a countable family of norms by $\mathscr{S}(\mathbb{R}^n)$,
and also by $\mathscr{S}'(\mathbb{R}^n)$ the set of all continuous linear
functionals on $\mathscr{S}(\mathbb{R}^n)$ equipped with the
weak$-\ast$ topology.

\begin{definition}\label{df.FT}
For any $f \in \mathscr{S} ({\mathbb{R}^n})$, its \emph
{Fourier transform} $\mathscr{F}(f)$ or $\widehat f$ is defined
by setting, for any $\boldsymbol{x}\in{\mathbb{R}^n}$,
$$\widehat f(\boldsymbol{x}):= \mathscr{F}(f)
(\boldsymbol{x}):=\frac{1}{(\sqrt{2\pi}
)^n}\int_{\mathbb{R}^n}f(\boldsymbol{x})
{e^{-i\boldsymbol{x}\cdot\boldsymbol{y}}}\,d\boldsymbol{y}$$
and its \emph{inverse Fourier transform}  ${f^\vee}$ is defined
by setting, for any $\boldsymbol{x}\in
{\mathbb{R}^n}$, $f^\vee(\boldsymbol{x}):=
\widehat f(-\boldsymbol{x})$.
\end{definition}

As the scopes and the subjects of research continue to
expand, the limitations of the Fourier transform in
handling non-stationary signals have become increasingly
evident. These limitations manifest primarily in the
following way. The Fourier transform is a global transform
that provides the overall spectrum of a signal but fails to
capture the local time-frequency characteristics crucial for
such signals. Consequently, it cannot accurately represent
the frequency behavior of signals as they evolve over time.
This renders the Fourier transform inadequate for analyzing
non-stationary and time-varying signals. In order to overcome these
limitations,  the fractional Fourier transform,
the short-time Fourier transform, the Wigner distribution,
and the wavelet transform were proposed. Among these, the fractional Fourier
transform, as an extension of the Fourier transform, offers
increased flexibility by introducing additional parameters,
particularly when dealing with non-stationary signals (see,
for example, \cite{mk87, z24}). Notably, the research on
the multidimensional fractional Fourier transform has obtained
an increasing attention in recent years (see, for example,
\cite{krz21,krz22,krz23,ozk01,z2018, z2019}). To be precise,
Namias \cite{n1980} introduced  the following
 multidimensional fractional Fourier transform;
 see also the recent monograph \cite{z24} of Zayed
 for a systematic study on the fractional Fourier transform.

\begin{definition}\label{frft}
Let  $\boldsymbol{\alpha}:=(\alpha_1,\ldots,\alpha_n)
\in [0,2\pi)^n$ and $f\in\mathscr{S}(\mathbb{R}^n)$.
The  \emph{fractional Fourier transform} (for short, FrFT)
$\mathcal{F}_{\boldsymbol{\alpha}}(f)$,
with order $\boldsymbol{\alpha}$, of $f$ is defined by
setting, for any $\boldsymbol{u}\in\mathbb{R}^n$,
$$\mathcal{F}_{\boldsymbol{\alpha}}(f)(\boldsymbol
{u}):=\int_{\mathbb{R}^n}f(\boldsymbol{x})K_{
\boldsymbol{\alpha}}(\boldsymbol{x},\boldsymbol{u})\,
d\boldsymbol{x},$$ where, for any $\boldsymbol{x}:=
(x_1,\ldots,x_n),\boldsymbol{u}:=(u_1,\ldots,u_n)\in
\mathbb{R}^n$, $$K_{\boldsymbol{\alpha}}(\boldsymbol
{x},\boldsymbol{u}):=\prod^n_{k=1}K_{\alpha_k}
(x_k,u_k)$$ and, for any $k\in \{1,\ldots,n\}$,
\begin{equation*}
K_{\alpha_k}(x_k,u_k):=\left\{
\begin{array}
[c]{ll}
\frac{c(\alpha_k)}{\sqrt{2\pi}}e^{\pi i[a(\alpha_k)(x_k^2
+u_k^2)-2b(\alpha_k)x_ku_k]}
&\quad \mathrm{if}\ \alpha_k\not\neq0,\pi ,\\
\delta(x_k-u_k) &\quad \mathrm{if}\ \alpha_k=0,\\
\delta(x_k+u_k) &\quad \mathrm{if}\ \alpha_k=\pi
\end{array}
\right.
\end{equation*}
with  $a(\alpha_k):=\frac{\cot(\alpha_k)}{2}$, $b(\alpha_k):=
\csc(\alpha_k)$, $c(\alpha_k):=\sqrt{1-i\cot(\alpha_k)}$, and
$\delta$ being the \emph{Dirac measure} at $0$.
\end{definition}
Observe that,
when $\boldsymbol{\alpha}=(\frac{\pi}{2},\ldots,\frac{\pi}{2})$,
the FrFT becomes the classical Fourier transform.
The FrFT is currently employed in various fields
of scientific research and engineering, including
swept filters, artificial neural networks, wavelet
transforms, time-frequency analysis, and complex
transmission (see, for example, \cite{cfgw21,
flyy24, mun02,  sdj11, 
tlw10, Yetik}).

The linear canonical transform (for short, LCT), as
an extension of the FrFT, is a more comprehensive
novel mathematical transform for the analysis and
the processing of non-stationary signals. The one-dimensional
(for short, 1D) LCT is characterized by four parameters
$\{a,b,c,d\}$, and it is also referred to as the ABCD
matrix transform \cite{b1996}. Additionally, the LCT
is also called the generalized Fresnel transform \cite
{ja1996}, the quadratic phase system \cite{b1978},
 and the extended FrFT \cite{hll1997}.

In the 1D Euclidean space, the LCT features three free
parameters, which enhances its flexibility compared
to the FrFT. In 1961, Bargmann \cite{b1961} first
introduced the LCT with complex parameters. Collins
\cite{c1970}, along with Moshinsky and Quesne
\cite{mq71},  formally presented the LCT in the
1970s. Initially, the LCT was employed to solve
differential equations and analyze optical systems.
Subsequently, the relation between the LCT and existing
transforms and implementation methods of the LCT in
optical systems have been established (see, for example, \cite{
bko1997,hkos15}). With the rapid development
of the FrFT in the 1990s, the LCT gradually garnered
attention in the field of signal processing.   Barshan
et al. \cite{bko1997} gave the application of
the LCT in optimal filter design in 1997.   Pei
et al.  \cite{pd2000} explored the discrete algorithms
of the LCT, providing two types of discrete forms and
applying this discrete algorithm to filter design and
other areas; \cite{pd2002} gave the eigenvalues and
the eigenfunctions of the LCT, the LCT is divided into
seven cases based on different parameter assumptions,
and the eigenfunctions are then applied to the self-imaging
problem in optics. For more studies of LCTs,
we refer to \cite{ddp24,hs22,hs222,yz24}.

In multidimensional signal processing, it is crucial to extend
commonly used 1D signal processing tools to accommodate
multidimensional signals. The multidimensional LCT is a more
general form of the multidimensional Fourier transform and FrFT,
and it has been widely used in various applications, including
filter design, sampling, image processing, and pattern
recognition  (see, for example, \cite{dp1999,
kock08,ko1998,sko98}). The multidimensional LCT was
 introduced in \cite{dp1999} as follows.

\begin{definition}\label{nDLCT}
Let $\boldsymbol{A}:=(A_1,\ldots,A_n)$
with the matrix $A_k:=\begin{bmatrix}
{a_k}&{b_k}\\{c_k}&{d_k}
\end{bmatrix}\in{M_{2 \times 2}}(\mathbb{R}
)$ ($M_{2\times 2}$ denotes the set of all $2\times 2$
real matrixes with determinant equalling 1) for any $k\in\{1,\ldots,n\}$.
For any $f \in \mathscr{S}({\mathbb{R}^n})$,
the \emph{linear canonical transform}
${\mathscr{L}_{\boldsymbol{A}}}$
is defined by setting, for any $\boldsymbol{u}\in\mathbb{R}^n$,
$$\mathscr{L}_{\boldsymbol A}(f)(\boldsymbol{u}):=
\int_{\mathbb{R}^n}f(\boldsymbol{x})K_{\boldsymbol{A}}
(\boldsymbol{x},\boldsymbol{u})\,d\boldsymbol{x},$$
where, for any $\boldsymbol{x}:=(x_1,\ldots,x_n),
\boldsymbol{u}:=(u_1,\ldots,u_n)\in \mathbb{R}^n$,
$$K_{\boldsymbol{A}}(\boldsymbol{x},\boldsymbol{u})
:=\prod^n_{k=1}K_{A_k}(x_k,u_k)$$ and, for any
$k\in\{1,\ldots,n\}$,
\begin{equation*}
K_{A_k}(x_k,u_k):=\left\{
\begin{array}
[c]{ll}
C_{A_k}e_{b_k,a_k}({x_k})
e_{-b_k}\left( x_k,u_k \right)e_{b_k,d_k}({u_k})
&\quad \mathrm{if}\ b_k\not\neq0,\\
\sqrt {d_k} e^{\frac{i}{2}\left( {c_k}{d_k}u_k^2
\right)}\delta(x_k-d_ku_k) &\quad \mathrm{if}\ b_k=0
\end{array}
\right.
\end{equation*}
with $C_{A_k}:=\sqrt {\frac{1}{i2\pi b_k}}$,
$e_{b_k,a_k}({x_k}):=e^{i\frac{a_k}{2b_k}x_k^2}$,
$e_{-b_k}(u_k, x_k):= e^{-i{\frac{{u_k}{x_k}}
{b_k}} }$,
$e_{b_k,d_k}({u_k}):= e^{i\frac{d_k}{2b_k}u_k^2}$,
and  $\delta$ as in Definition \ref{frft}.
\end{definition}

In Definition \ref{nDLCT}, when the parameter matrices
$\boldsymbol{A}:=({A_1,\ldots,A_n})$ take several different
specific values, the LCT, as a general integral transform,  
reduces to several classical well-known integral transforms 
(see Remark \ref{lctspecial}). These
transforms include the Fourier transform, the FrFT, the scaling
operator, the chirp multiplication, and the Fresnel transform.
Such transforms play a crucial role in solving problems
related to electromagnetic, acoustic, and wave propagation
(see, for example, \cite{asf2017,bshyh2007,ss04}).

\begin{remark}\label{lctspecial}
Let $j\in\{1,\ldots,n\}$, $\boldsymbol{A}:=(A_1,\ldots,A_n)$ with
$A_j:=\begin{bmatrix}
{a_j}&{b_j}\\{c_j}&{d_j}
\end{bmatrix}\in {M_{2 \times 2}}
(\mathbb{R})$, and 
$f \in { \mathscr{S}({\mathbb{R}^n})}$.
\begin{enumerate}[(i)]
\item If $A_j :=
\begin{bmatrix}
0&1\\{-1}&0
\end{bmatrix}$ for all $j$, then the LCT
$\mathscr{L}_{\boldsymbol{A}}$ in this case
reduces to the Fourier transform
 multiplied by $(\sqrt{-i})^n$. That is,
 for any $\boldsymbol{u}\in\mathbb{R}^n$,
$\mathscr{L}_{\boldsymbol{A}}(f)(\boldsymbol{u})={(\sqrt
{-i})^n}\mathscr{F}(f)(\boldsymbol u)$;
\item If $A_j:= \begin{bmatrix}	
{\cos{\alpha_j}}\,&{\sin{\alpha _j}}\\
{-\sin{\alpha_j}}\,&{\cos{\alpha_j}}
\end{bmatrix}$ for all $j$, then the LCT
$\mathscr{L}_{\boldsymbol{A}}$ in this case
reduces to the FrFT multiplied by $n$ fixed phase factors.
That is,  for any $\boldsymbol{u}\in\mathbb{R}^n$, $\mathscr{L}_
 {\boldsymbol{A}}(f)(\boldsymbol u)=\prod\limits_{j=1}^n
 \sqrt{ e^{-i\alpha _j}}\mathscr{F}_{\boldsymbol{\alpha}}(f)
 (\boldsymbol u)$;
\item If $A_j:=\begin{bmatrix}
\sigma &0\\0&{\frac{1}{\sigma}}\end{bmatrix}$ for all $j$,
then the LCT $\mathscr{L}_{\boldsymbol{A}}$ in this case 
reduces to the scaling operator;
\item If $A_j:= \begin{bmatrix}
1&\frac{z\lambda}{2\pi}\\0&1\end{bmatrix}$ for all $j$,
then the LCT $\mathscr{L}_{\boldsymbol{A}}$ in this case
reduces to the Fresenel transform,
where  $\lambda$ is the wave length and $z$ is the
dissemination distance;
\item If $A_j:=\begin{bmatrix}
1&0\\{-q}&1\end{bmatrix}$ for all $j$, then
the LCT $\mathscr{L}_{\boldsymbol{A}}$ in this case
reduces to the chirp multiplication.
\end{enumerate}
\end{remark}

\begin{definition}\label{cf}
Let $\boldsymbol{a}:=(a_1,\ldots,a_n),
\boldsymbol{b}:=(b_1,\ldots,b_n)\in\mathbb R^n$.
The \emph{chirp functions}
$e_{\boldsymbol{a},\boldsymbol{b}}$  and $e_{\boldsymbol{a}}$
are defined, respectively, by setting,  for any $\boldsymbol{x}:=(x_1,\ldots,x_n),
\boldsymbol{t}:=(t_1,\ldots,t_n)\in \mathbb{R}^n$,
\begin{equation}\label{1eq}
{e}_{\boldsymbol{a},\boldsymbol{b}}( \boldsymbol x
) :=e^{i\sum_{j=1}^n\frac{b_j}{2a_j}x_j^2},
\end{equation}
and
\begin{equation}\label{2eq}
e_{\boldsymbol{a}}(\boldsymbol{x},\boldsymbol{t}):=
e^{i\sum_{j=1}^n\frac{1}{a_j}x_jt_j}.
\end{equation}
\end{definition}

\begin{remark}\label{rem-lct}
Let $\boldsymbol{A}:=(A_1,\ldots,A_n
)$ with $A_j:=
\begin{bmatrix}
{{a_j}}&{{b_j}}\\{{c_j}}&{{d_j}}
\end{bmatrix}\in {M_{2\times2}}
(\mathbb{R})$ and $b_j\neq 0$ for any $j\in \{1,\ldots,n\}$.
Define $C_{\boldsymbol{A}}:=\prod^n_{k=1}C_{A_k}$,
where $C_{A_k}$ is the same as in Definition \ref{nDLCT}.
Let $\boldsymbol{a}:=(a_1,\ldots,a_n)$,
$\boldsymbol{b}:=(b_1,\ldots,b_n)$,
$\boldsymbol{d}:=(d_1,\ldots,d_n)$, 
and $f\in{\mathscr{S}(\mathbb{R}^n)}$.
  By Definitions \ref{df.FT}
and \ref{nDLCT}, the LCT $\mathscr{L}_{\boldsymbol{A}}(f)$
of the signal $f$ is
exactly the Fourier transform of the signal
$g:={e}_{\boldsymbol{b},
\boldsymbol{a}}f$
multiplied by the chirp functions, where ${e}_{\boldsymbol{b},
\boldsymbol{a}}$ is the same as in \eqref{1eq}
with exchanging the position of  $\boldsymbol a$ and
$\boldsymbol{b}$. In other words,
the LCT $\mathscr{L}_{\boldsymbol{A}}(f)$
 of $f$ can be decomposed into that, for any
 $\boldsymbol x\in \mathbb{R}^n$,
\begin{equation}\label{equation1.2}
 \mathscr{L}_{\boldsymbol A}(f)(\boldsymbol x)=
 {\left( \sqrt {2\pi }  \right)^n}C_{\boldsymbol A}
 {e}_{\boldsymbol{b},\boldsymbol{d}}
 (\boldsymbol x)\mathscr{F}\left({e}_{\boldsymbol{b},
 \boldsymbol{a}}f \right)
 (\widetilde{\boldsymbol{x}}),
\end{equation}
where ${e}_{\boldsymbol{b},
\boldsymbol{d}}$ is the same as in \eqref{1eq}
with $\boldsymbol{a}$ and $\boldsymbol{b}$
replaced, respectively,  by $\boldsymbol{b}$
and $\boldsymbol{d}$
and where $\boldsymbol{\widetilde x}:=
\frac{\boldsymbol x}{\boldsymbol b}:= (\frac{x_1}{b_1},
\ldots, \frac{x_n}{b_n})$.	
Based on \eqref{equation1.2}, we observe
that the LCT $\mathscr{L}_{\boldsymbol{A}}(f)$ of
a function $f\in{\mathscr{S}(\mathbb{R}^n)}$  consists of the following four simpler operators
(as illustrated in Figure \ref{img1}):

\begin{enumerate}[(i)]
\item The multiplication by a chirp function, that is,
for any $\boldsymbol t\in \mathbb{R}^n$,
$g(\boldsymbol{t}) := {e}_{\boldsymbol{b},\boldsymbol{a}}
f(\boldsymbol{t})$;
\item  The Fourier transform, that is, for any $\boldsymbol x
\in \mathbb{R}^n$, $\widehat{g}(\boldsymbol{x})
:=(\mathscr{F}g)(\boldsymbol x)$;
\item The scaling operator, that is, for any $\boldsymbol x\in \mathbb{R}^n$,
$\widetilde{g}(\boldsymbol x)
:= g(\frac{\boldsymbol x}{\boldsymbol b})$;
\item The multiplication by a chirp function again, that is,
for any $\boldsymbol x\in \mathbb{R}^n$,
$$\mathscr{L}_A(f)(\boldsymbol{x}) :=
\left(\sqrt{2\pi}\right)^n{C_A}{{e}_{\boldsymbol{b},\boldsymbol{d}}}\widetilde{g}
(\boldsymbol{x}).$$
\end{enumerate}
\begin{figure}[H]
\centering
\includegraphics[width=0.8\linewidth]{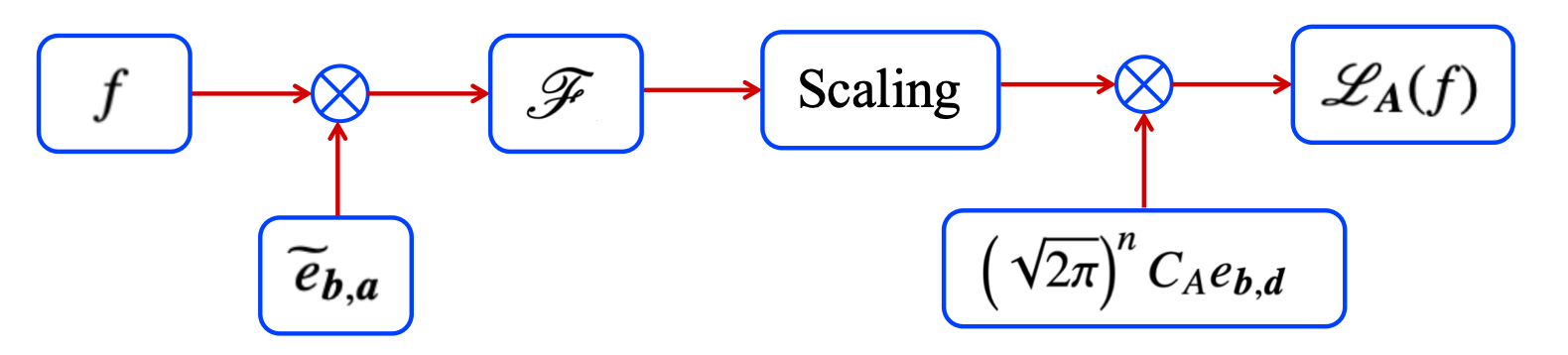}
\caption{The decomposition of the LCT $\mathscr{L}_{\boldsymbol A}$.}
\label{img1}
\end{figure}
\end{remark}
		
The definition of the multidimensional inverse LCT in \cite{mac22} is
as follows.	
	
\begin{definition}\label{nLCT}
Let $\boldsymbol{A}:=(A_1,\ldots ,A_n)$
and $\boldsymbol{A}^{-1}:=(A_1^{-1},\ldots ,A_n^{-1})$,
respectively, with the matrices $A_k:=\begin{bmatrix}
{a_k}&{b_k}\\
{c_k}&{d_k}
\end{bmatrix},A_k^{-1}:=
\begin{bmatrix}
{d_k}&{-b_k}\\
{-c_k}&{\ a_k}
\end{bmatrix}\in{M_{2 \times 2}}
( \mathbb{R})$ for any $k\in\{1,\ldots,n\}$.
For any $f\in\mathscr{S}(\mathbb{R}^n)$,
its \emph{multidimensional inverse} LCT
$({\mathscr{L}_{\boldsymbol{A}}})^{-1}(f)$
is defined by setting, for any $\boldsymbol{x}\in \mathbb{R}^n$,
$ \left({\mathscr{L}_{\boldsymbol{A}}}\right)^{-1}(f)
(\boldsymbol x):={\mathscr{L}_{\boldsymbol{A}^{-1}}}(f)
(\boldsymbol x).$
\end{definition}	

The Hilbert transform occupies a significant position in 
Fourier analysis and plays a crucial role in signal processing. 
It has been extensively applied to modulation theory, edge 
detection, and filter design (see, for example, \cite{b19921,b19922,g1946,ppst23}). In recent years, the FrFT and the LCT are 
respectively combined with the Hilbert transform to lead 
to the fractional Hilbert transform and the linear canonical 
Hilbert transform (for short, LCHT), which yield significant 
results (see, for example, \cite{cfgw21,z1998}). The Riesz 
transform serves as a natural extension of the Hilbert transform 
in the $n$-dimensional Euclidean space. In 2023, Fu et al. 
\cite{fglwy23} combined the FrFT with the Riesz transform, 
introducing the fractional Riesz transform and applying it 
to edge detection. Compared to the classical Riesz transform, 
the fractional Riesz transform is capable of detecting not 
only global information in images but also local information
in arbitrary directions.

In this article, combining the linear canonical transform and the Riesz
transform, we introduce the  linear canonical Riesz 
transform (for short, LCRT), which is further proved to 
be a linear canonical multiplier.   Using this LCRT multiplier, 
we conduct numerical simulations on images. Notably, 
the LCRT multiplier significantly reduces the complexity of 
the algorithm. Based on these we introduce the new concept of
the sharpness $R^{\rm E}_{\rm sc}$ of the edge strength and continuity of images 
associated with the LCRT  and, using it, we propose a new LCRT 
image edge detection method (for short, LCRT-IED method) 
and provide its mathematical foundation. Our experiments indicate 
that this sharpness $R^{\rm E}_{\rm sc}$ characterizes 
the macroscopic trend of edge  variations of the image under
 consideration, while this new LCRT-IED 
method not only controls the overall edge strength and continuity 
of the image, but also excels in feature extraction in some local 
regions.  These highlight the fundamental differences between the LCRT and 
the Riesz transform, which are precisely due to the 
multiparameter of the former. This new LCRT-IED method might be of significant 
importance  for image feature extraction, image matching, and 
image refinement.

The remainder of this article is organized as follows.
	
In Section \ref{sec2}, we introduce the LCRT and  prove
that the LCRT is the linear canonical multiplier (Theorem
\ref{LCRT-rf}). Additionally, we obtain the boundedness
of the LCRT on Lebesgue spaces (Theorem \ref{Lp}).

In Section \ref{sec3}, we conduct numerical simulations
of images using the LCRT, demonstrating the image of the
original test image (Gaussian function) after the LCRT
(Figure \ref{FIG5.2}) and, in the LCT domain,
the amplitude, the real part, and the imaginary part of both
the original test image and the image after the LCRT (Figures
\ref{FIG5.3} and \ref{FIG5.4}). The obtained results demonstrate
that the LCRT exhibits both the amplitude attenuate and the
shifting effect in the LCT domain. Additionally, we conduct image
numerical simulations of the LCHT to compare the differences
between the multiplier of the LCHT and the LCRT (Figure \ref{FIG9}).

In Section \ref{sec4},  we first introduce  the new concept of 
the sharpness  $R^{\rm E}_{\rm sc}$ of the edge strength
 and continuity of images associated 
with the LCRT in Definition \ref{is}. Based on this we then propose 
the new LCRT-IED method and provide its mathematical foundation (Theorem 
\ref{p-sharp}).  This sharpness  $R^{\rm E}_{\rm sc}$ is used to characterize the 
macroscopic trend of edge changes in the image under consideration. 
Our experimental results (Figures \ref{FIG10}, 
 \ref{FIG10-1}, \ref{FIG11},  \ref{FIG11-1}, \ref{FIG1111}, and
 \ref{FIG1111-1} and Tables \ref{tab:1}, \ref{tab:2}, and \ref{tab:3})
demonstrate that the LCRT-IED method can gradually control 
of the edge strength and  continuity by subtly adjusting the 
parameter matrices of the LCRT. 
Additionally, this new LCRT-IED method also demonstrates its capabilities 
in feature extraction within some local regions, enabling better 
preservation of image details. Also, in this section, we further 
summarize the advantages and features of this new LCRT-IED 
method in image processing.

A conclusion is given in Section \ref{sec5}.

It is worth mentioning that the experimental results
presented in this article do not require the final binarization
like the fractional Riesz transform edge detection method in 
\cite{fglwy23} and, moreover, the results are equally striking.
This can be attributed to the fact that the LCT is able
to perform not only rotations but also scaling operations
in the time-frequency plane. Furthermore, this new LCRT-IED method 
also performs exceptionally well when processing RGB images.

Finally, we make some conventions on notation.
Let  $\mathbb{N}:=\{1,2,\ldots\}$, ${C}^\infty({\mathbb{R}^n})$ denote the
set of all infinitely differentiable functions on $\mathbb{R}^n$.
We use $C$ to denote a positive constant,
which is independent of the main parameters involved,
whose value may vary from line to line. The \emph{symbol}
$g\lesssim h$ means $g\leq Ch$. If $g\lesssim h$ and $h\lesssim g$,
we then write $g\sim h$.
Let $p\in(0,\infty]$.
 The \emph{Lebesgue space} $L^p(\mathbb{R}^n)$
is defined to be the set of all measurable functions
$f$ on $\mathbb{R}^n$
such that, when $p\in(0,\infty)$,  
$\|f\|_{L^p(\mathbb{R}
^n)}:=\left[\int_{\mathbb{R}^n}|f(x)|^p\,
dx\right]^{\frac 1p}<\infty$
and
$$\|f\|_{L^{\infty}(\mathbb{R}^n)}:=\esssup_{
x\in \mathbb{R}^n}|f(x)|<\infty.$$
Moreover, when we prove a theorem (and the like), in its
proof we always use the same symbols as those used in
the statement itself of that theorem (and the like).

\section{Linear Canonical Riesz Transforms}\label{sec2}

The Riesz transform, as a natural extension of the
Hilbert transform in the $n$-dimensional
Euclidean space, is the archetype of Calder\'{o}n--Zygmund
operators,  which is not only an important class of
convolution operators but also a special class of Fourier multipliers.
The fractional Hilbert transform and the LCHT
have been extensively studied and widely applied in
signal processing (see, for example, \cite{fl08,hll1997,xwx09,zl18}). The  LCHT was introduced by Li et al.
in \cite{ltw06}  as follows.

\begin{definition}\label{LCTHilbertdef}
Let
$A :=\begin{bmatrix}a&b\\c&d\end{bmatrix}
\in {M_{2 \times 2}}(\mathbb{R})$ and $b\ne0 $.
The \emph{linear canonical Hilbert transform} (for short, LCHT)
${H_A}$ is defined by setting, for
any $f \in \mathscr{S}({\mathbb{R}})$ and $t\in {\mathbb{R}}$,
\begin{align*}
{H_A}(f)(t):= \frac{1}{\pi }e^{-i\frac{d}{2b}{t^2}}
{\rm p}.{\rm v}.\int_{\mathbb{R}}
{\frac{f(x)e^{i\frac{a}{2b}{x^2}}}{t-x}\,dx} .
\end{align*}	
\end{definition}

\begin{remark}	
If $A:= \begin{bmatrix}	
{\cos{\alpha}\,}&{\sin{\alpha}}\\
{-\sin{\alpha}\,}&{\cos {\alpha}}
\end{bmatrix}$ with $\alpha\in	(0,2\pi)$ and $\alpha\neq \pi$,
then the LCHT  reduces to the fractional Hilbert transform
in \cite[Definition 1]{z1998}.
\end{remark}

The two-dimensional (for short, 2D)   half-planed Hilbert transform was introduced by Xu et al. in \cite{xwx09} as follows.
\begin{definition}\label{2DdefHilbert1}
Let $A_j:=\begin{bmatrix}
{a_j}&{b_j}\\{c_j}&{d_j}
\end{bmatrix}\in{M_{2\times2}}(\mathbb{R})$ and
$b_j\neq0$ for any $j\in\{1,2\}$.
The \emph{half-planed Hilbert transforms} related to the  LCT (for short,
HLCHT) $H_x^{A_1}$ and $H_y^{A_2}$ are defined by setting,
for any $f \in \mathscr{S}
({\mathbb{R}^2})$ and $x,y\in \mathbb{R}$,
\begin{equation*}
H_x^{A_1}\left( f \right)\left( {x,y} \right) := \frac{e^{\frac
{-id_1x^2}{2b_1}}}{\pi}{\rm p}.{\rm v}.\int_{\mathbb{R}}\frac
{f( \xi,y)e^{\frac{-ia_1\xi ^2}{2b_1}}}{x-\xi }\,d\xi
\end{equation*}
and
\begin{equation*}
H_y^{A_2}\left( f \right)\left( {x,y} \right) := \frac{e^{\frac
{-id_2y^2}{2b_2}}}{\pi}{\rm p}.{\rm v}.\int_{\mathbb{R}}
\frac{f(x,\eta)e^{\frac{-ia_2\eta ^2}{2b_2}}}{y-\eta}\,d\eta,
\end{equation*}
where $H_x^{A_1}$ is the HLCHT along $X$ axis and
$H_y^{A_2}$ is the HLCHT  along $Y$ axis.
\end{definition}

Fu et al. \cite{fglwy23} further extended the concept of fractional
Hilbert transforms to higher dimensions, referred to as the fractional
Riesz transforms, and successfully applied this extension to image
edge detection. Compared with the classical Riesz transform
used in edge detection,  the fractional Riesz transform not only extracts
global edge information from images but also flexibly captures local
edge information in any directions by adjusting the fractional
parameters.

\begin{definition}\label{def-frt}(see \cite{fglwy23})
For any $j\in\{1,\ldots, n\}$  and  $\boldsymbol{\alpha}
:=(\alpha_1,\ldots,\alpha_n)\in [0,2\pi)^n$ with
$\alpha_k\notin\{0,\pi\}$  for any $k\in\{1,\ldots,n\}$,
the $j$th \emph{fractional Riesz transform}
 $R_j^{\boldsymbol{\alpha}}(f)$ of
$f\in \mathscr{S}(\mathbb{R}^n)$ is defined by setting,
for any $\boldsymbol{x}\in \mathbb{R}^n$,
\begin{align*}
R_j^{\boldsymbol{\alpha}}(f)(\boldsymbol{x}):=\widetilde{c}_n\
{\rm p}.{\rm v}.\ e_{-\boldsymbol{\alpha}}(\boldsymbol{x})
\int_{\mathbb{R}^n}\frac{x_j-y_j}{|\boldsymbol{x}
-\boldsymbol{y}|^{n+1}}f(\boldsymbol{y})e_{\boldsymbol
{\alpha}}(\boldsymbol{y})\,d\boldsymbol{y},
\end{align*}
where $\widetilde{c}_n:=\Gamma(\frac{n+1}{2})/\pi^{\frac{n+1}{2}}$, 
$e_{\boldsymbol{\alpha}}(\boldsymbol{x}):=e^{i\sum\nolimits
_{k = 1}^n\frac{\cot({\alpha _k})}{2}{x_k^2}}$, and
$e_{-\boldsymbol{\alpha}}(\boldsymbol{x}):=e^{-i\sum\nolimits
_{k = 1}^n\frac{\cot({\alpha _k})}{2}{x_k^2}}$.
\end{definition}

In other words, the fractional Riesz transform can also be viewed
as a fractional convolution operator.  Fu et al. \cite{fglwy23}
proved that the fractional Riesz transform is a fractional multiplier,
demonstrating that the fractional Riesz transform is entirely
equivalent to the composition of the FrFT, the fractional multiplier,
and the inverse FrFT.

Based on the LCHT and the fractional Riesz transform,
we introduce the $j$th linear canonical Riesz transform 
as follows.

\begin{definition}\label{Rieszdef1}
Let  $\boldsymbol{A}:=(A_1,\ldots,A_n)$
with $A_k:=\begin{bmatrix}
{a_k}&{b_k}\\
{c_k}&{d_k}
\end{bmatrix}\in{M_{2\times2}}(\mathbb{R})$
and  $b_k\ne0$ for any $k\in\{1,\ldots,n\}$,
and let
$\boldsymbol{a}:=(a_1,\ldots,a_n)$,
$\boldsymbol{b}:=(b_1,\ldots,b_n)$,
and $\boldsymbol{d}:=(d_1,\ldots,d_n)$.
For any given $j\in \{1,\ldots,n\}$, the $j$th \emph
{linear canonical Riesz transform} (for short, LCRT)
$R_j^{\boldsymbol{A}}(f)$ of $f \in \mathscr{S}
(\mathbb{R}^n)$ is defined  by setting, for any
 $\boldsymbol{x}\in \mathbb{R}^n$,
\begin{equation}\label{eqRiesz1}
R_j^{\boldsymbol A}(f)(\boldsymbol x):={\widetilde{c}_n}
{e}_{\boldsymbol b,-\boldsymbol d}(\boldsymbol x)\,{\rm p}.
{\rm v}.\int_{\mathbb{R}^n} \frac{({x_j}-{y_j})
f(\boldsymbol y) {e}_{\boldsymbol b,\boldsymbol a}
(\boldsymbol y)}{| \boldsymbol x-\boldsymbol y
|^{n+1}}\,d\boldsymbol{y},
\end{equation}
where $\widetilde{c}_n:= \frac{\Gamma (\frac{n + 1}{2})}
{\pi ^{\frac{n + 1}{2}}}$,  ${e}_{\boldsymbol b,
-\boldsymbol d}$  is the same as in \eqref{1eq}
with $\boldsymbol{a}$ and $\boldsymbol{b}$ replaced,
respectively, by $\boldsymbol b$ and $-\boldsymbol d$
and  ${e}_{\boldsymbol b,\boldsymbol a}$
is the same as in \eqref{1eq} with 
exchanging the position of  $\boldsymbol a$ and
$\boldsymbol{b}$.

\end{definition}

\begin{remark}\label{rem3.5}
\begin{enumerate}[(i)]
\item
Let $k,j\in\{1,\ldots,n\}$. If $\boldsymbol{A}:=({{A_1},\ldots,{A_n}})$
with $A_k:= \begin{bmatrix}	
 \cos{\alpha_k}&\sin{\alpha _k}\\
-\sin {\alpha_k}&\cos{\alpha _k}
\end{bmatrix}$ and both $\alpha_k\in(0,2\pi)$ and
$\alpha_k\neq\pi$,
then the LCRT $R_j^{\boldsymbol A}$ in this case
becomes the fractional Riesz transform
$R_j^{\boldsymbol{\alpha}}$ in \cite[Definition 1.3]{fglwy23}.
If $\boldsymbol{A}:=({{A_1},\ldots,{A_n}})$
with $A_k:=
\begin{bmatrix}
0&1\\{-1}&0
\end{bmatrix}$ for all $k$,  then the LCRT
$R_j^{\boldsymbol A}$ in this case becomes
 the classical Riesz transform $R_j$ in
 \cite[Definition 4.1.13]{g20141}.
\item
If $n=1$, then the LCRT becomes the LCHT
in this case in \cite[(7), Section 2]{ltw06}.	
\end{enumerate}
\end{remark}

Now, we prove
that the LCRT is a linear canonical multiplier, which
demonstrates that the LCRT is equivalent
to the composition of the LCT and a multiplier.

\begin{theorem}\label{LCRT-rf}
Let  $k,j\in\{1,\ldots,n\}$ and $\boldsymbol A:=(A_1,\ldots ,A_n)$
with $A_k:=\begin{bmatrix}
{a_k}&{b_k}\\
{c_k}&{d_k}
\end{bmatrix}\in{M_{2\times2}}(\mathbb{R})$
and both  $b_k\ne0$ and $a_k=d_k$ for all $k$, and let
$\boldsymbol{a}:=(a_1,\ldots,a_n)$,
$\boldsymbol{b}:=(b_1,\ldots,b_n)$,
and $\boldsymbol{d}:=(d_1,\ldots,d_n)$.
From the perspective of the {\rm LCT}, the $j$th {\rm LCRT}
$R_{j}^{\boldsymbol A}$ is given
 by the multiplication by the function $\frac
{-i\omega_j}{b_j}/{|\frac{\boldsymbol{\omega}}
{\boldsymbol b} |}$.
That is, for any $f \in\mathscr{S}
(\mathbb{R}^n)$  and $\boldsymbol{\omega}\in
\mathbb{R}^n$,  one has
\begin{equation}\label{eqRieszM}	
\mathscr{L}_{\boldsymbol A}\left[R_{j}^{\boldsymbol A}
(f)\right](\boldsymbol{\omega}) = -i \frac{\frac
{\omega_j}{b_j}}{|\frac{\boldsymbol{\omega}}
{\boldsymbol b}|} \mathscr{L}_{\boldsymbol A}
(f)(\boldsymbol{\omega}),
\end{equation}
where $\frac{\boldsymbol{\omega}}{\boldsymbol b}:=
(\frac{\omega_1}{b_1},\ldots,\frac{\omega_n}{b_n})$.
\end{theorem}

\begin{proof} From the definitions of LCTs and LCRTs,
it follows that, for any $f \in\mathscr{S}
(\mathbb{R}^n)$  and $\boldsymbol{\omega}\in
\mathbb{R}^n$,
\begin{align}\label{eq3.3}	
&\mathscr{L}_{\boldsymbol A}\left[R_j^{\boldsymbol{A}}
\left( f\right)\right](\boldsymbol{\omega})\\
&\quad={\widetilde{c}_n}\int_{\mathbb{R}^n}
{e}_{-\boldsymbol b,\boldsymbol d}
\left(\boldsymbol t \right)\lim_{\varepsilon\rightarrow0}
\int_{|\boldsymbol{t}-\boldsymbol{x}|\geq\varepsilon}
{\frac{( {t_j}-{x_j})f(\boldsymbol x)
{e}_{\boldsymbol b,\boldsymbol a}\left( \boldsymbol x
\right)}{{|{\boldsymbol{t}-\boldsymbol{x}}|}
^{n+1}}\,d\boldsymbol{x}}K_{\boldsymbol A}
(\boldsymbol{t},\boldsymbol{\omega})\,d\boldsymbol{t}\nonumber\\
&\quad=C_{\boldsymbol A}\int_{\mathbb{R}^n}
{f\left( \boldsymbol x\right)} {e}_{\boldsymbol b,
\boldsymbol a}\left( \boldsymbol{\omega}\right)
{e}_{\boldsymbol b,\boldsymbol a}\left(\boldsymbol x \right)
\left[ \frac{\Gamma \left( \frac{n+1}{2}\right)}
{\pi ^{\frac{n+1}{2}}}\lim_{\varepsilon\rightarrow0}
\int_{|\boldsymbol{t}-\boldsymbol{x}|\geq\varepsilon}
{\frac{(t_j-x_j){e_{-\boldsymbol b}}
\left( \boldsymbol{t} ,\boldsymbol{\omega} \right)}
{{|\boldsymbol{t}-\boldsymbol{x}|}^
{n+1}}\,d\boldsymbol{t}}
\right]\,d\boldsymbol{x},\nonumber
\end{align}
where  ${e}_{\boldsymbol b,
\boldsymbol a}$ is the same as in \eqref{1eq}
with exchanging the position of  $\boldsymbol a$ and
$\boldsymbol{b}$
and $e_{-\boldsymbol b}$
is the same as in \eqref{2eq}
with $\boldsymbol a$ replaced by $-\boldsymbol{b}$.
Let $$I :=\frac{\Gamma \left(\frac{n + 1}{2} \right)}
{\pi ^{\frac{n+1}{2}}}\lim_{\varepsilon\rightarrow0^{+}}
\int_{|\boldsymbol{t}-\boldsymbol{x}|\geq\varepsilon}
\frac{(t_j-x_j)e_{-\boldsymbol{b}}({\boldsymbol{t},
\boldsymbol{\omega}})}{{| \boldsymbol{t}-\boldsymbol{x}|}
^{n+1}}d\boldsymbol{t},$$
here and thereafter, $\varepsilon\rightarrow0^{+}$
means $\varepsilon\in (0,c_0)$, with $c_0$ being some positive 
constant, and $\varepsilon\rightarrow0$.
By a change of variables and the spherical coordinate transformation,
we conclude that
\begin{align*}	
I & = -\frac{\Gamma (\frac{n+1}{2})}{\pi^{\frac{n+1}
{2}}}\lim_{\varepsilon\rightarrow0^{+}}
\int_{|\boldsymbol{t}-\boldsymbol{x}|\geq\varepsilon}\frac{(x_j-t_j)e_{-
\boldsymbol b}(\boldsymbol{t},\boldsymbol{\omega})}
{| \boldsymbol{x}-\boldsymbol{t}|^{n + 1}}
 \,d\boldsymbol{t}\\
&=-\frac{\Gamma(\frac{n+1}{2})}
{\pi^{\frac{n +1}{2}}}\lim_{\varepsilon\rightarrow0^{+}}
\int_{|\boldsymbol{y}|\geq\varepsilon}
 \frac{y_j e^{-i\sum\limits_{j = 1}^n
\frac{(x_j-y_j)\omega _j}{b_j}}}
{{|\boldsymbol y |}^{n+1}}\,d\boldsymbol{y}\\
& = -\frac{\Gamma(\frac{n+1}{2})}
{\pi^{\frac{n+1}{2}}}e^{-i\sum\limits_{j = 1}^n
{\frac{{x_j}{\omega _j}}{b_j}}}
\lim_{\varepsilon\rightarrow0}\int_{s^{n-1}}\int_
{\varepsilon\leq r\leq\frac{1}{\varepsilon}}{ e^{i\sum\limits_{j = 1}^n
{\frac{{r{\theta _j}{\omega _j}}}{b_j}}}\frac{1}{r}}
{\theta _j}\,dr\,d\boldsymbol{\theta} \\
&=-\frac{\Gamma ({\frac{n+1}{2}})}
{\pi^{\frac{n + 1}{2}}}e^{-i\sum\limits_{j = 1}
^n{\frac{x_j{\omega _j}}{b_j}}}\int_{S^{n-1}}
{\int_0^\infty  i\frac{\sin ( \sum\limits_{j = 1}^n
{\frac{r{\theta _j}{\omega _j}}{b_j}})}{r}}
{\theta _j}\,dr\,d\boldsymbol{\theta} \\
& = -\frac{\Gamma( \frac{n + 1}{2})}
{\pi ^{\frac{n + 1}{2}}}e^{-i\sum\limits_{j = 1}^n
{\frac{x_j{\omega _j}}{b_j}} }\int_{S^{n-1}}
{\int_0^\infty  i \frac{^{\sin( r\sum\limits_{j = 1}
^n {\frac{\theta _j \omega _j}{b_j}})}}
{r\sum\limits_{j = 1}^n \frac{\theta _j \omega_j}
{b_j}}}{\theta _j}d\left( \sum\limits_{j = 1}^n
{\frac{\theta _j \omega_j}{b_j}}r\right)
\,d\boldsymbol{\theta} \\
&=-i\frac{\Gamma ({\frac{n+1}
{2}})}{\pi ^{\frac{n+1}{2}}}e^{-i\sum\limits_{j = 1}^n \frac{x_j{\omega _j}}
{b_j}}\frac{1}{2}\pi\frac{2\pi ^{\frac{n-1}{2}}}
{\Gamma \left( \frac{n+1}{2} \right) }\frac{\frac
{\omega_j}{b_j}}{\left| {\frac{\boldsymbol{\omega }}
{\boldsymbol{b}}} \right|}=-ie^{-i\sum\limits_{j = 1}^n \frac{x_j{\omega _j}}
{b_j}}\frac{\frac
{\omega_j}{b_j}}{| {\frac{\boldsymbol{\omega }}
{\boldsymbol{b}}}|}.
\end{align*}
Substituting the above formula into \eqref{eq3.3}, we then obtain
\begin{align*}	
\mathscr{L}_{\boldsymbol A}\left[R_j^{\boldsymbol A}
\left(f\right)  \right] (\boldsymbol{\omega} )
& =-iC_{\boldsymbol A}\int_{\mathbb{R}^n}f
(\boldsymbol x) e^{i\sum\limits_{j = 1}^n ({\frac
{a_j}{2b_j}\omega _j^2-\frac{x_j{\omega _j}}{b_j}
+ \frac{a_j}{{2b_j}}x_j^2}) }\frac{\frac{\omega_j}
{b_j}}{\left|{\boldsymbol{\frac{\omega }{b}}} \right|}
\,d\boldsymbol{x}= -i\frac{{\frac{\omega _j}{b_j}}}{|
{\frac{\boldsymbol{\omega}}{\boldsymbol{b}}}|}
{\mathscr{L}_{\boldsymbol A}}(f)(\boldsymbol{\omega} ).
\end{align*}
This  finishes the proof of Theorem \ref{LCRT-rf}.
\end{proof}

\begin{remark}
\begin{enumerate}[(i)]
\item	
When $n=1$,  Theorem \ref{LCRT-rf} in this case  coincides with
 \cite[Theorem 2]{ltw06} which is on the LCHT.
\item
Let $j\in\{1,\ldots,n\}$. When $A_j:= \begin{bmatrix}	
 {\cos{\alpha_j}}&{\sin{\alpha _j}}\\
{-\sin {\alpha_j}}&{\cos {\alpha _j}}
\end{bmatrix}$ with $\alpha_j\in (0,2\pi)$
and $\alpha_j\neq \pi $,
Theorem \ref{LCRT-rf} in this case reduces to
\cite[Theorem 2.1]{fglwy23} which is on the fractional Riesz transform.
\end{enumerate}
\end{remark}

 The following lemma gives
 the inverse LCT theorem, which is precisely
 \cite[Reversible property (13)]{dp1999}.

\begin{lemma}\label{lem-ilct}
Let $\boldsymbol{A}:=(A_1,\ldots ,A_n)$
with $A_k:=\begin{bmatrix}
{a_k}&{b_k}\\
{c_k}&{d_k}
\end{bmatrix}\in{M_{2 \times 2}}
(\mathbb{R})$ for any $k\in\{1,\ldots,n\}$.
For any $f\in\mathscr{S}(\mathbb{R}^n)$ and
$\boldsymbol{x}\in \mathbb{R}^n$,
$({\mathscr{L}_{\boldsymbol{A}}})^{-1}
{\mathscr{L}_{\boldsymbol{A}}}(f)
(\boldsymbol x)=f
(\boldsymbol x).$
\end{lemma}

Let $k,j\in\{1,\ldots,n\}$ and $\boldsymbol A:=(A_1,\ldots ,A_n)$
with $A_k:=\begin{bmatrix}
{a_k}&{b_k}\\
{c_k}&{d_k}
\end{bmatrix}\in{M_{2\times2}}(\mathbb{R})$
and both  $b_k\ne0$ and $a_k=d_k$ for all $k$.
By Lemma \ref{lem-ilct} and Theorem \ref{LCRT-rf},
the $j$th LCRT $R_j^{\boldsymbol A}$ can be rewritten as,
for any $f\in\mathscr{S}(\mathbb{R}^n)$
and  $\boldsymbol{x}\in \mathbb{R}^n$,
\begin{equation}\label{eqRiesz2}
R_j^{\boldsymbol A}f(\boldsymbol x) = \mathscr{L}
_{{\boldsymbol A}^{-1}}\left\{-i\frac{\frac{\omega _j}
{b_j}}{| {\frac{\boldsymbol{\omega}}
{\boldsymbol b}}|}{\mathscr{L}_{\boldsymbol A}}
\left[f(\boldsymbol{\omega})\right]\right\}(\boldsymbol x).
\end{equation}
For simplicity of presentation, for any  $\boldsymbol{\omega}
\in \mathbb{R}^n$,  let  $m^j_{\boldsymbol A}(\boldsymbol{\omega}):=
-i\frac{\frac{\omega _j}{b_j}}{| \frac{\boldsymbol
{\omega} }{\boldsymbol b}|}$. Based on \eqref{eqRiesz2},
we observe that the  LCRT  $R_j^{\boldsymbol A}f$ of  $f$
consists of the following three simpler operators
(as illustrated in Figure \ref{img2}):

\begin{enumerate}[(i)]
\item The LCT ${\mathscr{L}_{\boldsymbol A}}f$ of $f$,
that is, for any $\boldsymbol w\in \mathbb{R}^n$, $g(\boldsymbol w):=
{\mathscr{L}_{\boldsymbol A}}f(\boldsymbol w)$;
\item The multiplier $m^j_{\boldsymbol A}$,
 that is,  for any $\boldsymbol w\in \mathbb{R}^n$,
 $h(\boldsymbol w) := m^j_{\boldsymbol A}
 (\boldsymbol w)g(\boldsymbol w)$;
\item  The inverse LCT $\mathscr{L}_{{\boldsymbol A}^{-1}}h$
of  $h$, that is, for any $\boldsymbol x\in \mathbb{R}^n$,
 $R_j^{\boldsymbol A}f(\boldsymbol x) := \mathscr{L}
 _{{\boldsymbol A}^{-1}}h(\boldsymbol x)$.
\end{enumerate}

\begin{figure}[H]
\centering
\includegraphics[width=0.8\linewidth]{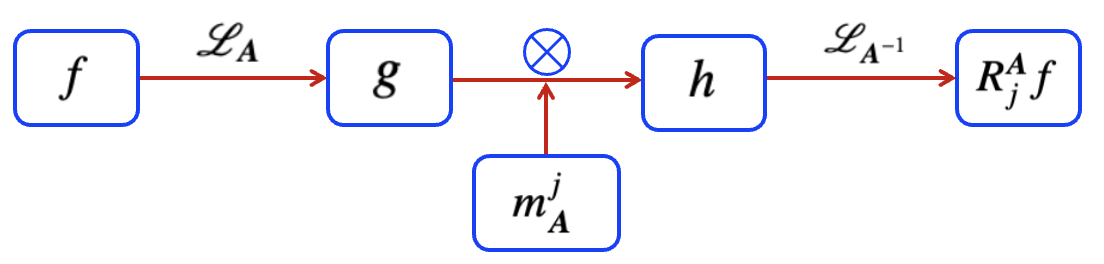}
\caption{The decomposition of the $j$th LCRT $R_j^{\boldsymbol A}f$.}
\label{img2}
\end{figure}

Next, we establish the boundedness of the LCRT on Lebesgue spaces.

\begin{theorem}\label{Lp}
Let $k,j\in\{1,\ldots,n\}$ and $\boldsymbol A:=(A_1,\ldots ,A_n)$
with $A_k:=\begin{bmatrix}
{a_k}&{b_k}\\
{c_k}&{d_k}
\end{bmatrix}\in{M_{2\times2}}(\mathbb{R})$
and both $b_k\ne0$ and $a_k=d_k$ for all $k$.
For any give   $p\in(1,\infty)$, $R_j^{\boldsymbol A}$
can be extended to a linear bounded operator on $L^p(\mathbb{R}^n)$
and, moreover, there exists a positive constant $C$ such that, for any
 $f \in L^p(\mathbb{R}^n)$,
 $\|R_j^{\boldsymbol A}(f)\|_{L^p(\mathbb{R}^n)}\leq C
 \|f\|_{L^p(\mathbb{R}^n)}.$	
\end{theorem}
\begin{proof}
We first assume $f \in \mathscr{S}(\mathbb{R}^n)$.
By the integral definitions [see Definition \ref{Rieszdef1} and Remark \ref{rem3.5}(i)]
of both the Riesz transform and the LCRT, 
 we obtain, for any $\boldsymbol{x} \in{\mathbb{R}^n}$,
$R_j^{\boldsymbol A}(f)
(\boldsymbol x) = {e}_{\boldsymbol{b},-\boldsymbol{d}}
(\boldsymbol x)R_{j}({e}_{\boldsymbol{b},\boldsymbol{a}}f)(\boldsymbol x),$
where  ${e}_{\boldsymbol{b},-\boldsymbol{d}}$ 
is the same as in \eqref{1eq}
with $\boldsymbol{a}$ and $\boldsymbol{b}$ replaced,
respectively, by ${\boldsymbol b}$ and ${-\boldsymbol d}$,
${e}_{\boldsymbol{b},
\boldsymbol{a}}$ is the same as in \eqref{1eq}
with exchanging the position of  $\boldsymbol a$ and
$\boldsymbol{b}$, and $R_{j}$ is the same as in Remark
 \ref{rem3.5}(i).
Using this and the boundedness of the classical
 Riesz transform on $L^p(\mathbb{R}^n)$ (see, for example, \cite[Corollary 5.2.8]{g20141}),
we immediately obtain the boundedness of the LCRT 
$R_j^{\boldsymbol A}$ on $L^p(\mathbb{R}^n)\cap\mathscr{S}( \mathbb{R}^n)$.
 
Now, let $p\in (1,\infty)$ and $f\in L^p(\mathbb{R}^n)$. Since $\mathscr{S}( \mathbb{R}^n)$ 
is dense in $L^p(\mathbb{R}^n)$  (see, for example, \cite[p. 20]{sw1971}), 
we deduce that there exists a sequence 
$\{g_l\}_{l\in\mathbb{N}}$ in $\mathscr{S}(\mathbb{R}^n)$ such that $g_l \to f$ in 
$L^p(\mathbb{R}^n)$ as $l\to\infty$.  Thus, $\{g_l\}_{l\in\mathbb{N}}$ is a Cauchy
sequence in $L^p(\mathbb{R}^n)$. Note that $\{{e}_{\boldsymbol{b},\boldsymbol{a}}g_l\}_{l\in\mathbb{N}}$ is still
a sequence  in $\mathscr{S}(\mathbb{R}^n)$ and, by the just proved boundedness of 
$R_j^{\boldsymbol A}$ on $L^p(\mathbb{R}^n)\cap\mathscr{S}(\mathbb{R}^n)$,
we find that, for any $l,m\in\mathbb{N}$,
$$\left\|R_j^{\boldsymbol A}(g_l) - R_j^{\boldsymbol A}(g_m)\right\|_{L^p
(\mathbb{R}^n)}\lesssim \left\|g_l - g_m\right\|_{L^p(\mathbb{R}^n)},$$
which further implies that $\{R_j^{\boldsymbol A}(g_l)\}_{l\in\mathbb{N}}$
is also a Cauchy sequence in $L^p(\mathbb{R}^n)$. From this and the completeness of $L^p(\mathbb{R}^n)$, we infer that $\lim_{l\to\infty}R_j^{\boldsymbol A}(g_l)$
exists in $L^p(\mathbb{R}^n)$. Let $R_j^{\boldsymbol A}(f):=\lim_{l\to\infty}R_j^{\boldsymbol A}(g_l)$ in $L^p(\mathbb{R}^n)$.
Then it is easy to show that $R_j^{\boldsymbol A}(f)$ is independent of the choice of $\{g_l\}_{l\in\mathbb{N}}$ and hence is well-defined.
Moreover, by the Minkowski norm inequality of $L^p(\mathbb{R}^n)$ 
and the just proved boundedness of 
$R_j^{\boldsymbol A}$ on $L^p(\mathbb{R}^n)\cap\mathscr{S}(\mathbb{R}^n)$ 
again, we obtain 
$$\left\|R_j^{\boldsymbol A}(f)\right\|_{L^p(\mathbb{R}^n)}=
\lim_{l \to \infty} \left\|R_j^{\boldsymbol A}(g_l)\right\|_{L^p(\mathbb{R}^n)}
\lesssim \lim_{l \to \infty}  \|g_l\|_{L^p(\mathbb{R}^n)} =  \|f\|_{L^p(\mathbb{R}^n)},$$
which then completes the proof of Theorem \ref{Lp}.
\end{proof}

\section{Numerical Simulations\label{sec3}}

In this section, we perform LCRTs numerical simulations
of images by using the discrete LCT algorithm and the LCRT
multiplier (see, for example, the inspiring comprehensive monograph \cite{ppst23} for similar numerical simulations). Due to the fast algorithm of LCTs, the complexity of this algorithm is significantly reduced compared with the convolution type LCRT. Through a series of simulation experiments, we successfully demonstrate that LCRTs not only have the effect
of shifting terms, but can also attenuate the amplitude in the LCT
domain. Furthermore, we conduct numerical simulations of
images  by using LCHTs and conduct a complete comparison
of the numerical results between LCRTs and LCHTs, clarifying
the distinctions between them. This comparison not only
deepens our understanding of these two transforms, but also
lays the foundation for their further applications in signal
processing, thereby opening doors to new application directions.

\begin{figure}[H]
\centering
\subfigcapskip=-10pt
\subfigure[]{\includegraphics[width=0.48\linewidth]
{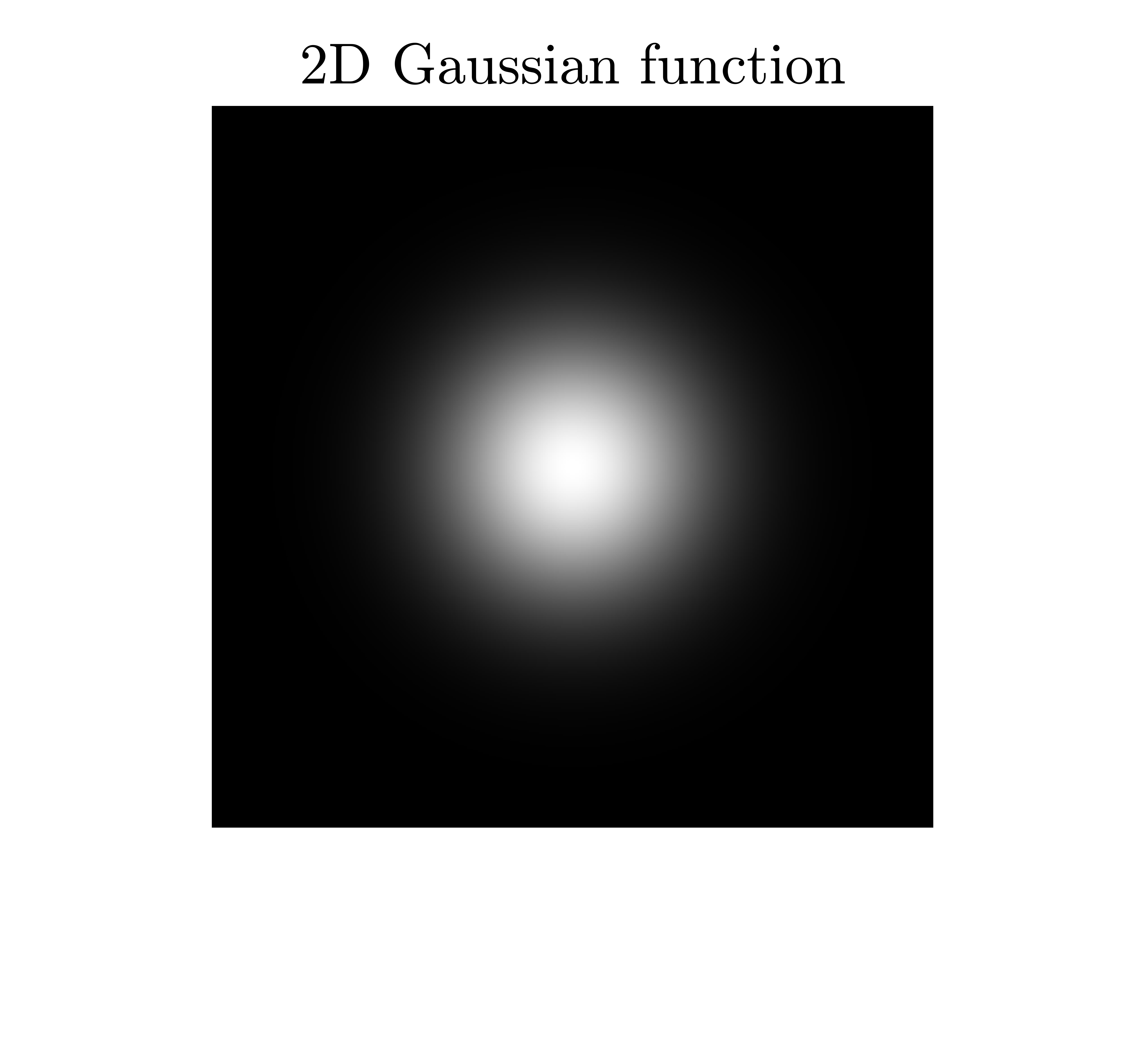}}\hspace{0.2cm}
\subfigure[]{\includegraphics[width=0.48\linewidth]
{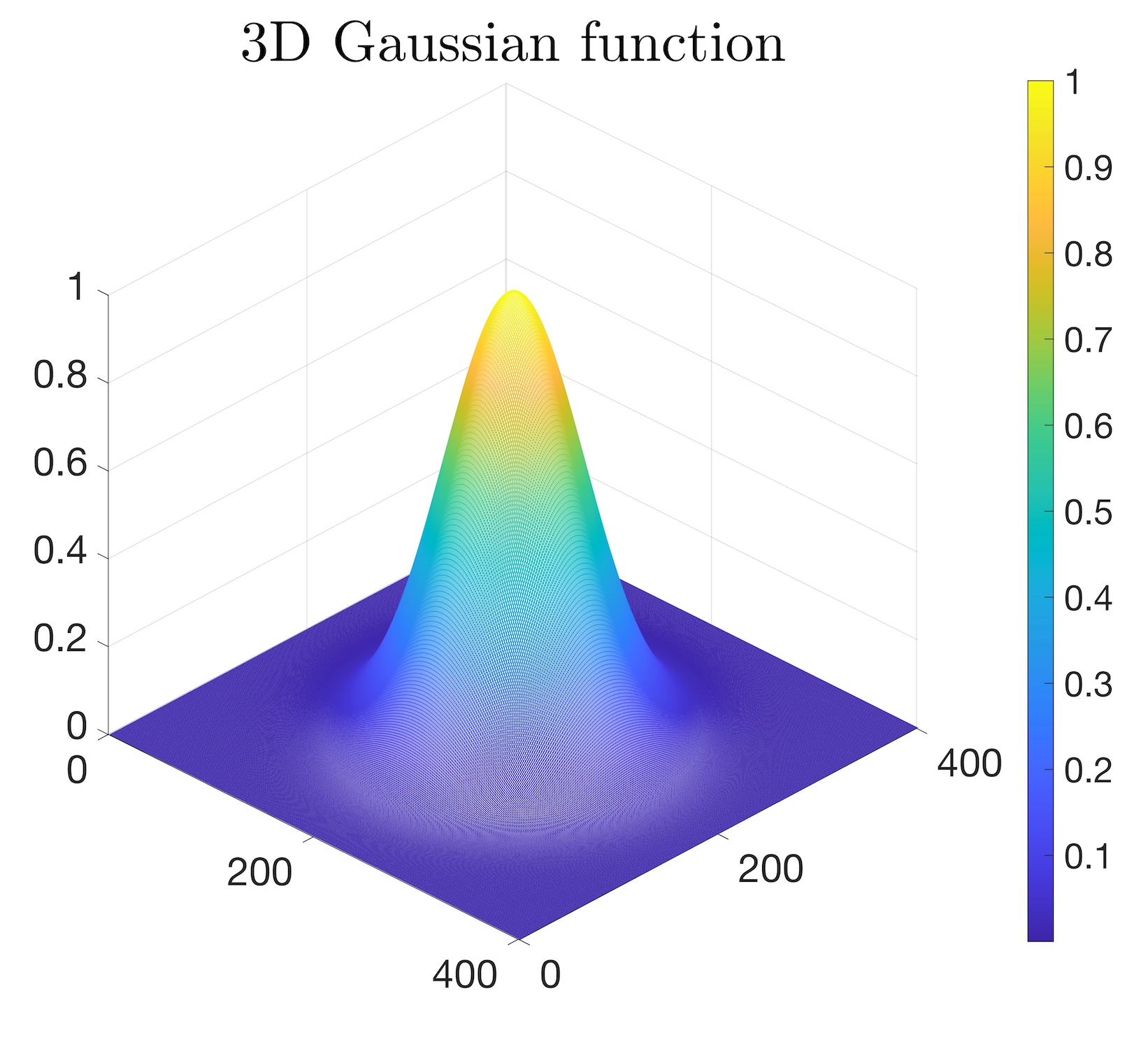}}
\caption{Gaussian function as a test image.}
\label{FIG5.1}
\end{figure}

Graphs (a) and (b) in Figure \ref{FIG5.1} are the test
images used for numerical
simulations in this whole section. Graph (a) presents a $400\times
400$ pixel 2D grayscale image, while Graph (b) is a
three-dimensional (for short, 3D) color image of Graph (a).
In the continuous case, Graph (a)  can be represented by the
following \emph{Gaussian function} $G$, which is
defined by setting, for any $(x_1, x_2)\in [0,400]^2$,
$$ G(x_1, x_2):=e^{-\frac{(x_1
- 200)^2+(x_2-200)^2}{2\cdot\sigma^2}}, $$ where $x_1$
and $x_2$ are the coordinates in the image matrix and
$\sigma$ is the standard deviation,  with $\sigma=50$. The center
of Graph (a)  is located at $(200, 200)$. Finally, the pixel
values of Graph (a) are normalized to ensure that its image data
falls within the range of $0$ to $1$, which is a common
practice to maintain consistency during display or
further processing.

\begin{figure}[H]
\centering
\subfigcapskip=-8pt
\subfigure[]{\includegraphics[width=0.461\linewidth]
{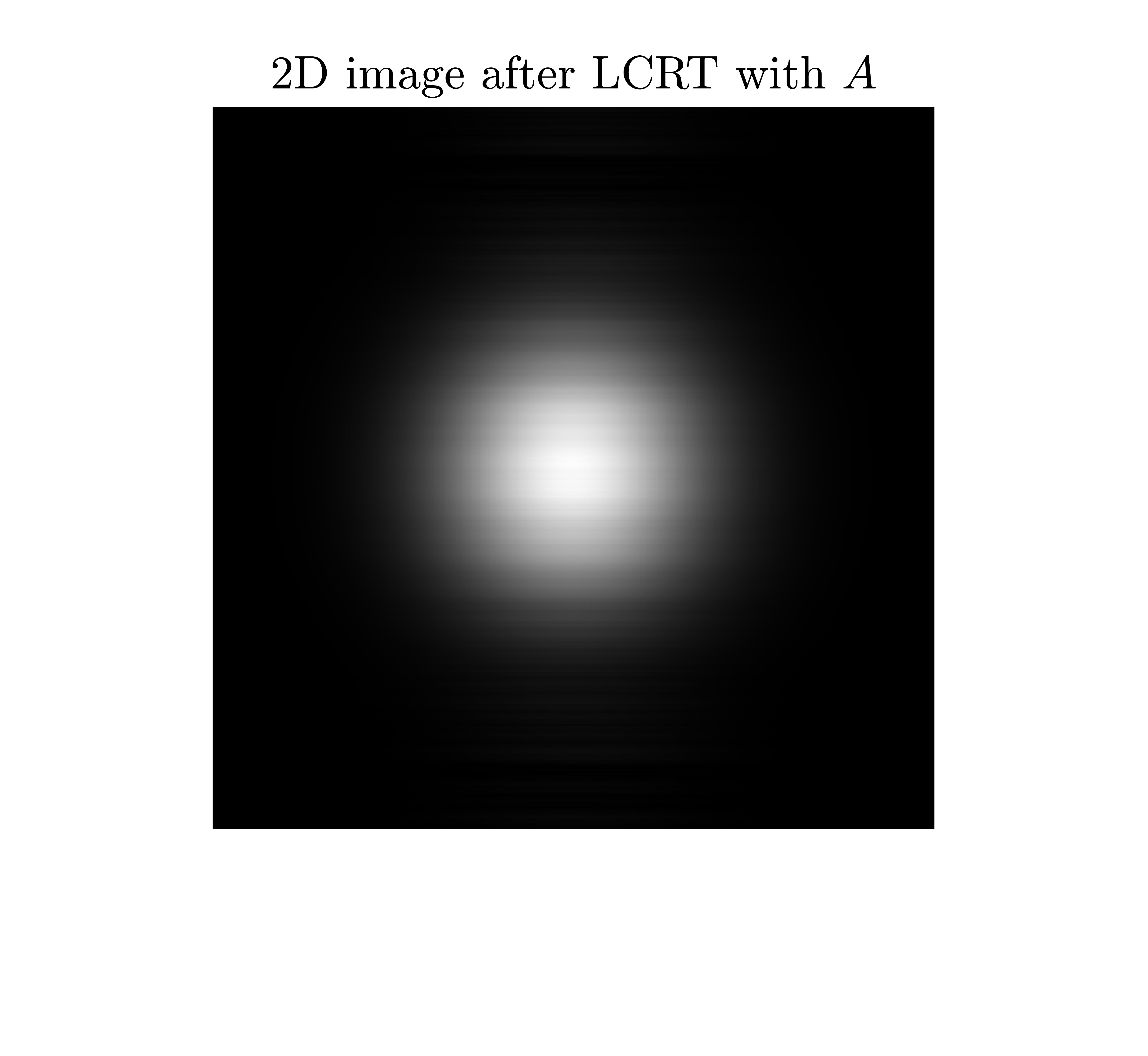}}\hspace{-0.5cm}
\subfigure[]{\includegraphics[width=0.461\linewidth]
{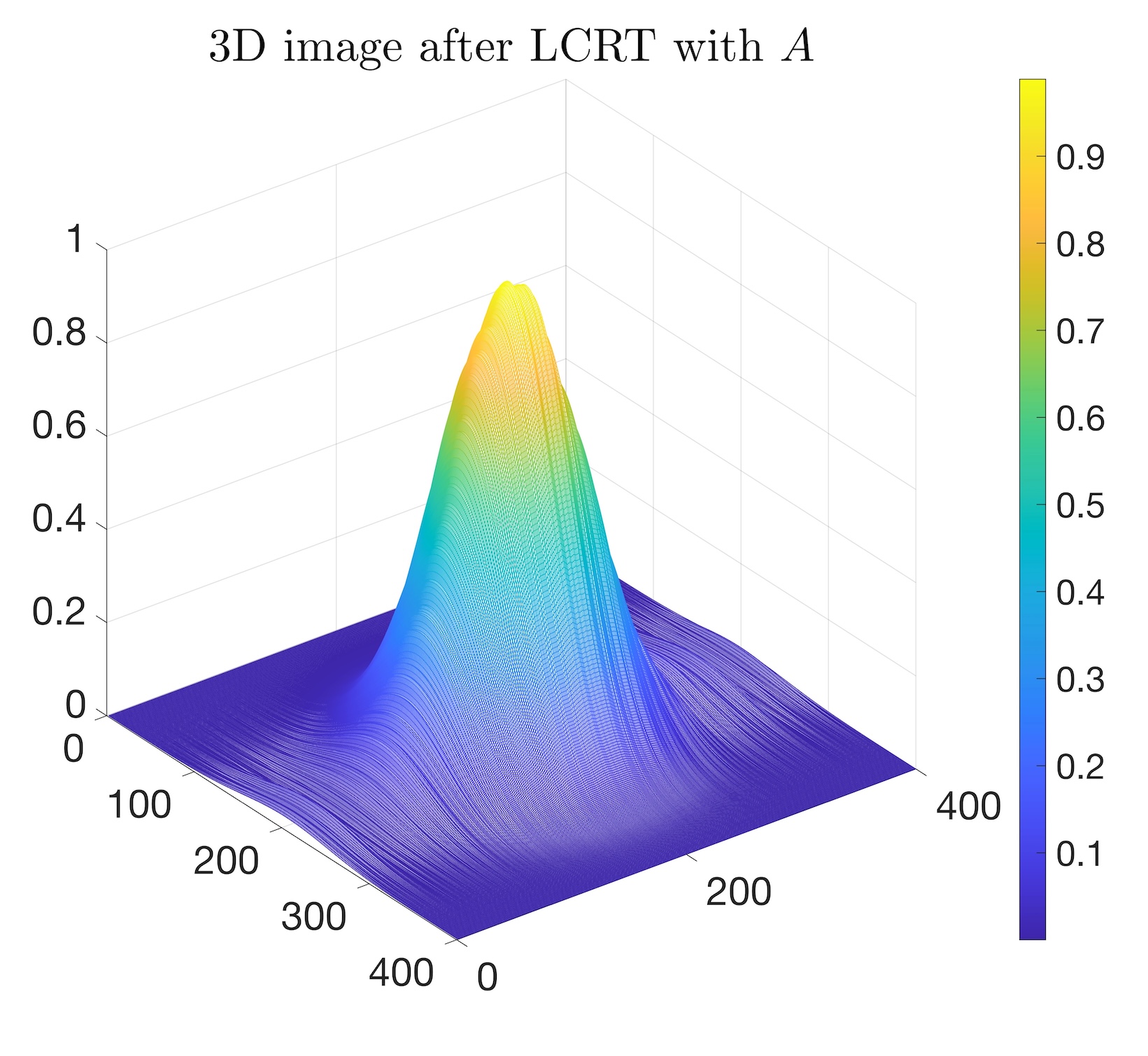}}\\\vspace{0.5cm}
\subfigure[]{\includegraphics[width=0.461\linewidth]
{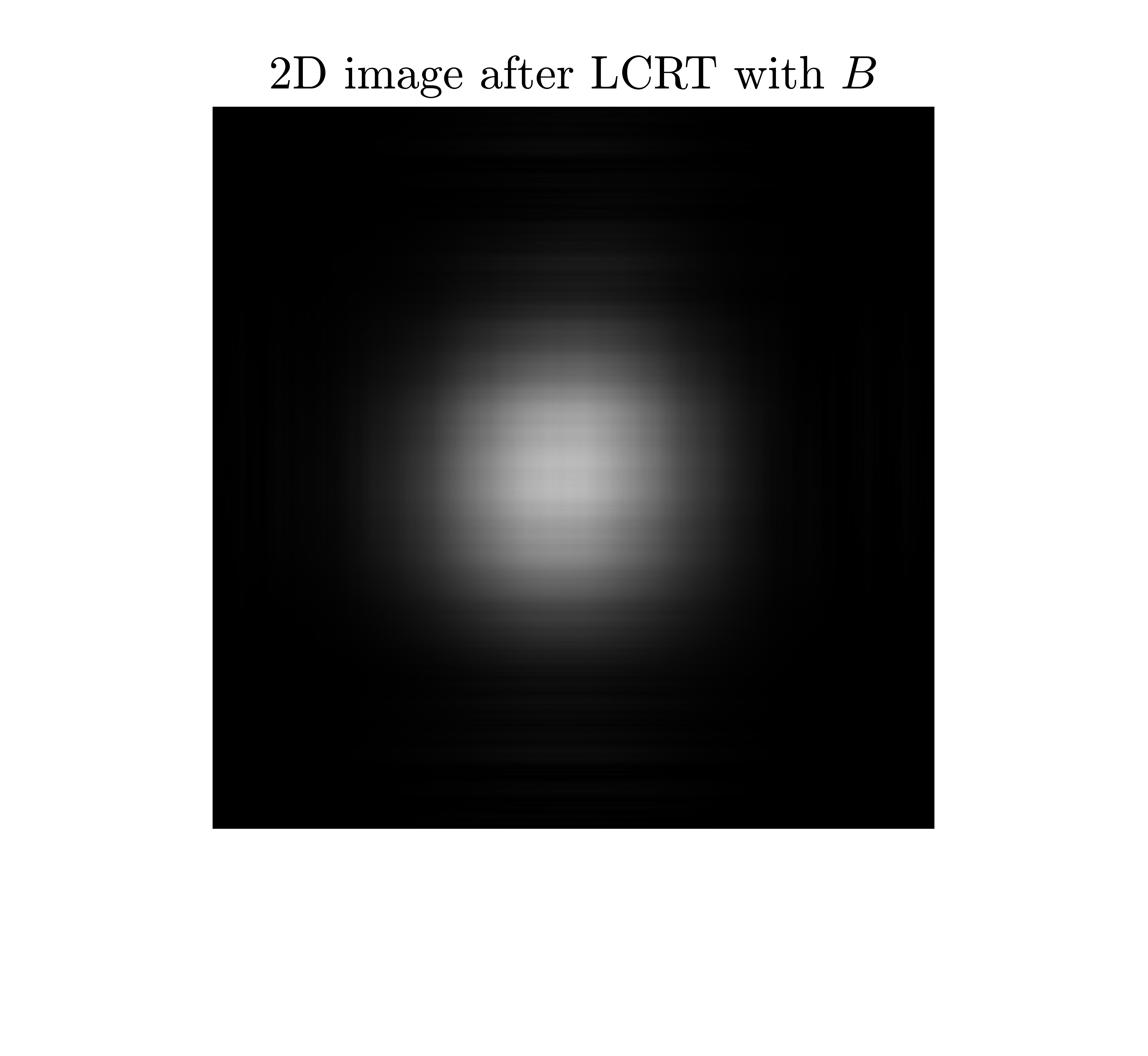}}\hspace{-0.5cm}
\subfigure[]{\includegraphics[width=0.461\linewidth]
{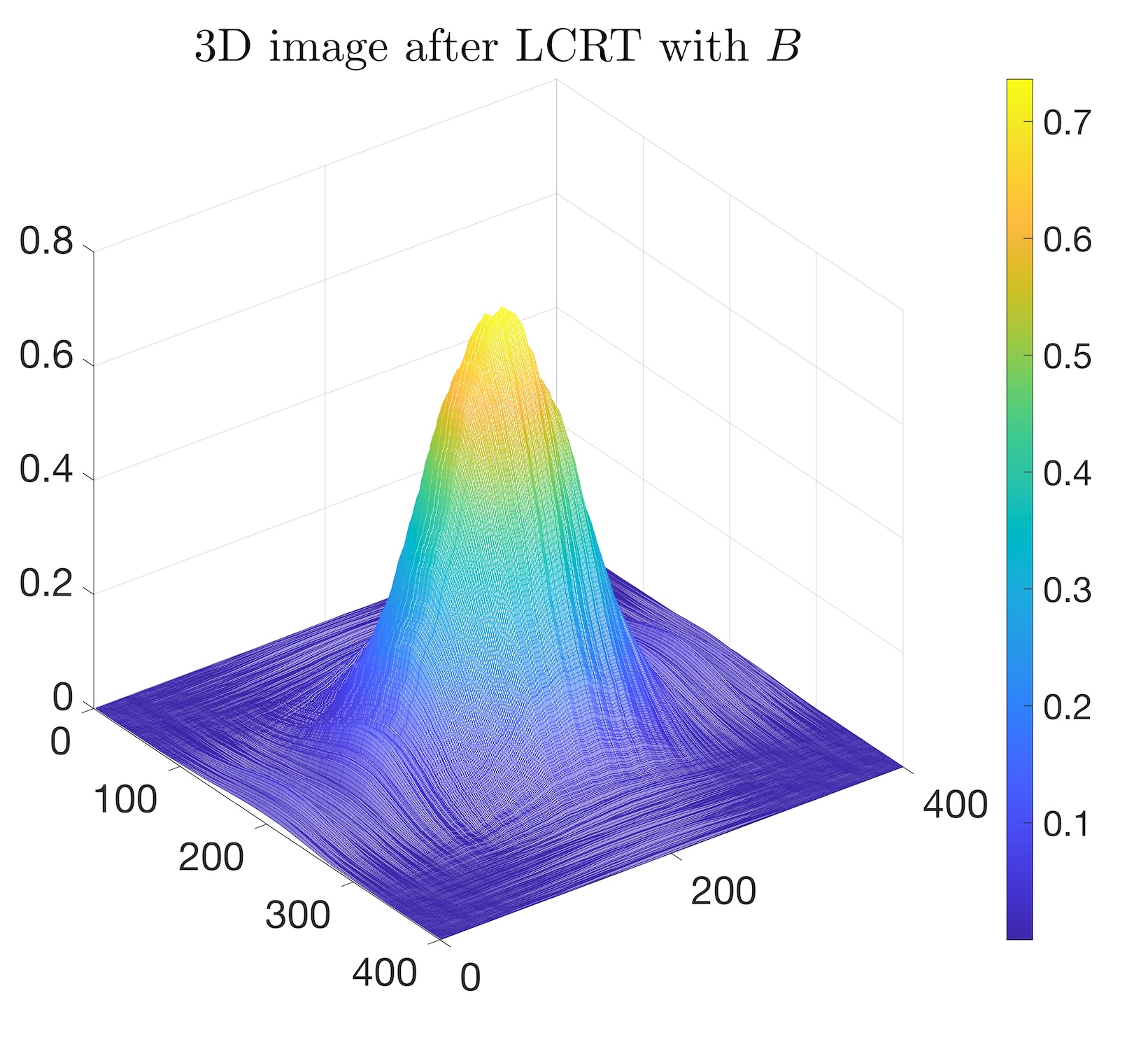}}\\\vspace{0.5cm}
\subfigure[]{\includegraphics[width=0.461\linewidth]
{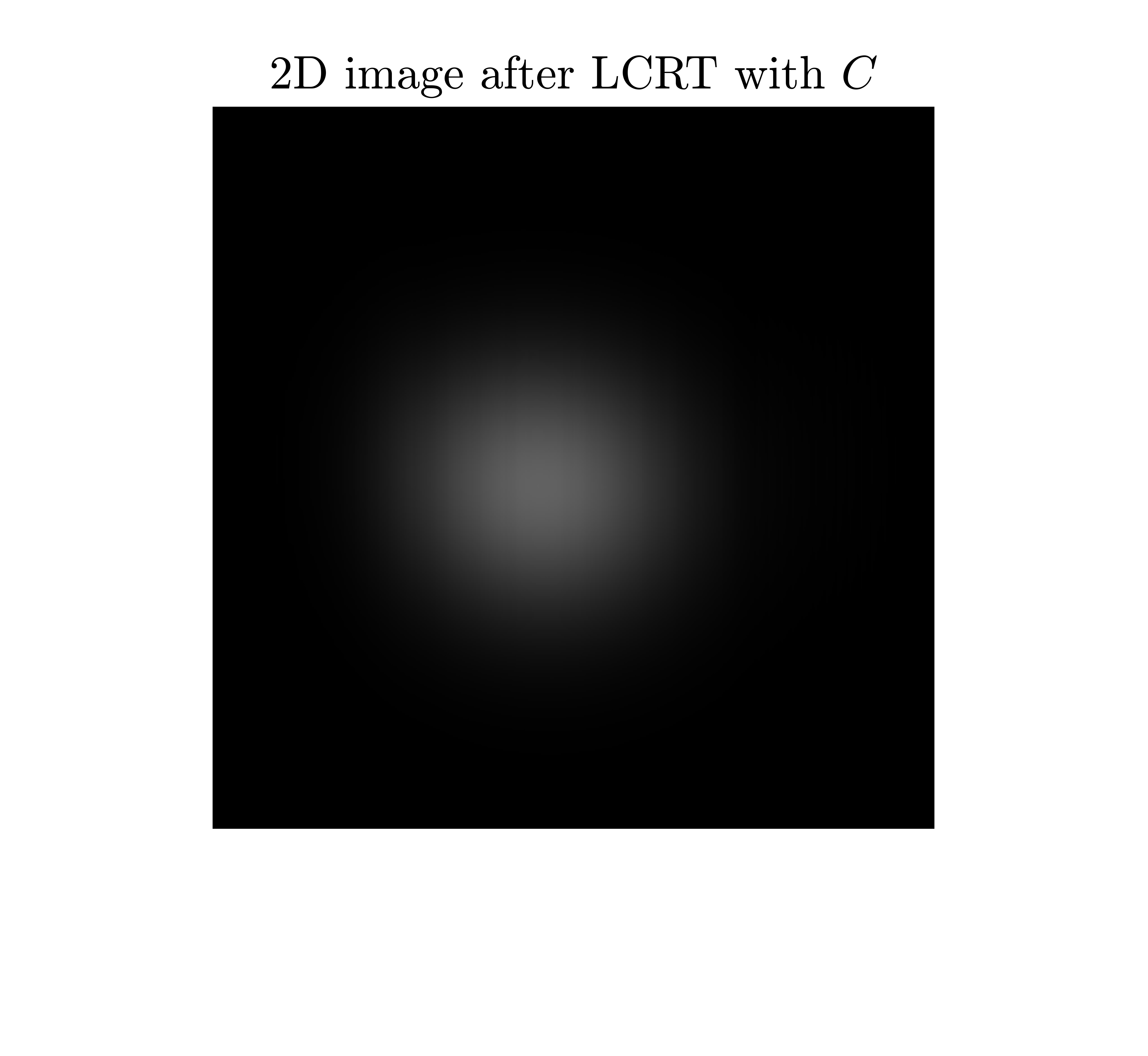}}\hspace{-0.5cm}
\subfigure[]{\includegraphics[width=0.461\linewidth]
{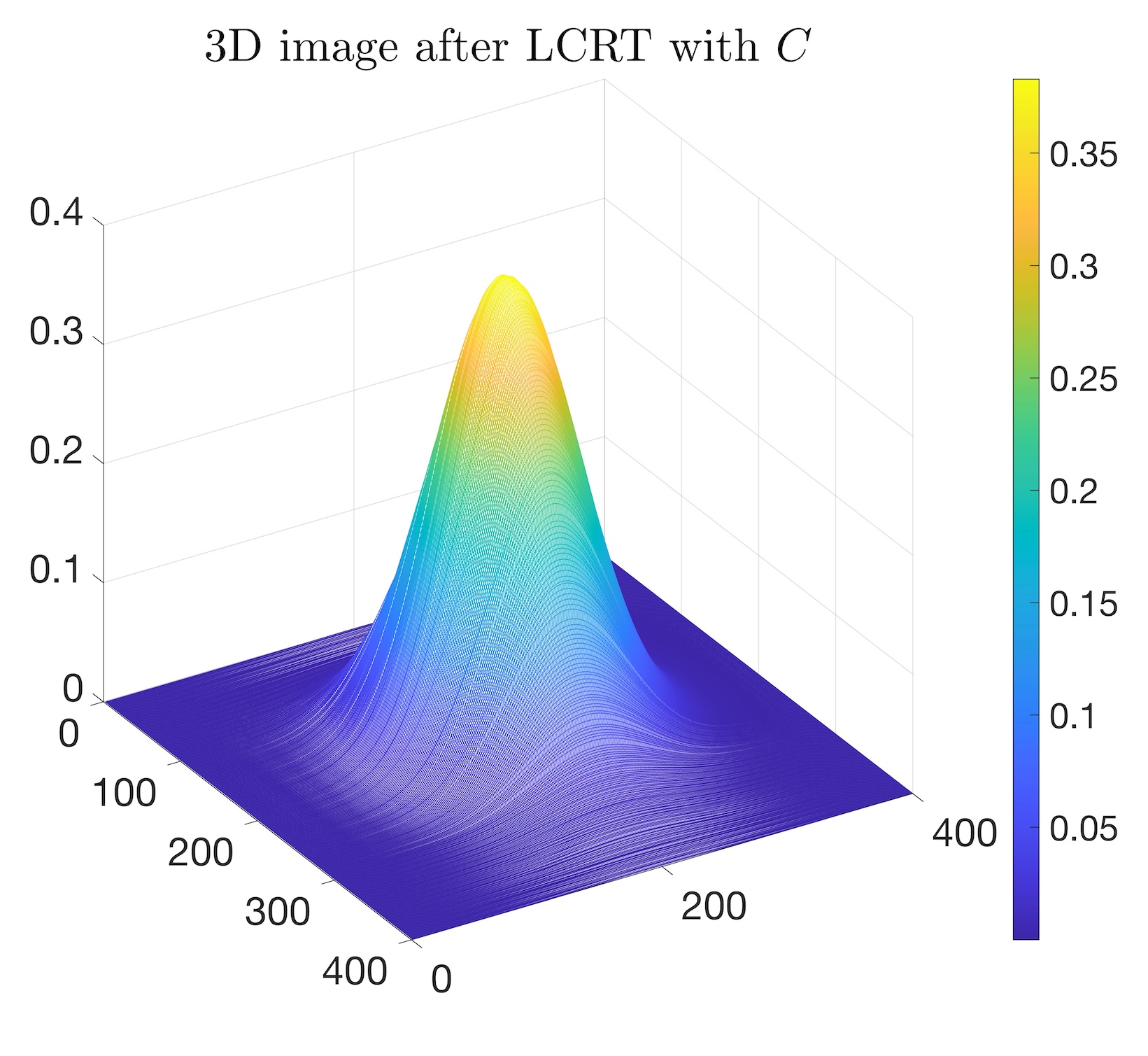}}
\caption{2D and 3D images of the test image after
applying LCRT with different parameters are shown.
Graphs (a) and (b), (c) and (d), and (d) and (e)
correspond to LCRT parameters ${\boldsymbol A}$,
${\boldsymbol B}$, and ${\boldsymbol C}$,
respectively.}
\label{FIG5.2}
\end{figure}

\begin{figure}[H]
\centering
\subfigcapskip=-8pt
\subfigure[]{\includegraphics[width=0.4506\linewidth]{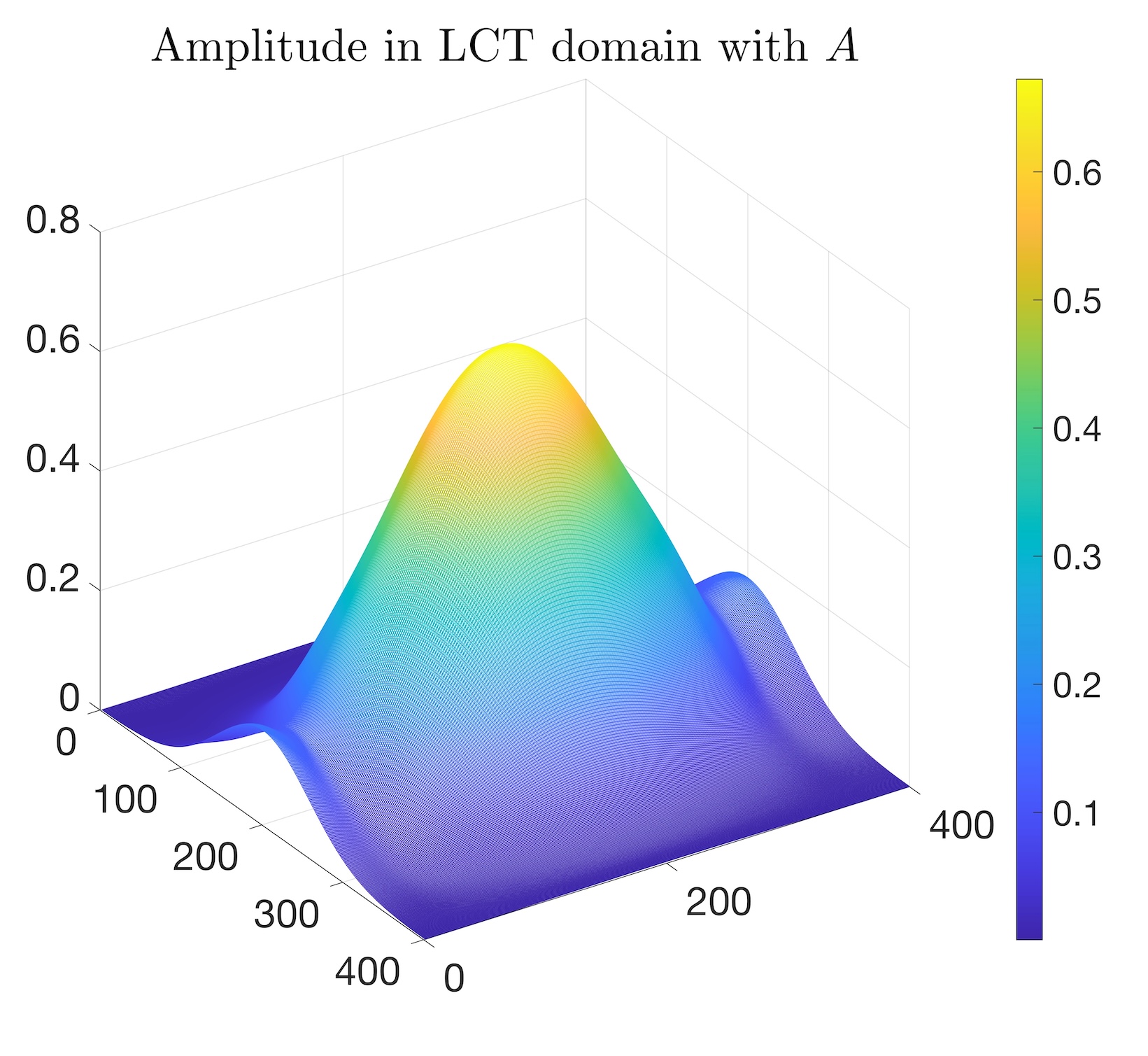}}\ \
\subfigure[]{\includegraphics[width=0.4506\linewidth]{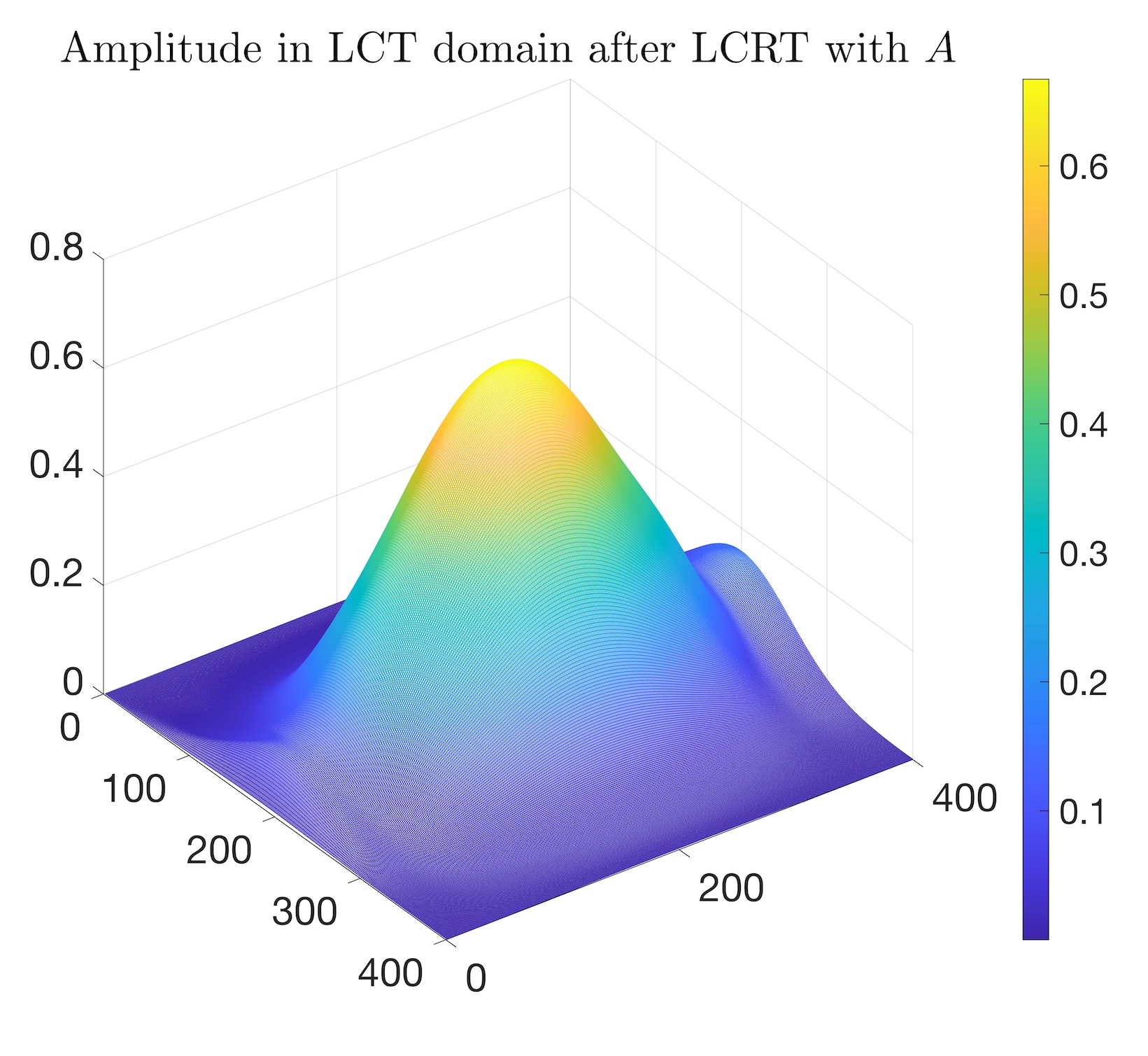}}\\
\vspace{0.5cm}
\subfigure[]{\includegraphics[width=0.4505\linewidth]{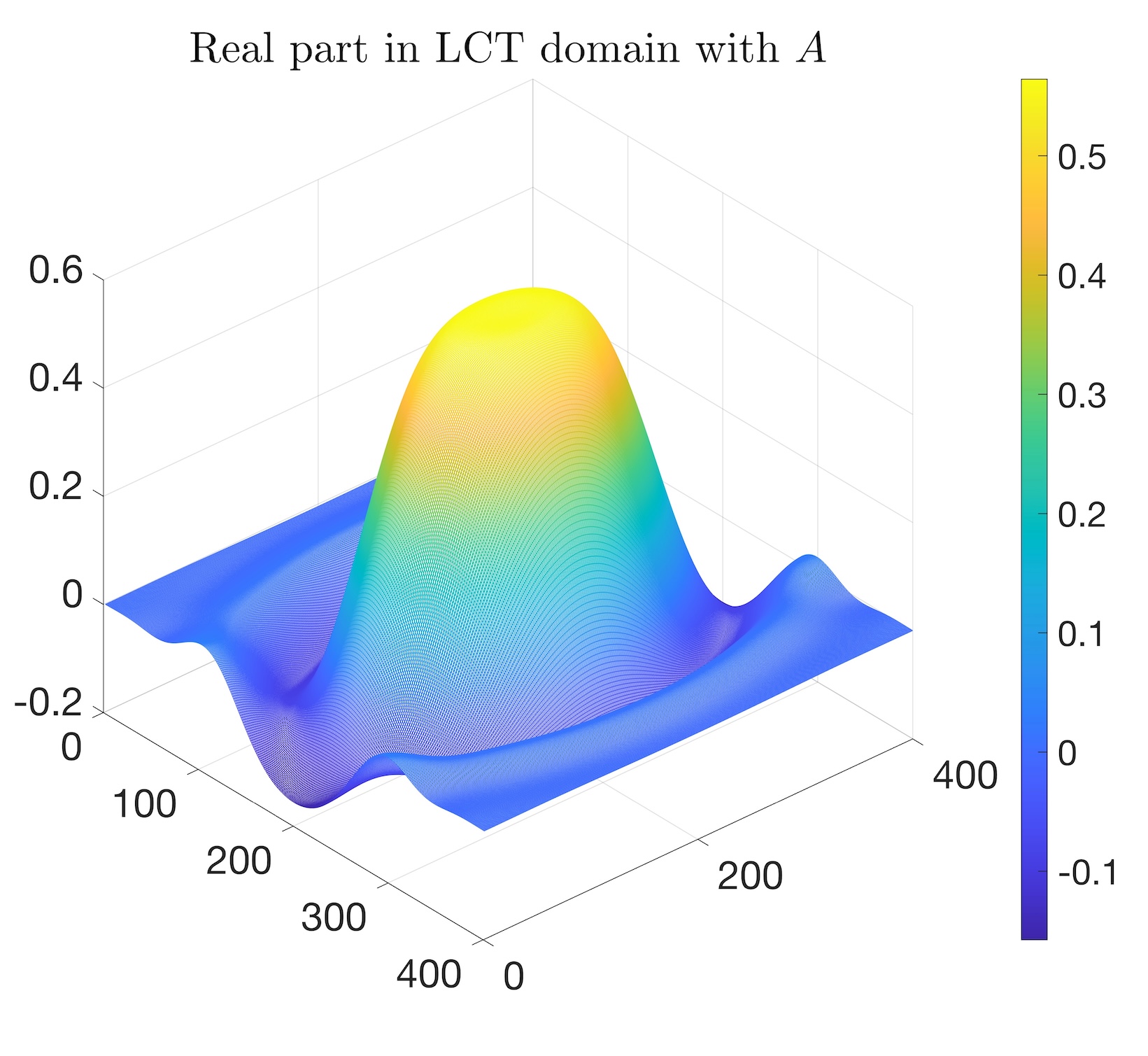}}\ \
\subfigure[]{\includegraphics[width=0.4505\linewidth]{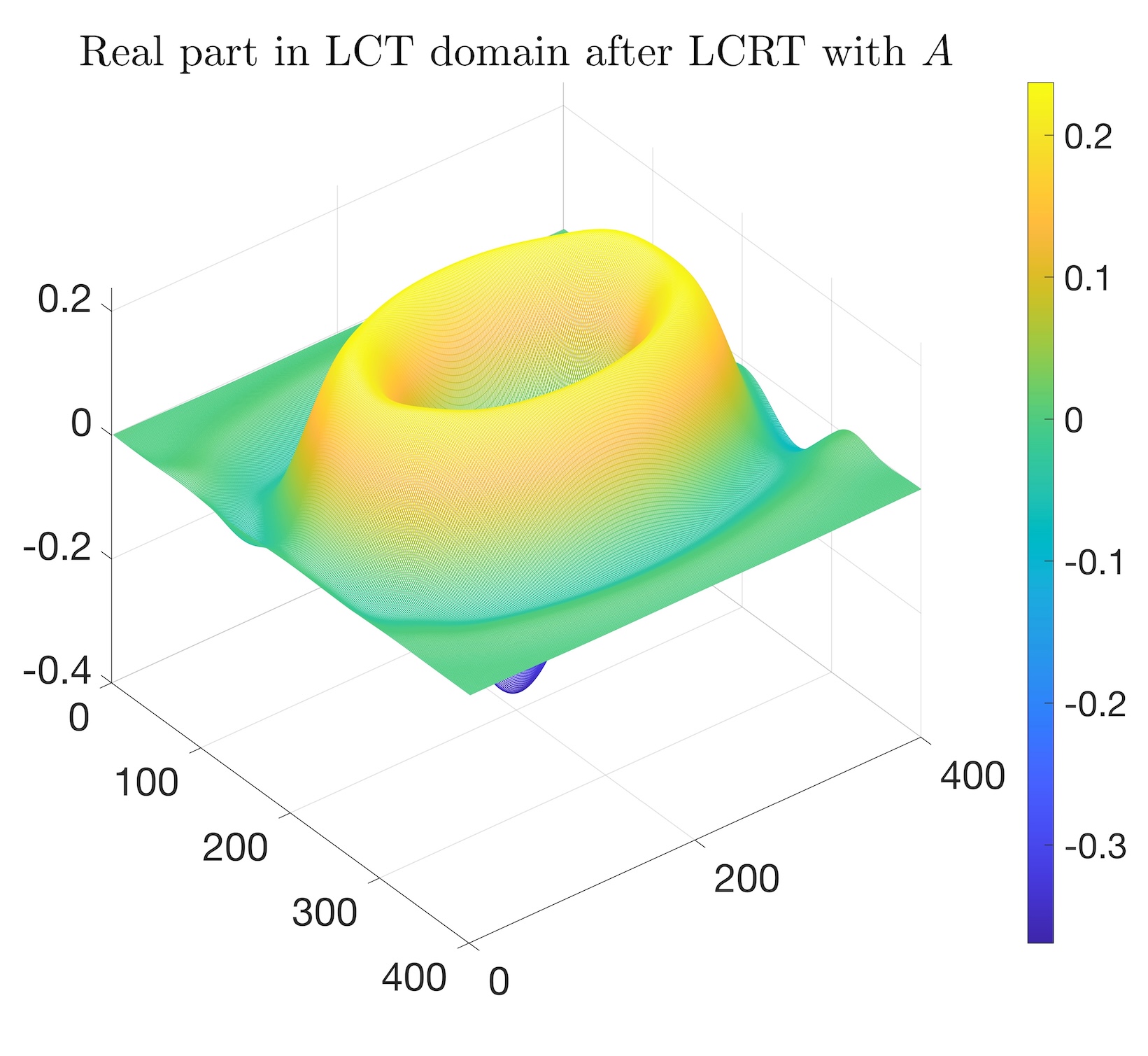}}\\
\vspace{0.5cm}
\subfigure[]{\includegraphics[width=0.45\linewidth]{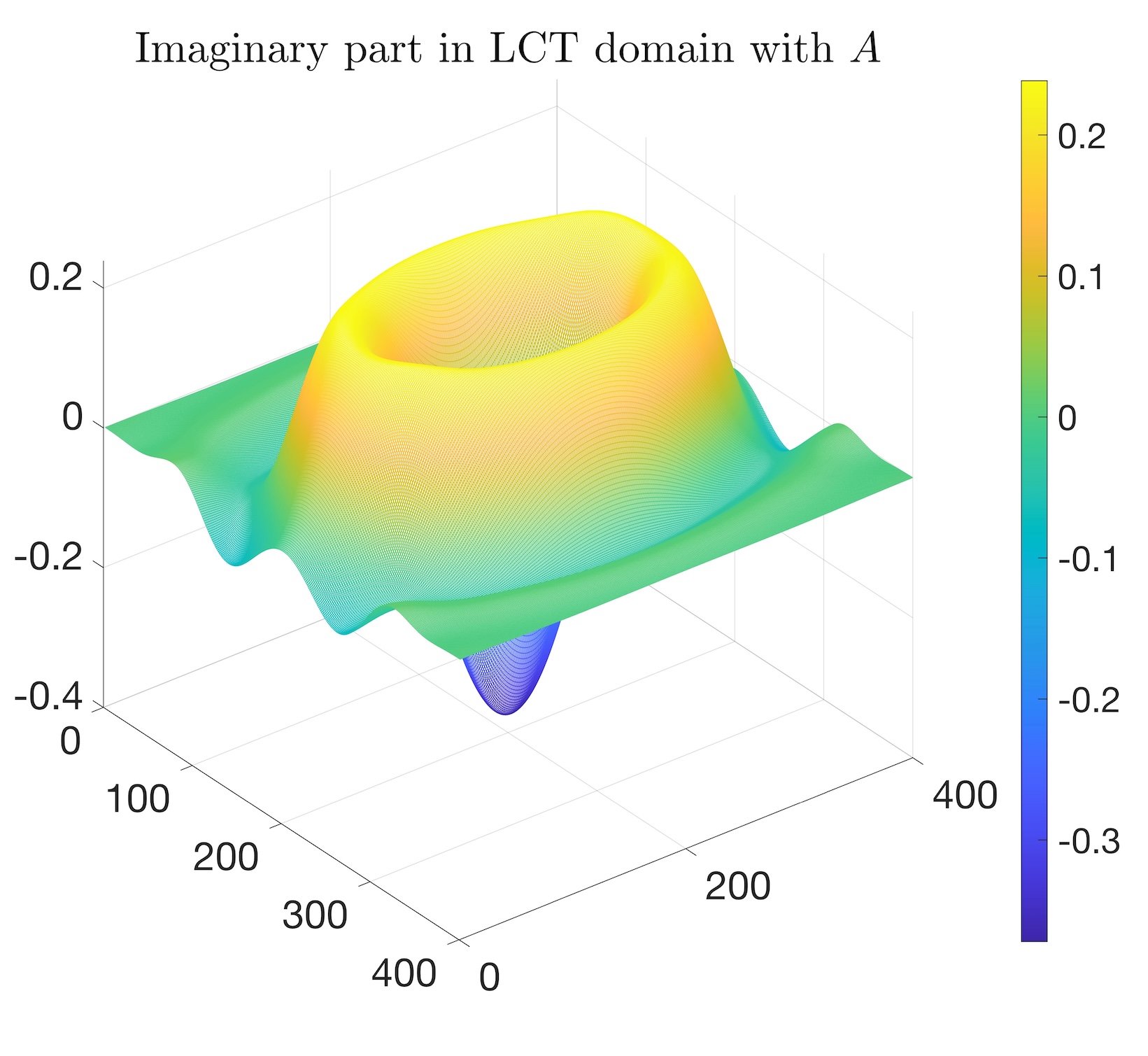}}\ \
\subfigure[]{\includegraphics[width=0.45\linewidth]{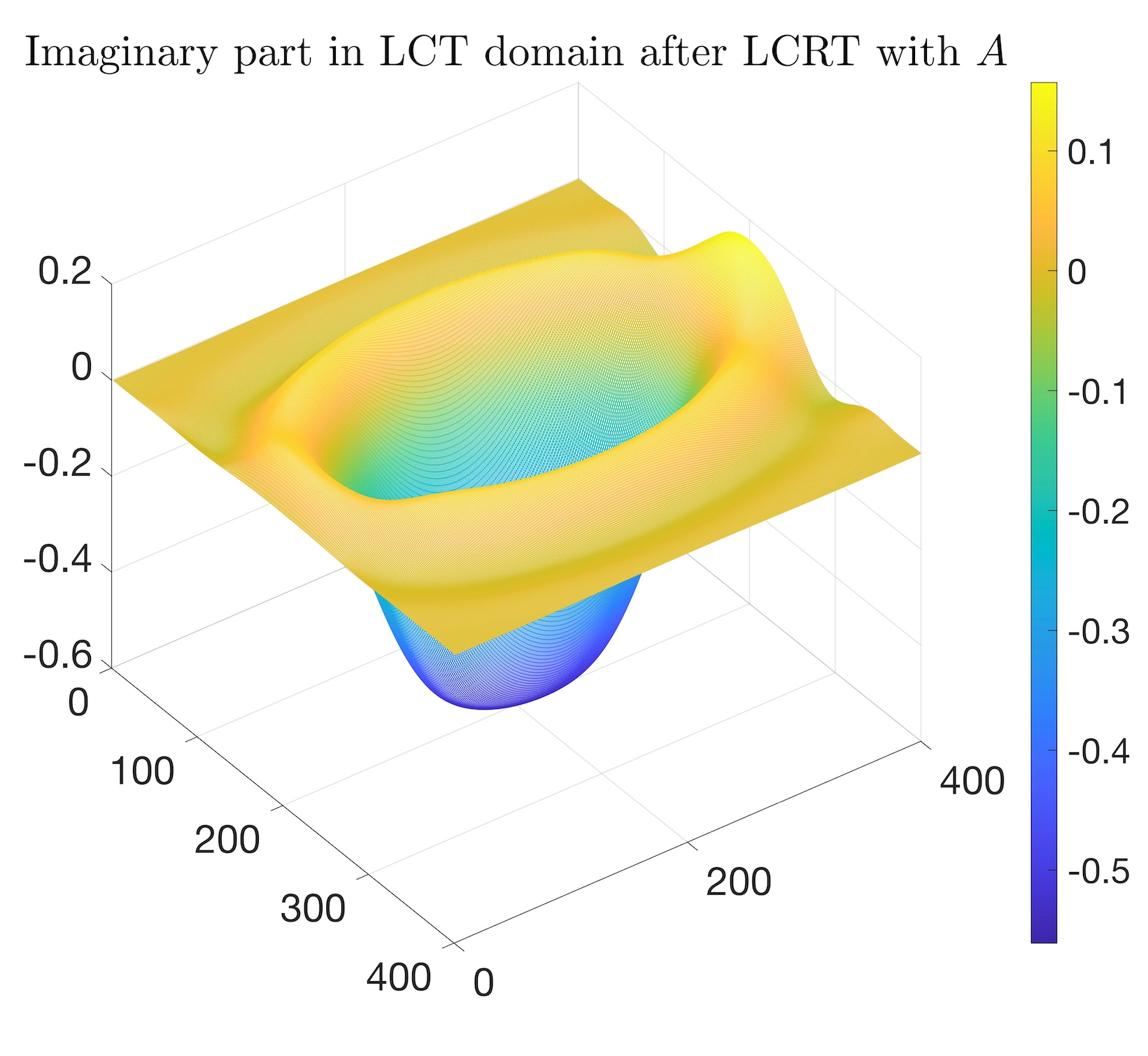}}
\caption{The amplitude, the real part, and the imaginary
part of the original test image in the LCT domain,
and those of the original image after applying
LCRTs in the LCT domain with the corresponding
parameter ${\boldsymbol A}$.}
\label{FIG5.3}
\end{figure}

\begin{figure}[H]
\centering
\subfigcapskip=-8pt
\subfigure[]{\includegraphics[width=0.4506\linewidth]{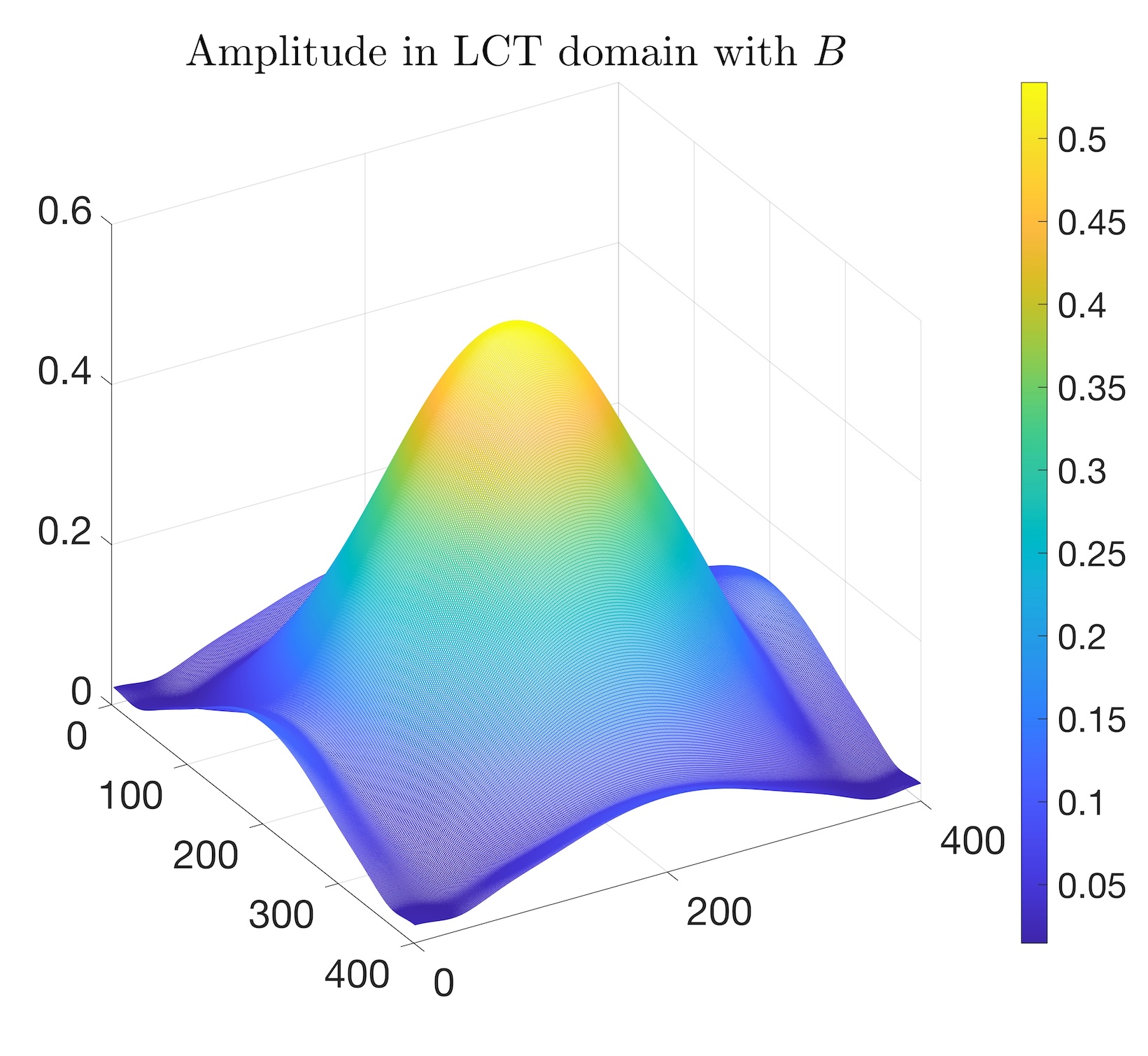}}\ \
\subfigure[]{\includegraphics[width=0.4506\linewidth]{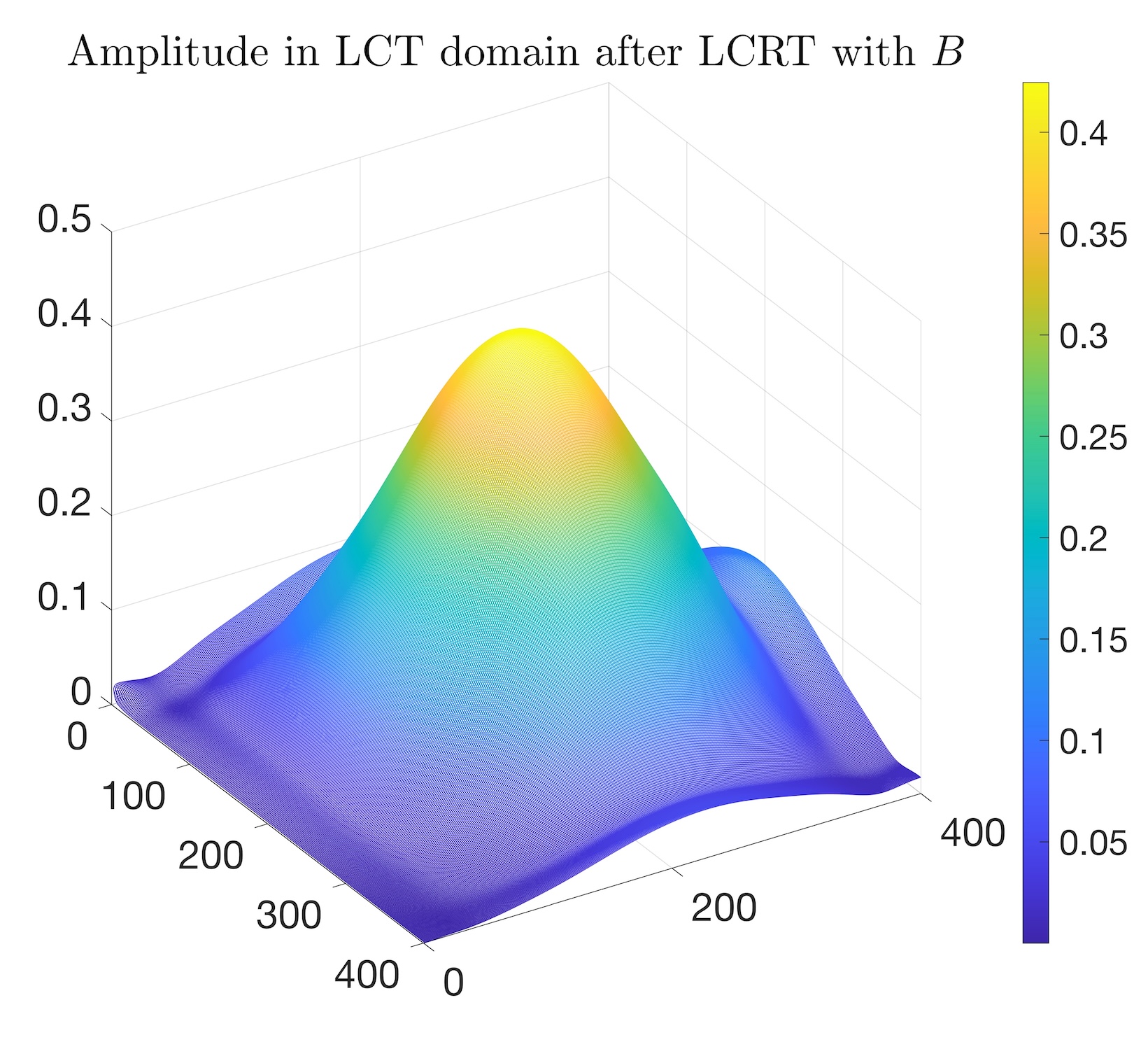}}\\
\vspace{0.5cm}
\subfigure[]{\includegraphics[width=0.4505\linewidth]{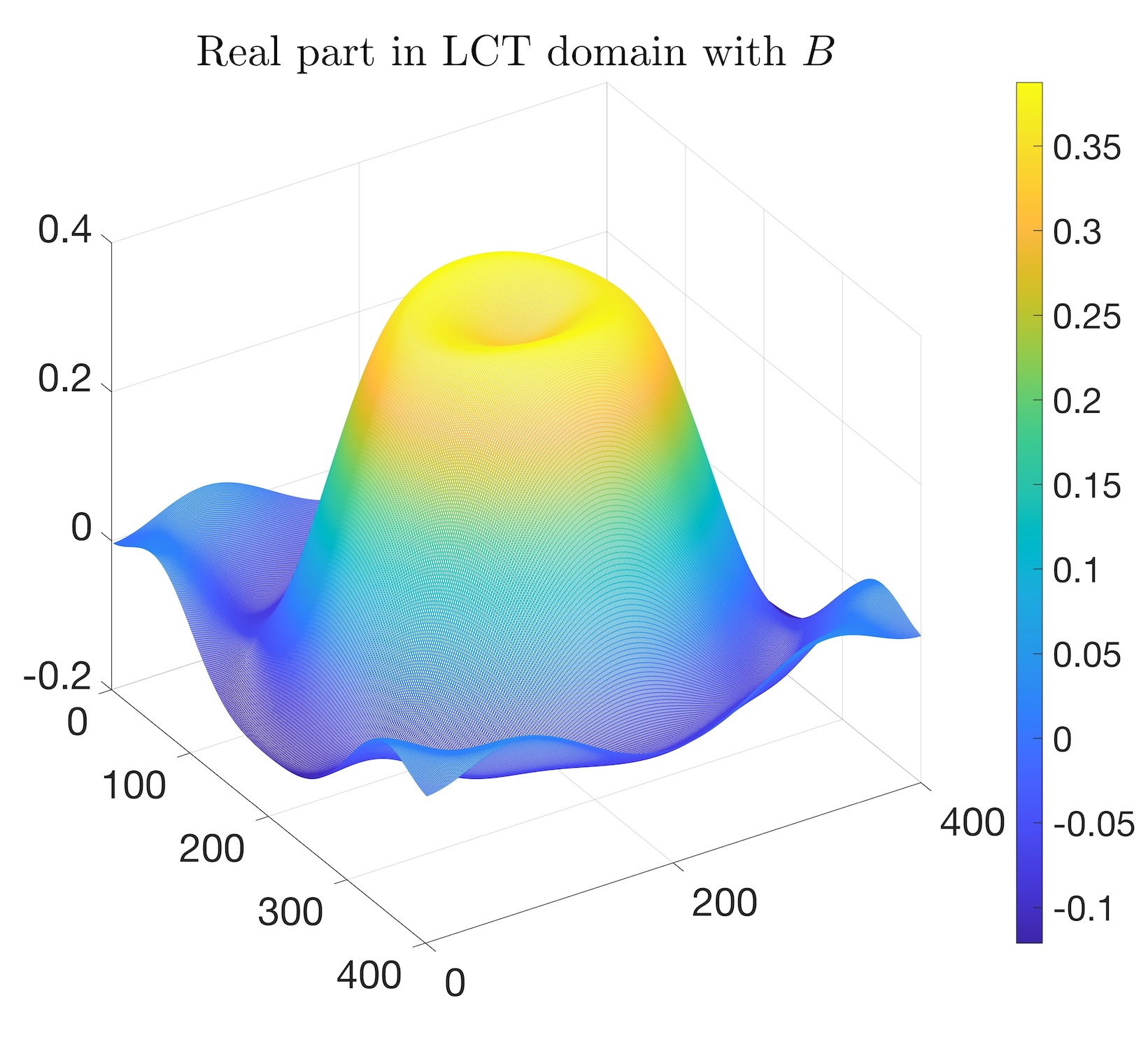}}\ \
\subfigure[]{\includegraphics[width=0.4505\linewidth]{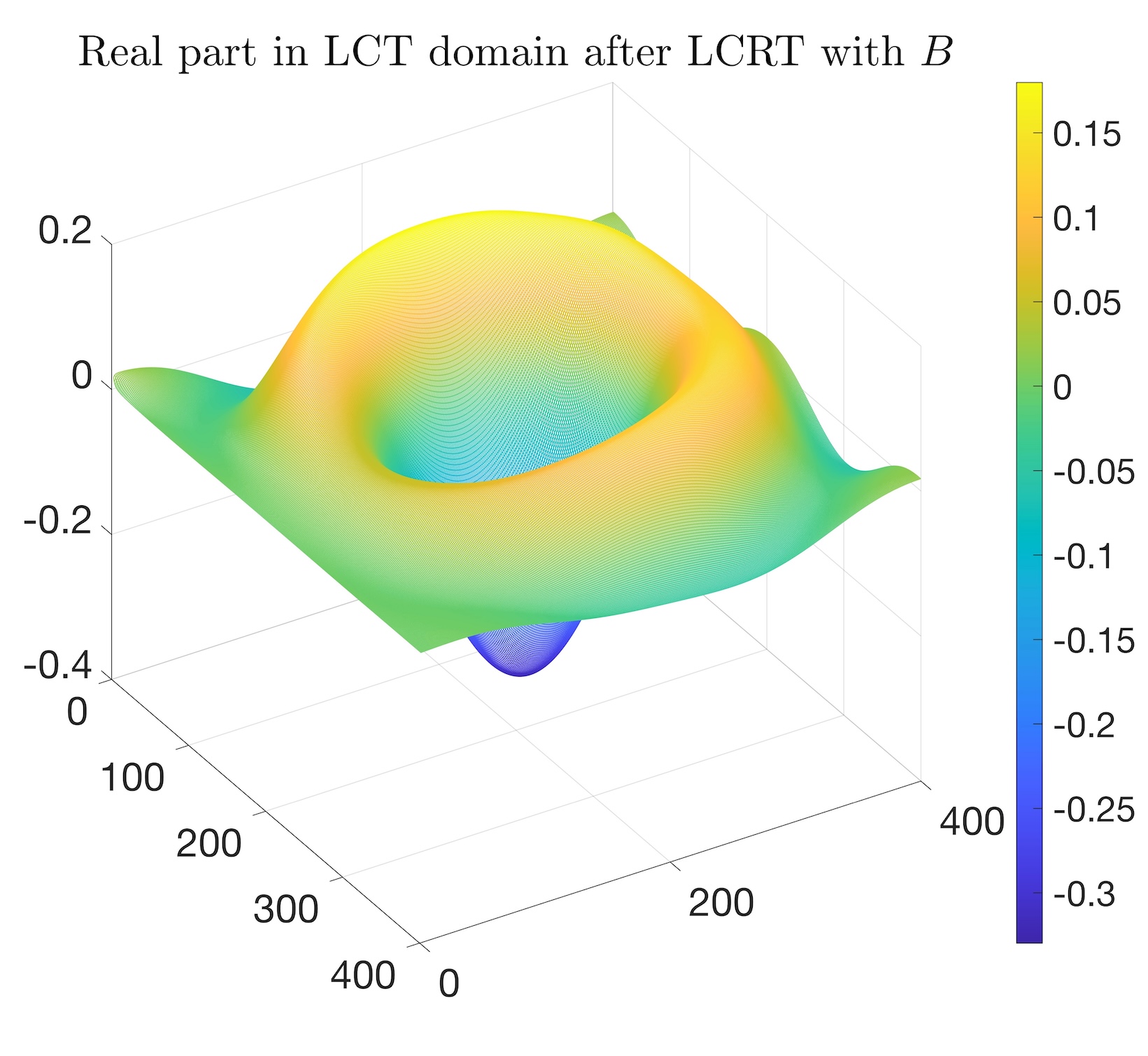}}\\
\vspace{0.5cm}
\subfigure[]{\includegraphics[width=0.45\linewidth]{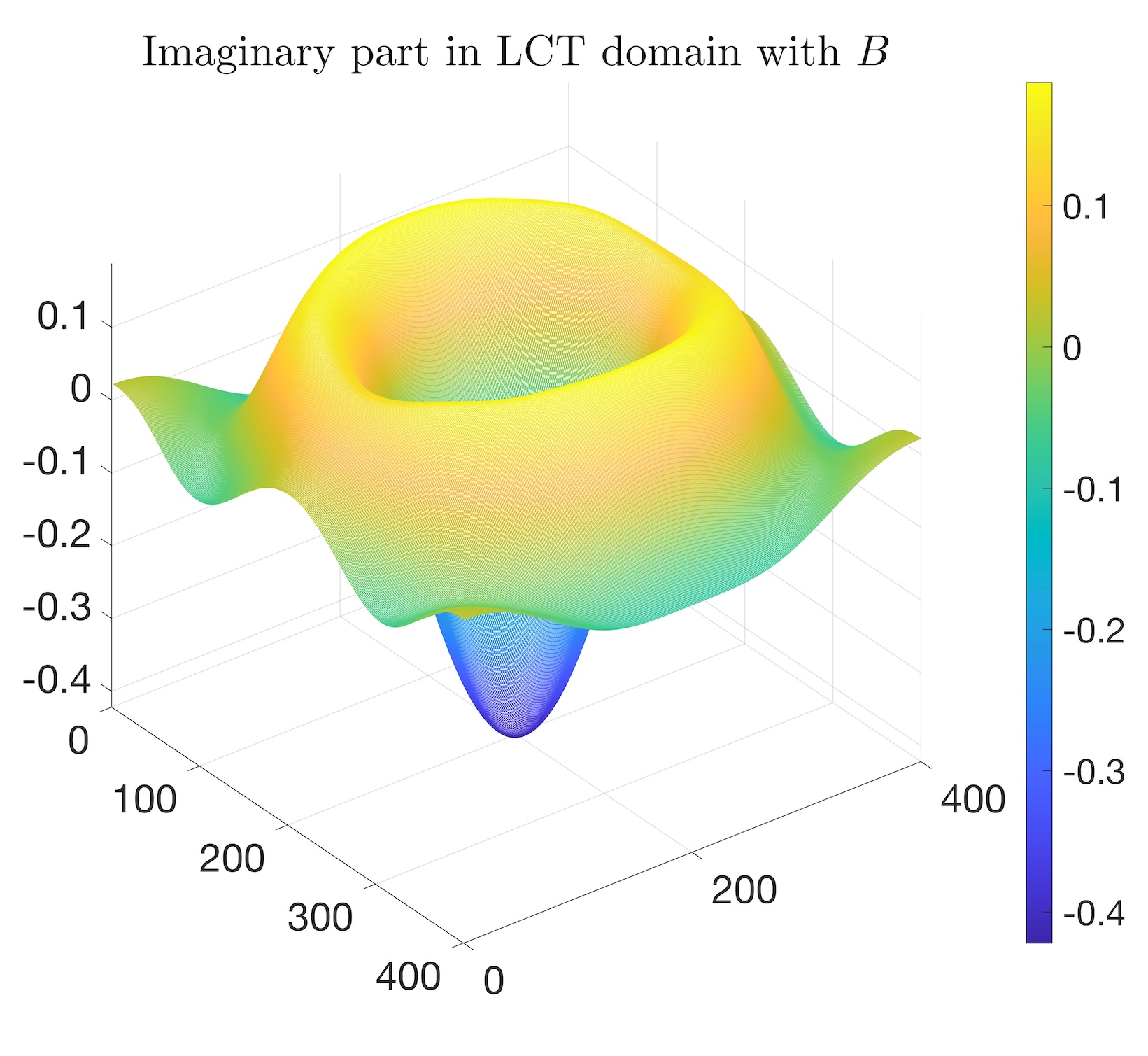}}\ \
\subfigure[]{\includegraphics[width=0.45\linewidth]{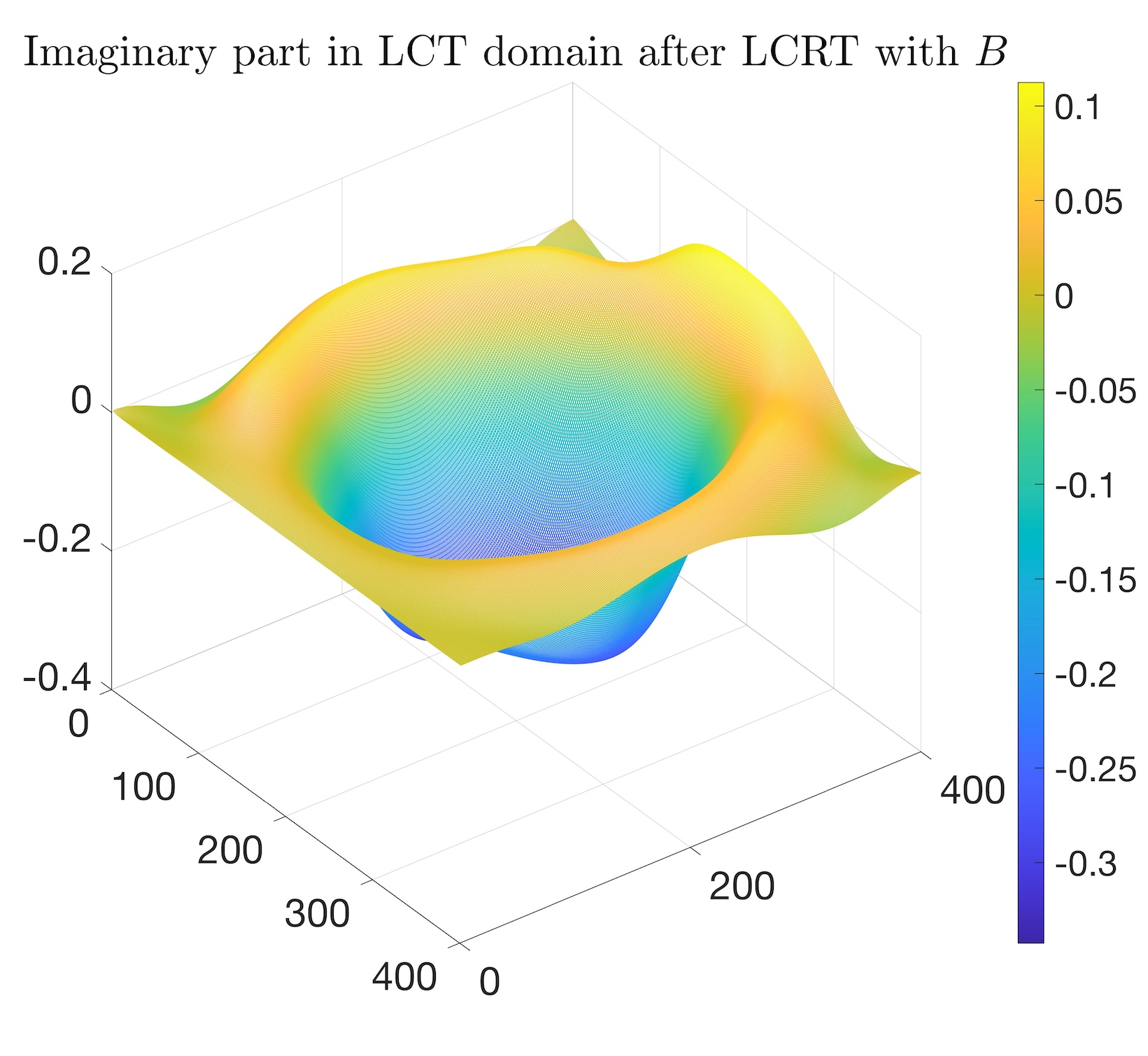}}
\caption{The amplitude, the real part, and the imaginary
part of the original test image in the LCT domain,
and those of the original image after applying
LCRTs in the LCT domain with the corresponding
parameter ${\boldsymbol B}$.}
\label{FIG5.4}
\end{figure}

\begin{figure}[H]
\centering
\subfigcapskip=-8pt
\subfigure[]{\includegraphics[width=0.32\linewidth]{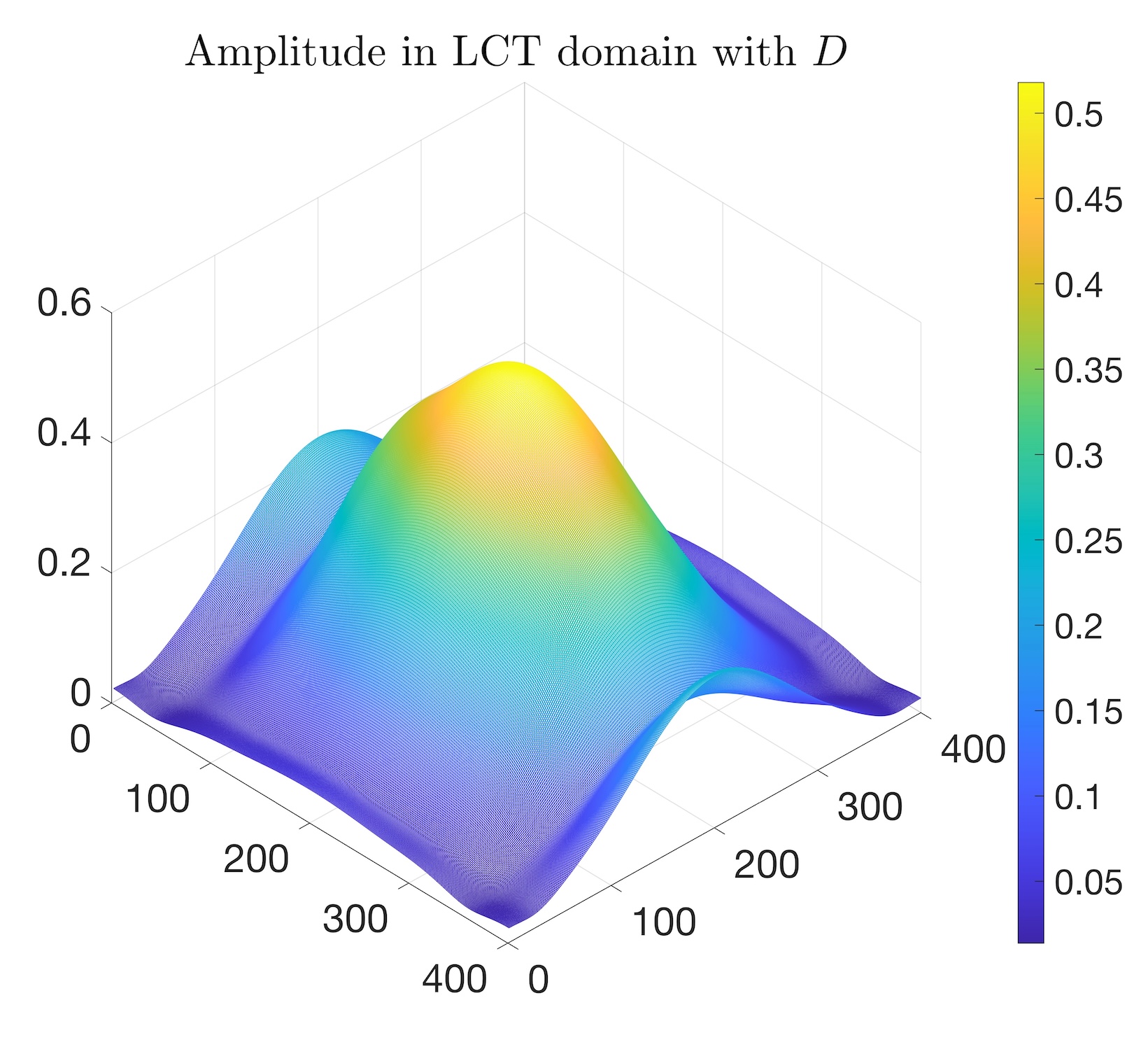}}\ \ \
\subfigure[]{\includegraphics[width=0.32\linewidth]{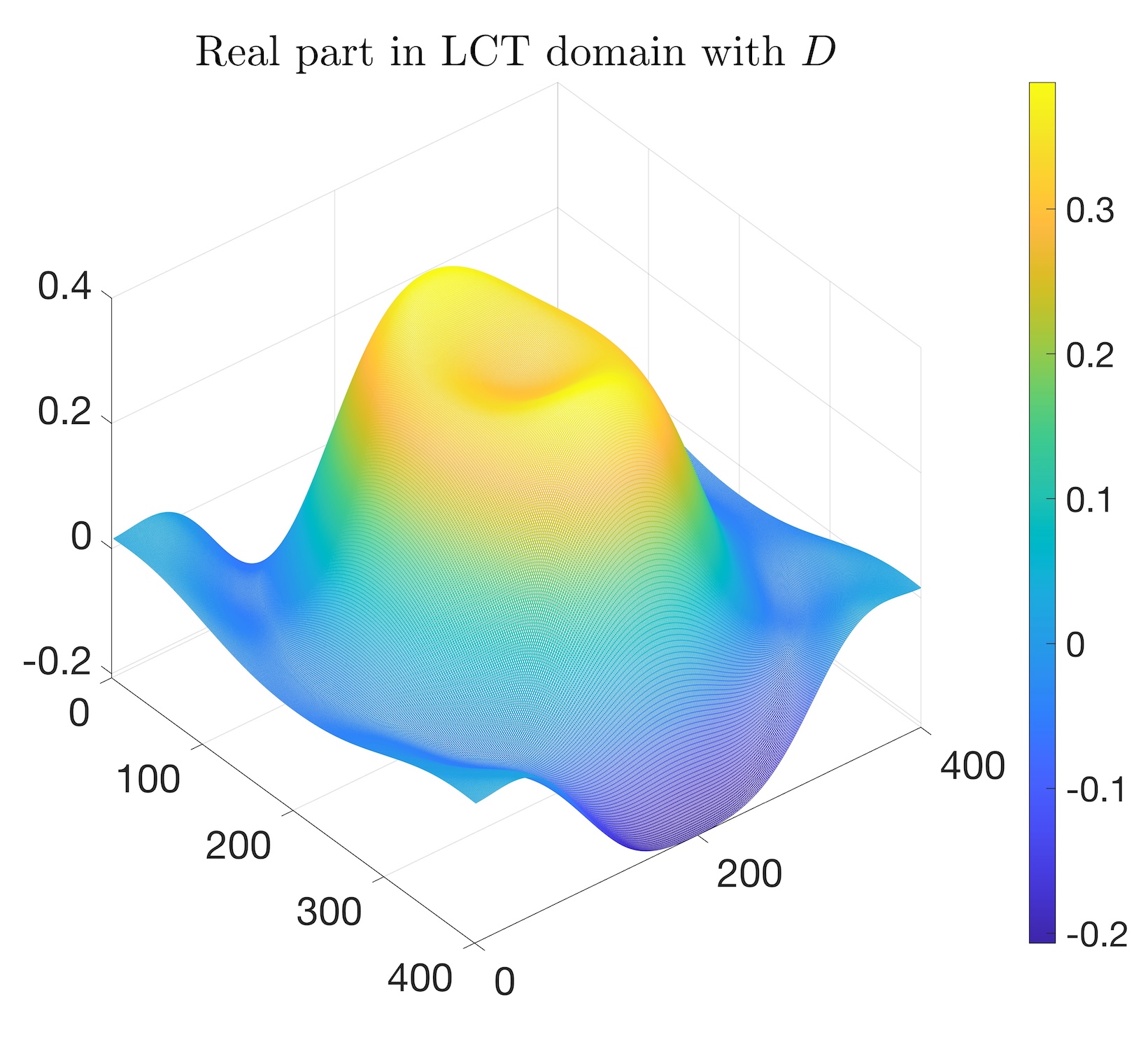}}\ \ \
\subfigure[]{\includegraphics[width=0.32\linewidth]{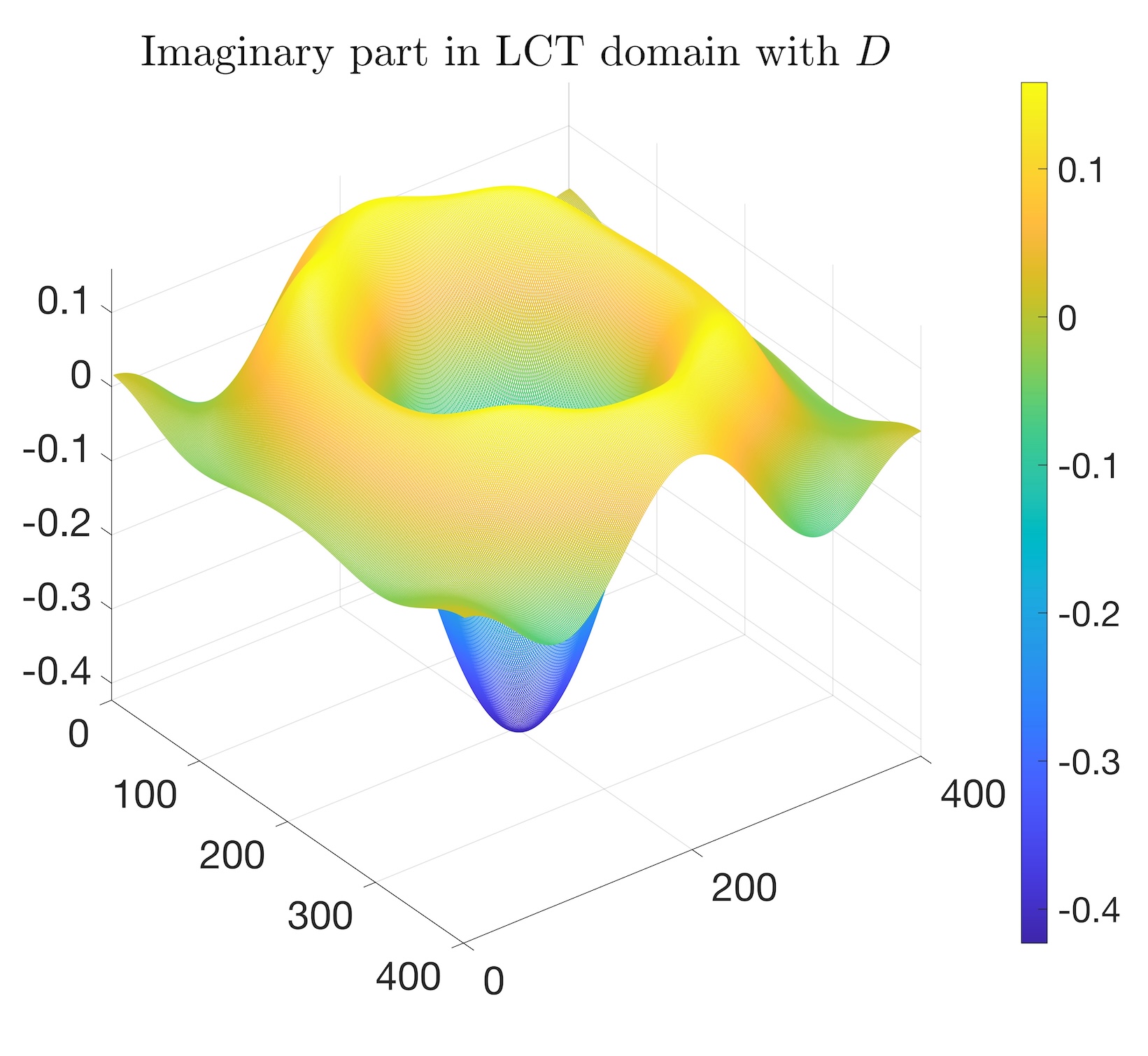}}\\
\vspace{0.3cm}
\subfigure[]{\includegraphics[width=0.32\linewidth]{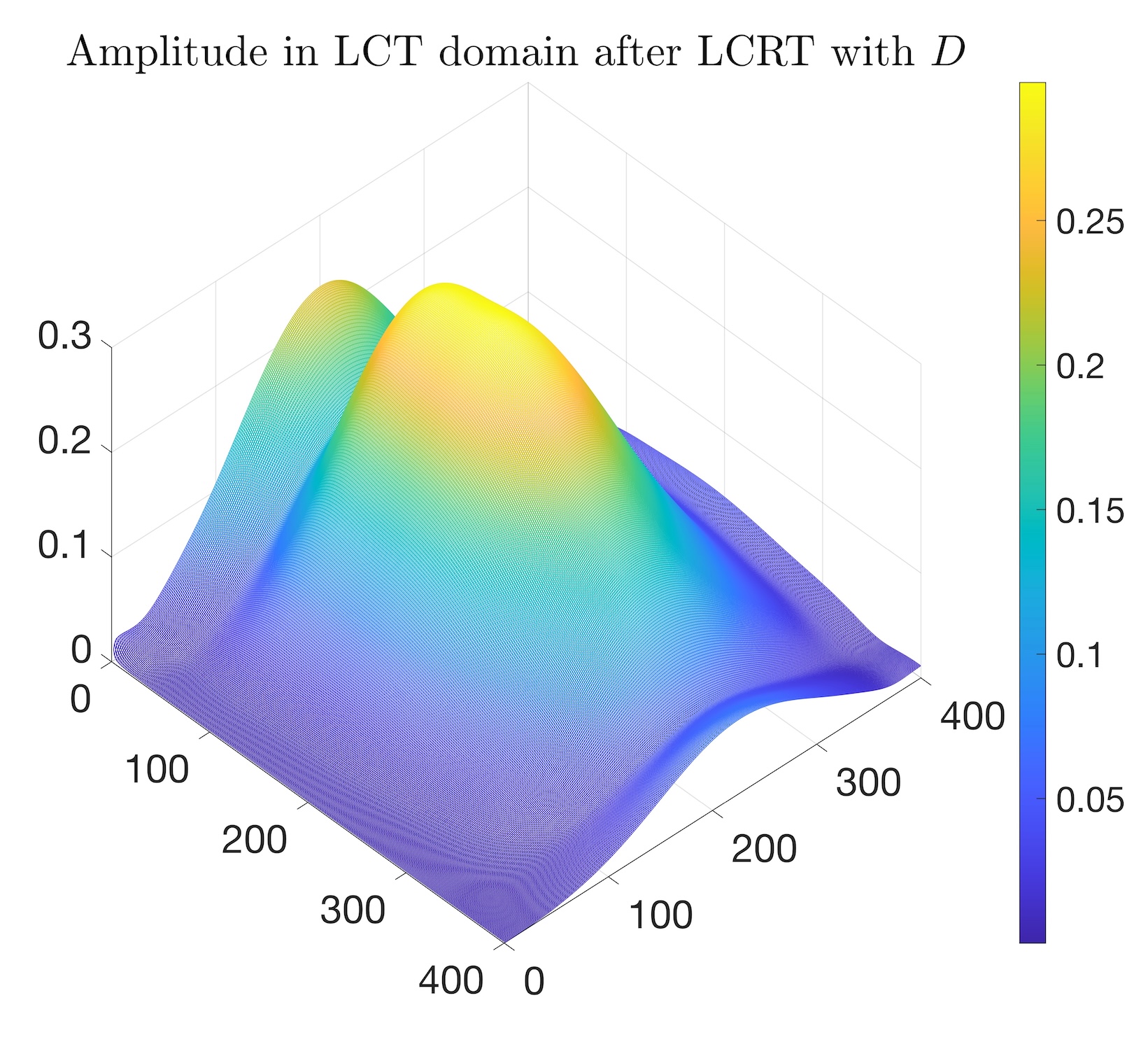}}\ \ \
\subfigure[]{\includegraphics[width=0.32\linewidth]{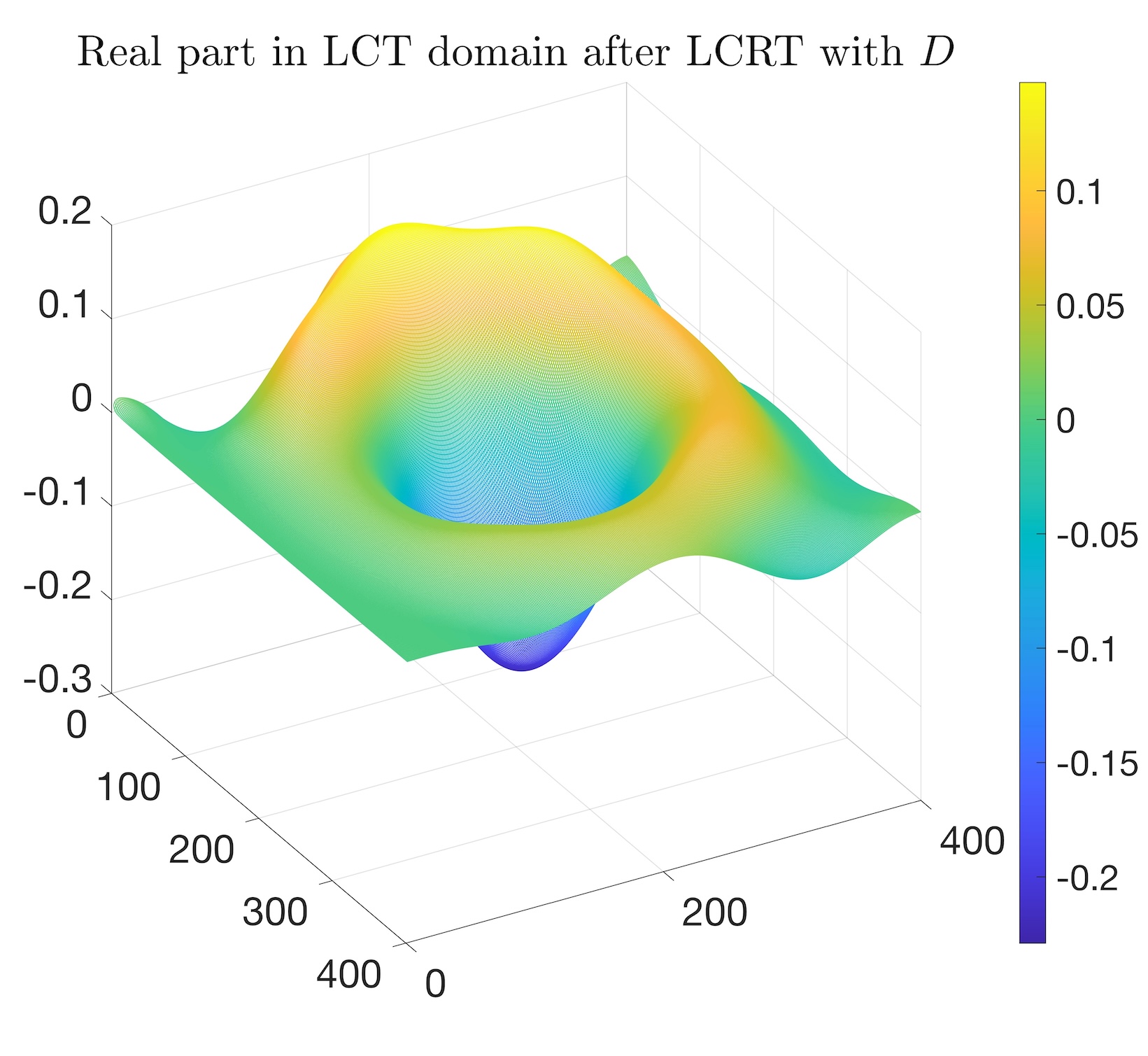}}\ \ \
\subfigure[]{\includegraphics[width=0.32\linewidth]{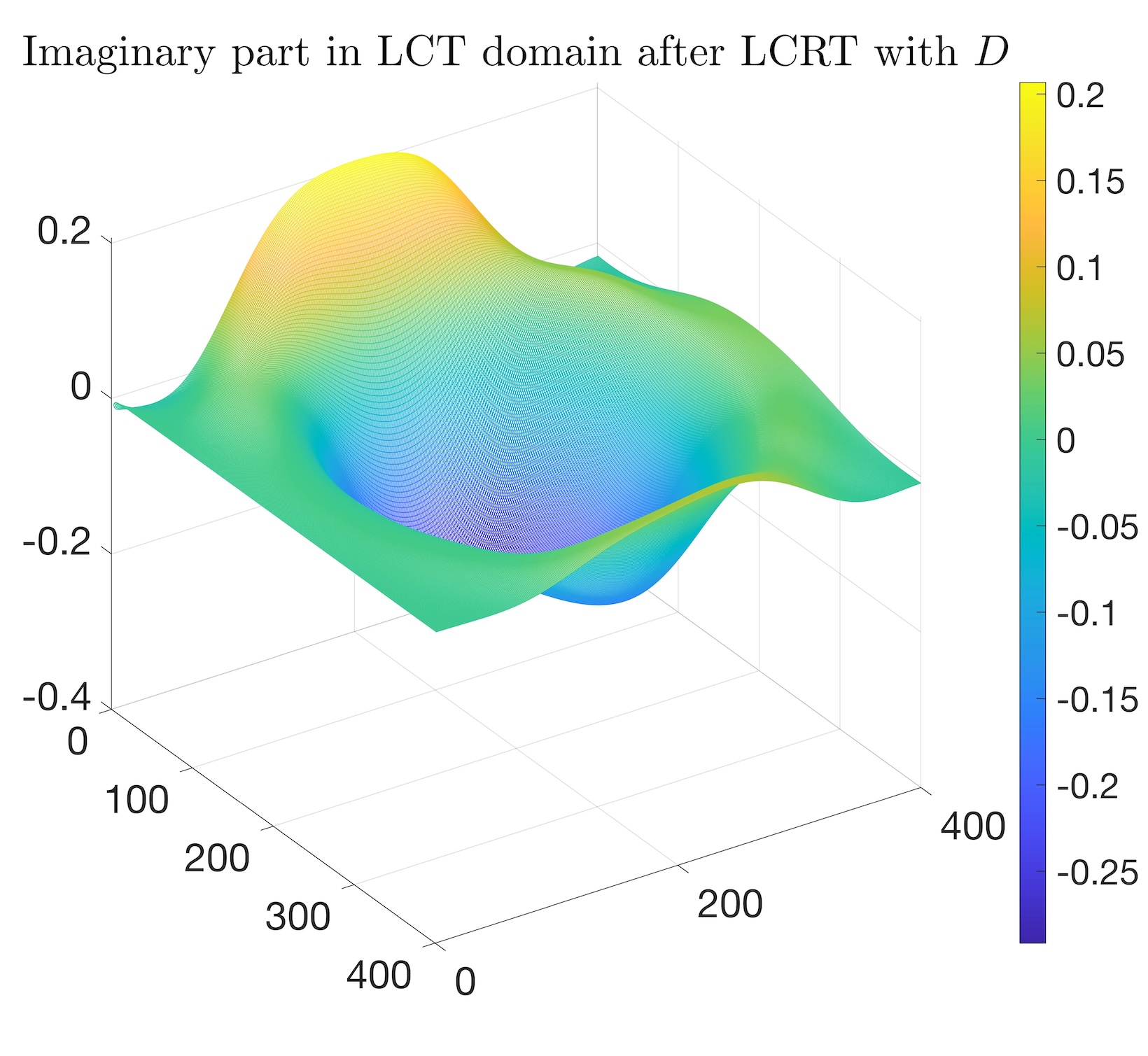}}\\
\vspace{0.3cm}
\subfigure[]{\includegraphics[width=0.32\linewidth]{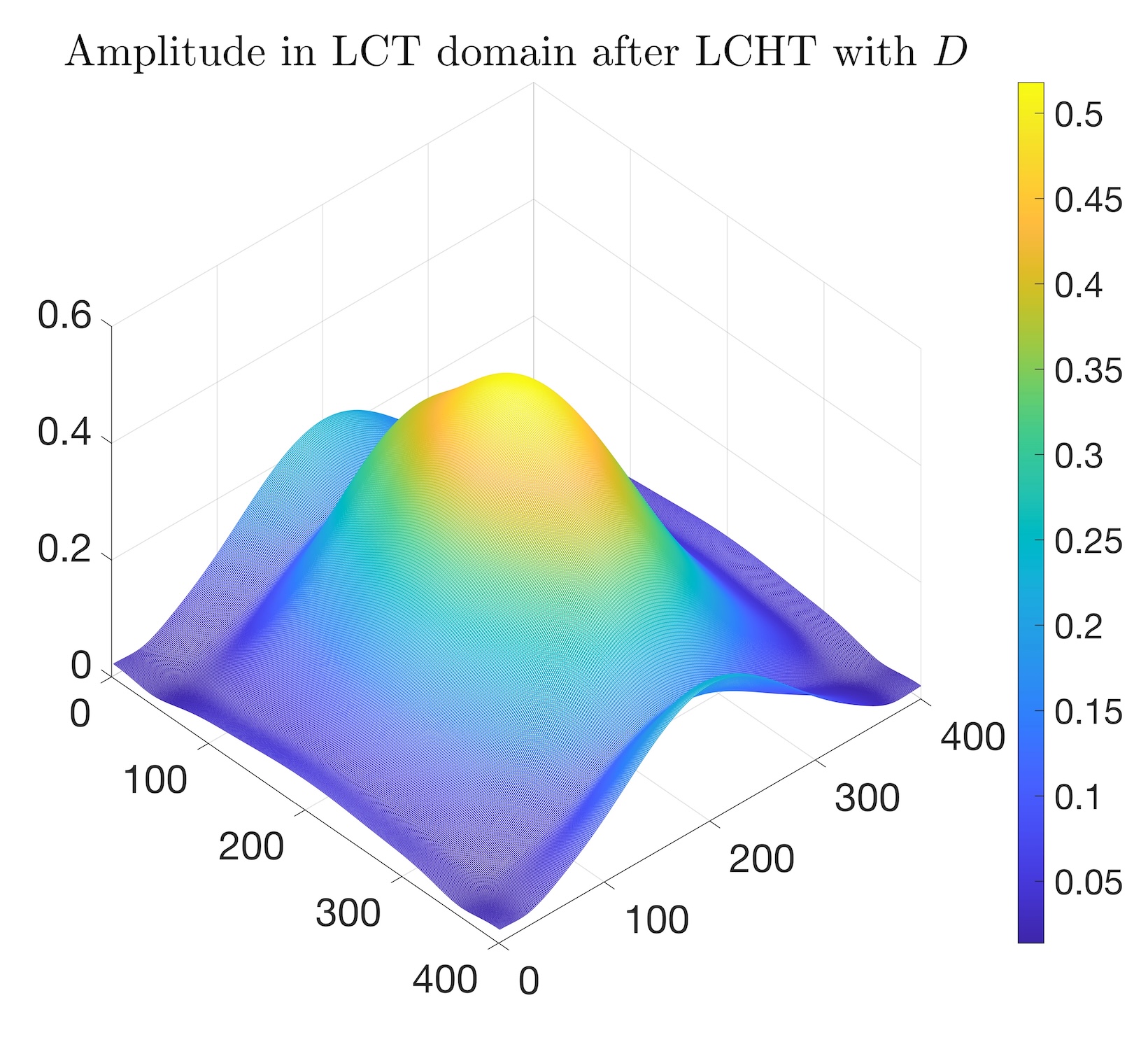}}\ \ \
\subfigure[]{\includegraphics[width=0.32\linewidth]{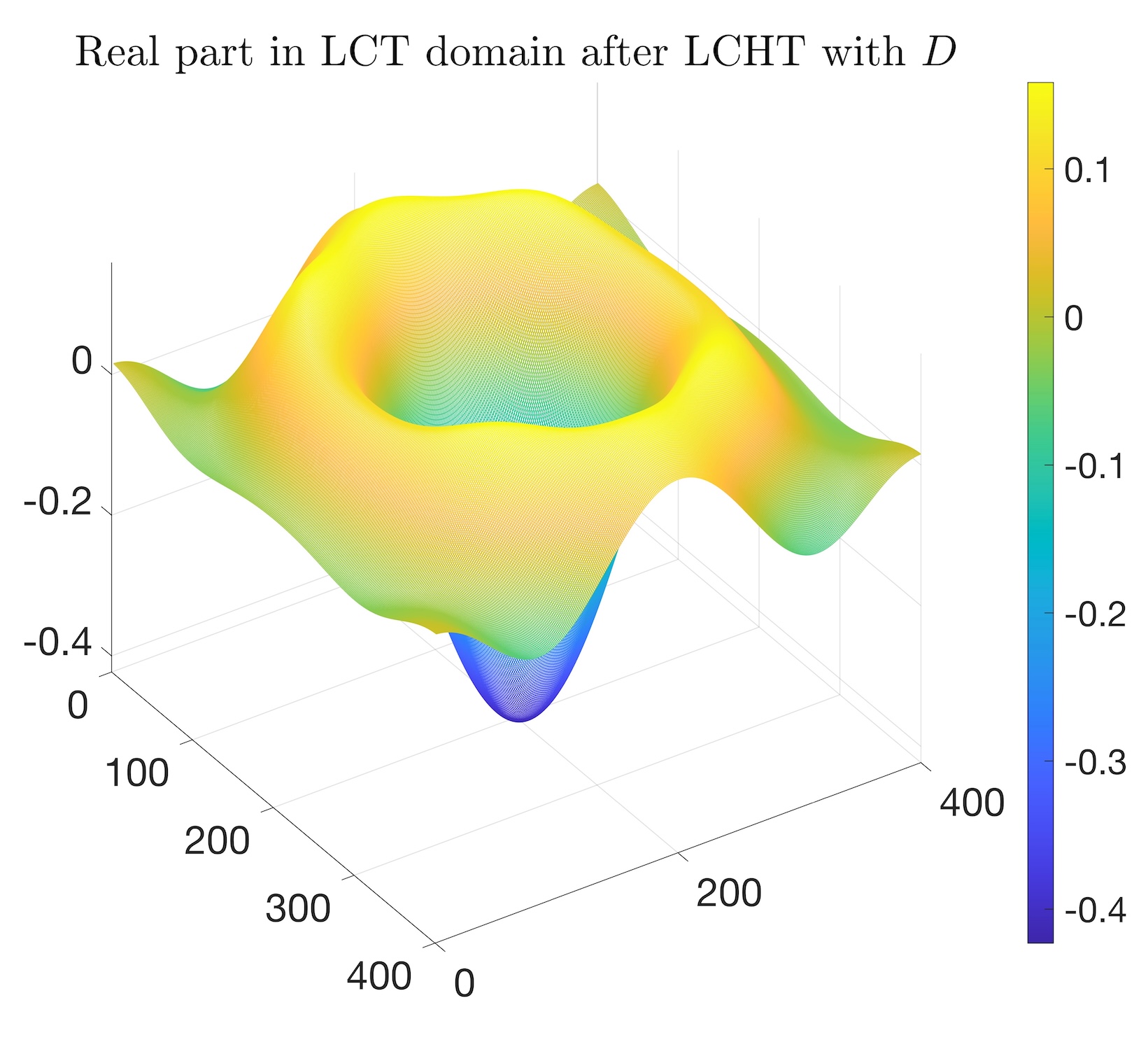}}\ \ \
\subfigure[]{\includegraphics[width=0.32\linewidth]{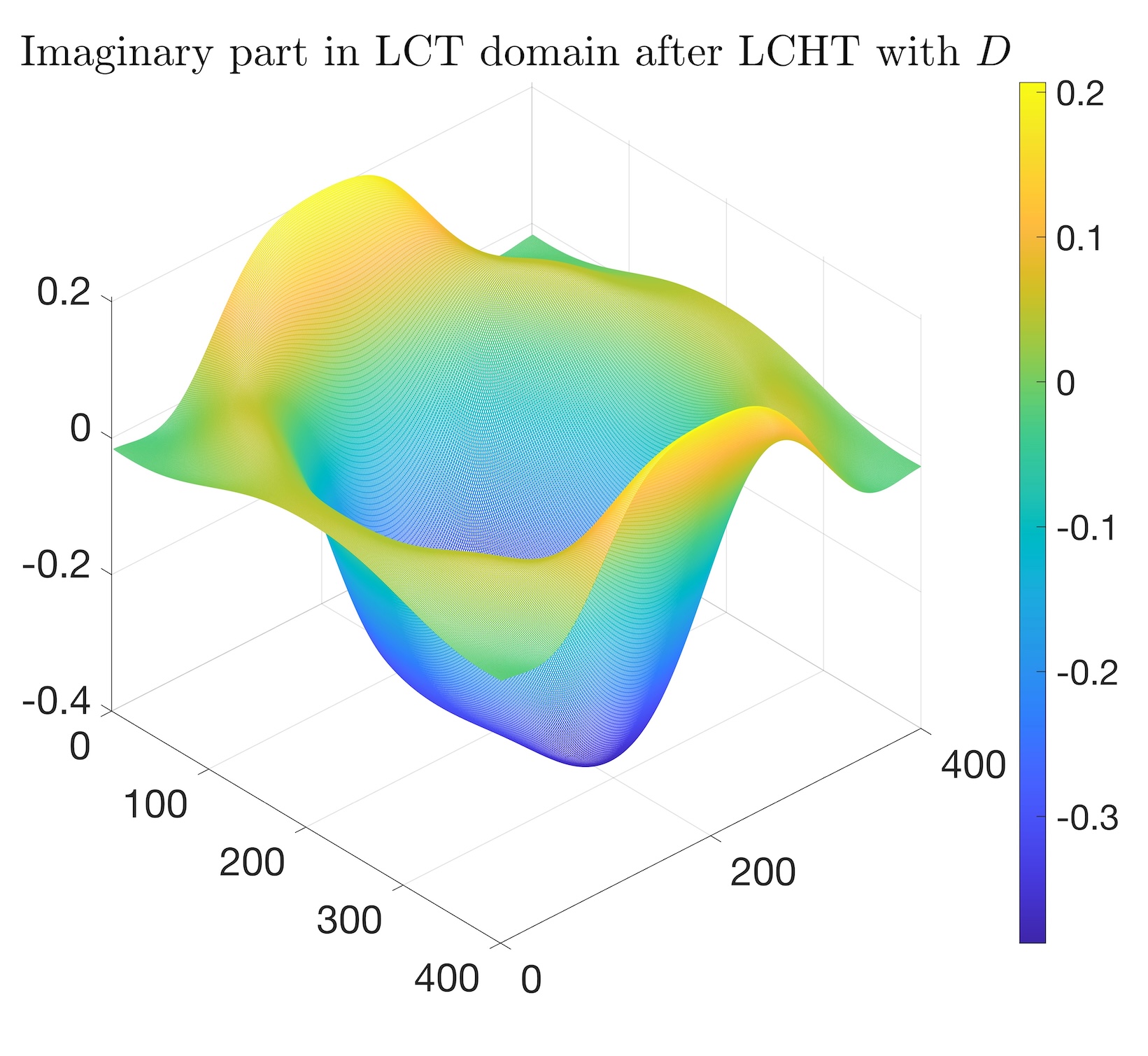}}
\caption{In the LCT domain, the amplitude, the real part,
as well as the imaginary part of the original test image, and those
of the original image after applying  LCRTs
and HLCHTs, are presented, with the corresponding
parameter being ${\boldsymbol D}$.}
\label{FIG9}
\end{figure}

In Figure \ref{FIG5.2}, Graphs (a), (c), and (e) represent
grayscale images obtained by applying LCRTs
$R_1^{\boldsymbol A}$, $R_1^{\boldsymbol B}$,
and $R_1^{\boldsymbol C}$ with different parameters on
Graph (a) in Figure
\ref{FIG5.1}. Their parameter matrices correspond, respectively, to
${\boldsymbol A}:=(A_1,A_2)$ with
$A_1:=\begin{bmatrix}{6}\ &{50}\\
{0.7}\ &{6}\end{bmatrix}$ and
$A_2:=\begin{bmatrix}{3}\ &{400}\\
{0.02}\ &{3}\end{bmatrix}$,
${\boldsymbol B}:=(B_1,A_2)$ with
$B_1:=\begin{bmatrix}{10}\ &{330}\\
{0.3}\ &{10}\end{bmatrix}$,
and ${\boldsymbol C}:=(C_1,C_2)$ with
$C_1:=\begin{bmatrix}{20}\ &{39.9}\\
{10}\ &{20}\end{bmatrix}$ and
$C_2:=\begin{bmatrix}{4}\ &{1}\\
{15}\ &{4}\end{bmatrix}$.  Also, in Figure \ref{FIG5.2},
Graphs (b), (d), and (f)   show
the 3D color images of Graphs (a), (c), and (e), respectively.
From Figure \ref{FIG5.2}, we infer that applying  LCRTs with
different parameter matrices to the same test image
yields significantly  different effects, which further
indicates that parameter matrices of LCRTs completely
determine its  role.

In Figure \ref{FIG5.3}, Graphs (a), (c), and (e) represent,
 respectively, the amplitude, the real part, and the imaginary part of the
original test image in Figure \ref{FIG5.1} in the LCT domain, while Graphs
(b), (d), and (f) represent, respectively, the amplitude, the real part, and the
imaginary part of the original test image in Figure \ref{FIG5.1} after
applying LCRTs in the LCT domain using the same parameter matrix
${\boldsymbol A}$ as in Figure \ref{FIG5.2}.
Comparing Graphs (a) with (b), we find that LCRTs
dramatically attenuate the amplitude in the LCT domain. Additionally,
by comparing Graphs (c) with (f) and Graphs (e) with (d),
we find that LCRTs  have an obvious shifting effect in the
LCT domain.

Replacing the parameter matrix
${\boldsymbol A}$  with the parameter matrix
${\boldsymbol B}$ (both ${\boldsymbol A}$ and
${\boldsymbol B}$ are the same as  in Figure \ref{FIG5.2})
and then repeating the experiment same as in
Figure \ref{FIG5.3}, we obtain Figure \ref{FIG5.4}.
We find  that, in different LCT domains,
LCRTs always have obvious and different  effects of shifting
and amplitude attenuation. Therefore, it is important to
select the appropriate LCT domain for solving
various problems.

In Figure \ref{FIG9}, we compare the amplitude, the real
part, and the imaginary part of the original test image in Figure
\ref{FIG5.1}, the
original test image in Figure
\ref{FIG5.1} after applying LCRTs $R_1^{\boldsymbol{D}}$,
and the original test image in Figure
\ref{FIG5.1} after applying HLCHTs  $H_x^{D_1}$
in the LCT domain, where the corresponding parameter matrix is
${\boldsymbol D}:=(D_1,D_2)$, where
$D_1:=\begin{bmatrix}{6}\ &{500}\\
{0.07}\ &{6}\end{bmatrix}$ and
$D_2:=\begin{bmatrix}{4}\ &{300}\\
{0.05}\ &{4}\end{bmatrix}$.
Graphs (a), (d), and (g) represent the amplitude for the aforementioned
three cases, while Graphs (b), (e), and (h) represent their
real parts, and Graphs (c), (f), and (i) represent their
imaginary parts. Observing Graphs (a), (d), and (g),
we conclude that LCRTs have the effect of reducing amplitude
in the LCT domain, whereas the HLCHT does not exhibit
this effect.  Comparing Graphs (b), (e), (h), (c), (f), and (i),
we find that LCRTs induce a phase shift of $\pi/2$ in the
LCT domain, and HLCHTs also have the same effect.

 In summary,  through a series of  image simulation experiments,
we have verified that LCRTs  are able to yield both shifting and attenuation
of amplitudes in the LCT domain. Compared to the HLCHT,
the LCRT is able to attenuate amplitudes in the LCT
domain, which may enhance performance in image feature
extraction, image quality enhancement, and texture analysis, while,
compared to fractional Riesz transforms, LCRTs provide a
broader selection of frequency domains and greater flexibility,
potentially resulting in more precise extraction of local image
feature information.

\section{Applications}\label{sec4}
 
In this section, we present the application of LCRT to 
refining image edge detection. In Section \ref{subsec5.1}, we 
introduce the concept of the sharpness $R^{\rm E}_{\rm sc}$ of the edge strength and  continuity of images associated with the LCRT, and propose 
 a new LCRT image edge detection method (namely LCRT-IED method)
and provide its mathematical foundation. 
 The experimental results  demonstrate that this new LCRT-IED
method is able to gradually control of the edge strength and  continuity by 
subtly adjusting the parameter matrices of the LCRT.  Meanwhile, 
 in local edge feature extraction, the LCRT-IED method better 
 preserves fine image edge details in some sub-regions. In Section 
 \ref{subsec5.2}, we further summarize the advantages and 
 features of the LCRT-IED method in image processing. 

\subsection{Refining Image Edge Detection by Using
$R^{\rm E}_{\rm sc}$\label{subsec5.1}}

The image processing has long been a central focus
in information sciences, playing a vital role in applied
 sciences and garnering significant attention (see, for
example, \cite{bshyh2007,gcn21,gcn21-2,hcg24}).
In image processing, edge detection is one of the
key steps in understanding image content. It involves extracting
important structural information from images, such as contours
and textures, which is crucial for image analysis and computer
vision tasks. While traditional edge detection methods, such
as the Canny and Sobel edge detectors, perform well in certain
scenarios, they often struggle with complex or noisy images
(see, for example, \cite{cd14,dwy13}). Recently, the Riesz
transform has garnered attention as a novel edge detection
technique due to its advantages in the analysis of multi-scale
and directional analysis tasks. The Riesz transform can provide
richer edge information compared to traditional methods,
demonstrating its superior performance in image processing
(see, for example, \cite{fglwy23,la10}). Based on
the theoretical  foundations of Riesz transforms and fractional Riesz
transforms, this section explore the application of LCRTs
in edge detection. Through the whole section, we work on $\mathbb R^2$.

In image processing and computer vision, edge strength
and continuity are two crucial attributes that are essential
for image analysis and understanding. \emph{Edge strength}
refers to the prominence of an edge within an image,
typically associated with the rate of strength change at
the edge. Regions with high edge strength exhibit
greater changes in image strength, indicating a more
pronounced edge. Edge strength information plays a
vital role in tasks, such as image segmentation, feature
extraction, and object recognition. In some cases,
edge strength can be used to differentiate among
different types of edges, such as contour edges and texture
edges. \emph{Edge continuity} describes the coherence of an
edge within an image. A continuous edge means that
edge points are closely connected without interruptions.
In image segmentation and object detection, edge
continuity is a key feature. Continuous edges help to
define object boundaries, leading to more accurate
image segmentation and object recognition.

Felsberg and Sommer \cite{fs01} introduced the monogenic
signal in image processing, while Larkin et al. \cite{lbo01}
introduced this signal in the optical context. For an image
$f$, the \emph{monogenic signal} $\left(p,q_1, q_2\right)$
is defined as the combination of $f$ and its LCRTs as follows:
$$\left(p,q_1, q_2\right):=\left(f,
R_1^{\boldsymbol{A}}f,
R_2^{\boldsymbol{A}}f \right),$$
where the LCRTs $R_1^{\boldsymbol{A}}$ and
$R_2^{\boldsymbol{A}}$  are the same as in
Theorem \ref{LCRT-rf} with $n=2$. Due to the existence of fast
algorithms for the LCT, this approach significantly reduces
the computational complexity compared to the convolution type LCRT.
For any $\boldsymbol{x}\in \mathbb{R}^2$,
the \emph{local amplitude value} $A(\boldsymbol{x})$,
the \emph{local orientation} $\theta(\boldsymbol{x})$,
and the \emph{local phase} $P(\boldsymbol{x})$ in the monogenic
signal in the image are respectively defined by setting
$$A(\boldsymbol{x}):=\sqrt{p(\boldsymbol{x})^2+|q_1
(\boldsymbol{x})|^2+|q_2(\boldsymbol{x})|^2},$$
\begin{equation}\label{eqlo}
\theta(\boldsymbol{x}):=\tan^{-1}\left\{\left|\frac{q_2
(\boldsymbol{x})}{q_1(\boldsymbol{x})}\right|\right\},\ \mathrm{and}\
P(\boldsymbol{x}):=\tan^{-1}\left[\frac{p(\boldsymbol{x})}
{\sqrt{|q_1(\boldsymbol{x})|^2+|q_2(\boldsymbol{x})|^2}}\right].
\end{equation}
 
Next, we introduce the sharpness of the edge strength and 
continuity of the image associated with the LCRT.

\begin{definition}\label{is}
Let $M:=\begin{bmatrix}
{a}&{b}\\{c}&{d}
\end{bmatrix}\in{M_{2\times2}}(\mathbb{R})$ be either of
$A_1$ and $A_2$ appearing in the definition of the LCRTs
$\{R_1^{\boldsymbol{A}},R_2^{\boldsymbol{A}}\}$ with
$\boldsymbol{A}:=(A_1,A_2)$ as in Definition \ref{Rieszdef1} with $n=2$.
The \emph{sharpness} of the edge strength and continuity
of  images related to LCRTs, $R^{\rm E}_{\rm sc}(M)$, is defined 
by setting
$$R^{\rm E}_{\rm sc}(M):=\frac{b}{c}.$$
\end{definition}

On the LCRTs $\{R_1^{\boldsymbol{A}},R_2^{\boldsymbol{A}}\}$, using the 
sharpnesses $\{R^{\rm E}_{\rm sc}(A_1),R^{\rm E}
_{\rm sc}(A_2)\}$ of the edge strength and continuity
of  images related to LCRTs introduced  in Definition \ref{is},
we find their following fundamental properties.

\begin{theorem}\label{p-sharp}
Let  $\boldsymbol A:=(A_1,A_2)$
with $A_j:=\begin{bmatrix}
		{a_j}&{b_j}\\
		{c_j}&{d_j}
\end{bmatrix}\in{M_{2\times2}}(\mathbb{R})$
and both $b_j\ne0$ and $c_j\ne0$ for any $j\in\{1,2\}$, and
let the {\rm LCRT}s $\{R_1^{\boldsymbol A},R_2^{\boldsymbol A}\}$
be the same as in Definition \ref{Rieszdef1} with $n=2$.
Then the following assertions hold.
\begin{itemize}
\item[\rm{(i)}]  For any $f\in\mathscr{S}(\mathbb{R}^2)$ 
and $\boldsymbol x\in\mathbb R^2$, 
the {\rm LCRT}s $\{R_1^{\boldsymbol A},R_2^{\boldsymbol A}\}$
can be expressed as
\begin{equation}\label{eq51}
R_1^{\boldsymbol A}f(\boldsymbol x)={\widetilde{c}_2}
e^{-i\sum^{2}_{j=1}\frac{d_jx_j^2}{R^{\rm E}
_{\rm sc}(A_j)c_j}}\,{\rm p}.
{\rm v}.\int_{\mathbb{R}^2} \frac{({x_1}-{y_1})
f(\boldsymbol y) e^{i\sum^{2}_{j=1}\frac{a_jy_j^2}
{R^{\rm E}_{\rm sc}(A_j)c_j}}}{| \boldsymbol x-\boldsymbol y
|^{3}}\,d\boldsymbol{y}
\end{equation}
and
\begin{equation}\label{eq52}
R_2^{\boldsymbol A}f(\boldsymbol x)={\widetilde{c}_2}
e^{-i\sum^{2}_{j=1}\frac{d_jx_j^2}{R^{\rm E}_{\rm sc}(A_j)c_j}}\,{\rm p}.
{\rm v}.\int_{\mathbb{R}^2} \frac{({x_2}-{y_2})
f(\boldsymbol y) e^{i\sum^{2}_{j=1}\frac{a_jy_j^2}{R^{\rm E}
_{\rm sc}(A_j)c_j}}}{| \boldsymbol x-\boldsymbol y
|^{3}}\,d\boldsymbol{y},
\end{equation}
where ${\widetilde{c}_2}$ is the same as in Definition \ref{Rieszdef1}.
\item[\rm{(ii)}]
For any $f\in\mathscr{S}(\mathbb{R}^2)$  and $\boldsymbol x\in\mathbb R^2$, one has
$$\lim_{\genfrac{}{}{0pt}{}{R^{\rm E}_{\rm sc}
(A_1)\rightarrow-1,\,a_1\rightarrow0,\,d_1\rightarrow0}{R^{\rm E}_{\rm sc}
(A_2)\rightarrow-1,\,a_2\rightarrow0,\,d_2\rightarrow0}}R_1^{\boldsymbol A}
f(\boldsymbol x)=R_1f(\boldsymbol x)
\ \ \text{and}\ \ 
\lim_{\genfrac{}{}{0pt}{}{R^{\rm E}_{\rm sc}
(A_1)\rightarrow-1,\,a_1\rightarrow0,\,d_1\rightarrow0}{R^{\rm E}_{\rm sc}
(A_2)\rightarrow-1,\,a_2\rightarrow0,\,d_2\rightarrow0}}R_2^{\boldsymbol A}
f(\boldsymbol x)=R_2f(\boldsymbol x).$$
\item[\rm{(iii)}]
For  any $p\in (1, \infty)$ and $f \in L^p(\mathbb{R}^2)$, one has
$$\lim_{\genfrac{}{}{0pt}{}{R^{\rm E}_{\rm sc}
(A_1)\rightarrow-1,\,a_1\rightarrow0,\,d_1\rightarrow0}{R^{\rm E}_{\rm sc}
(A_2)\rightarrow-1,\,a_2\rightarrow0,\,d_2\rightarrow0}}
 \|R_1^{\boldsymbol{A}}f - R_1f\|_{L^p(\mathbb{R}^2)} = 0
 \ \ \text{and} \ \ \lim_{\genfrac{}{}{0pt}{}{R^{\rm E}_{\rm sc}
 (A_1)\rightarrow-1,\,a_1\rightarrow0,\,d_1\rightarrow0}{R^{\rm E}_{\rm sc}
 (A_2)\rightarrow-1,\,a_2\rightarrow0,\,d_2\rightarrow0}} \|R_1^{\boldsymbol{A}}
 f - R_1f\|_{L^p(\mathbb{R}^2)} = 0.$$
 \item[\rm{(iv)}] 
For  any $p\in (1, \infty)$ and $f \in L^p(\mathbb{R}^2)$, one has
$$\lim_{\genfrac{}{}{0pt}{}{R^{\rm E}_{\rm sc}
(A_1)\rightarrow-1,\,a_1\rightarrow0,\,d_1\rightarrow0}{R^{\rm E}_{\rm sc}
(A_2)\rightarrow-1,\,a_2\rightarrow0,\,d_2\rightarrow0}}R_1^{\boldsymbol A}
f=R_1f
\ \ \text{and}\ \ 
\lim_{\genfrac{}{}{0pt}{}{R^{\rm E}_{\rm sc}
(A_1)\rightarrow-1,\,a_1\rightarrow0,\,d_1\rightarrow0}{R^{\rm E}_{\rm sc}
(A_2)\rightarrow-1,\,a_2\rightarrow0,\,d_2\rightarrow0}}R_2^{\boldsymbol A}
f=R_2f$$
in $ \mathscr{S}'(\mathbb{R}^2)$.
 \end{itemize}
\end{theorem}

\begin{proof}
To show (i), using Definitions \ref{cf}, \ref{Rieszdef1}, and \ref{is}, 
we derive \eqref{eq51} and \eqref{eq52}, which completes the proof of (i).

In order to prove (ii), by symmetry we only need to consider 
the case $R_1^{\boldsymbol A}$. To this end,  fix $f\in\mathscr{S}(\mathbb{R}^2)$ and 
$\boldsymbol x\in\mathbb R^2$. We obviously have
\begin{align*}
&\left|{\rm p}.{\rm v}.\int_{\mathbb{R}^2} \frac{({x_1}-{y_1})
f(\boldsymbol y) e^{i\sum^{2}_{j=1}\frac{a_jy_j^2}
{R^{\rm E}_{\rm sc}(A_j)c_j}}}{| \boldsymbol x-\boldsymbol y
|^{3}}\,d\boldsymbol{y}-{\rm p}.{\rm v}.\int_{\mathbb{R}^2} \frac{({x_1}-{y_1})
f(\boldsymbol y) }{|\boldsymbol x-\boldsymbol y
|^{3}}\,d\boldsymbol{y}\right|\\
&\quad\leq\left|{\rm p}.{\rm v}.\int_{|\boldsymbol x-\boldsymbol y|\leq 1} \frac{({x_1}-{y_1})
f(\boldsymbol y)[e^{i\sum^{2}_{j=1}\frac{a_jy_j^2}
{R^{\rm E}_{\rm sc}(A_j)c_j}}-1]}{| \boldsymbol x-\boldsymbol y
|^{3}}\,d\boldsymbol{y}\right|+\left|\int_{|\boldsymbol x-\boldsymbol y|>1} 
\cdots\,d\boldsymbol{y}\right|\\
&\quad=:H_1+H_2.
\end{align*}

We first estimate $H_2$. For this purpose, using the mean value theorem, we obtain, 
for any $\boldsymbol y\in\mathbb{R}^2$,
\begin{equation}\label{cpzzdl}
\left|e^{i\sum^{2}_{j=1}\frac{a_jy_j^2}
{R^{\rm E}_{\rm sc}(A_j)c_j}}-1\right|\leq 
\left|\sum^{2}_{j=1}\frac{2a_j\theta y_j}
{R^{\rm E}_{\rm sc}(A_j)c_j}\right|\left| e^{i\sum^{2}_{j=1}
\frac{a_j(\theta y_j)^2}
{R^{\rm E}_{\rm sc}(A_j)c_j}}\right|\left| \boldsymbol y\right|\leq
\left[\sum^{2}_{j=1}\frac{2|a_j||y_j|}
{|R^{\rm E}_{\rm sc}(A_j)||c_j|}\right]\left| \boldsymbol y\right|,
\end{equation}
where $\theta\in [0, 1]$.
Substituting this bound into the expression for $H_2$, we find that
\begin{align*}
0\leq H_2\leq&\sum^{2}_{j=1}\int_{|\boldsymbol x-\boldsymbol y|>1}
\frac{|f(\boldsymbol y)2a_jy_j|| \boldsymbol y|}{| \boldsymbol x-\boldsymbol y
|^{2}|R^{\rm E}_{\rm sc}(A_j)||c_j|}\,d\boldsymbol{y}\\
\leq&\sum^{2}_{j=1}\frac{2|a_j|}
{|R^{\rm E}_{\rm sc}(A_j)||c_j|}\int_{\mathbb R^2}
\left|f(\boldsymbol y)\right||y_j|\left| \boldsymbol y\right|\,d\boldsymbol{y}\to 0
\end{align*}
as $R^{\rm E}_{\rm sc}(A_j) \to -1$ and $a_j \to 0$ for any
$j\in\{1,2\}$,  which further implies that $H_2 \to 0$ 
as $R^{\rm E}_{\rm sc}(A_j) \to -1$ and $a_j \to 0$ 
for any $j\in\{1,2\}$.

Then we estimate  $H_1$. Again, by the mean value 
theorem and \eqref{cpzzdl}, we conclude that, 
for any $\boldsymbol y\in\mathbb{R}^2$ satisfying 
$|\boldsymbol x-\boldsymbol y|\leq1$,
\begin{align*}
&\left|f(\boldsymbol y) \left[e^{i\sum^{2}_{j=1}\frac{a_jy_j^2}
{R^{\rm E}_{\rm sc}(A_j)c_j}}-1\right]-f(\boldsymbol x) 
\left[e^{i\sum^{2}_{j=1}\frac{a_jx_j^2}
{R^{\rm E}_{\rm sc}(A_j)c_j}}-1\right]\right|\\
&\quad=\left|\left[f(\boldsymbol y)- f(\boldsymbol x)\right]
\left[e^{i\sum^{2}_{j=1}\frac{a_jy_j^2}
{R^{\rm E}_{\rm sc}(A_j)c_j}}-1\right]+f(\boldsymbol x) 
\left[e^{i\sum^{2}_{j=1}\frac{a_jy_j^2}
{R^{\rm E}_{\rm sc}(A_j)c_j}}-e^{i\sum^{2}_{j=1}\frac{a_jx_j^2}
{R^{\rm E}_{\rm sc}(A_j)c_j}}\right]\right|\\
&\quad\leq|\boldsymbol y-\boldsymbol x|\left\{\left|\nabla
 f\left[\theta_1 \boldsymbol y+(1-\theta_1)\boldsymbol x \right]\right|
 \left[\sum^{2}_{j=1}\frac{2|a_j||y_j|}
 {|R^{\rm E}_{\rm sc}(A_j)||c_j|}\right]\left| \boldsymbol y\right|\right.\\
 &\left.\qquad+
 \left| f\left(\boldsymbol x\right)\right|
 \left|\sum^{2}_{j=1}\frac{2a_j[\theta_2 y_j+(1-\theta_2) x_j]}
{R^{\rm E}_{\rm sc}(A_j)c_j}\right|\left|e^{i\sum^{2}_{j=1}
 \frac{a_j[\theta_2 y_j+(1-\theta_2) x_j]^2}{R^{\rm E}_{\rm sc}(A_j)c_j}}\right|\right\}\\
&\quad\leq\left\{\|\, |\nabla f|\,\|_{L^\infty(\mathbb R^2)}
(1+|\boldsymbol x|)^2+|f(\boldsymbol x)|(1+2|\boldsymbol x|)
\right\}|\boldsymbol y-\boldsymbol x|\sum^{2}_{j=1}\frac{2|a_j|}
 {|R^{\rm E}_{\rm sc}(A_j)||c_j|}\\
&\quad=:\,C_{(f, \boldsymbol x)}|\boldsymbol y-\boldsymbol x|\sum^{2}_{j=1}\frac{2|a_j|}
{|R^{\rm E}_{\rm sc}(A_j)||c_j|},
\end{align*}
where $\theta_1,\theta_2\in [0, 1]$.
By the definitions of both $H_1$ and the principle integral, together with 
the vanishing moment of  $x_1-y_1$ in the kernel of $H_1$ and
the last estimate, we find that 
\begin{align*}
0\leq H_1=&\left|\int_{|\boldsymbol x-\boldsymbol y|\le1} \frac{({x_1}-{y_1})
\{f(\boldsymbol y) [e^{i\sum^{2}_{j=1}\frac{a_jy_j^2}
{R^{\rm E}_{\rm sc}(A_j)c_j}}-1]-f(\boldsymbol x) [e^{i\sum^{2}_{j=1}\frac{a_jx_j^2}
{R^{\rm E}_{\rm sc}(A_j)c_j}}-1]\}}{|\boldsymbol x-\boldsymbol y
|^{3}}\,d\boldsymbol{y}\right|\\
\leq&C_{(f, \boldsymbol x)}\sum^{2}_{j=1}\frac{2|a_j|}
{|R^{\rm E}_{\rm sc}(A_j)||c_j|}
 \int_{|\boldsymbol x-\boldsymbol y|\le1} \frac{1}
 {| \boldsymbol x-\boldsymbol y|}\,d\boldsymbol{y}\to 0
\end{align*}
as $R^{\rm E}_{\rm sc}(A_j) \to -1$ and $a_j \to 0$ for 
$j\in\{1,2\}$,  which further implies that $H_1 \to 0$ as 
$R^{\rm E}_{\rm sc}(A_j) \to -1$ and $a_j \to 0$ for any $j\in\{1,2\}$.
Combining the estimates for $H_1$ and $H_2$, we finally conclude that
 $$\lim_{\genfrac{}{}{0pt}{}{R^{\rm E}_{\rm sc}
(A_1)\rightarrow-1,\,a_1\rightarrow0}{R^{\rm E}_{\rm sc}
(A_2)\rightarrow-1,\,a_2\rightarrow0}}
\left|{\rm p}.{\rm v}.\int_{\mathbb{R}^2} \frac{({x_1}-{y_1})
f(\boldsymbol y) e^{i\sum^{2}_{j=1}\frac{a_jy_j^2}
{R^{\rm E}_{\rm sc}(A_j)c_j}}}{| \boldsymbol x-\boldsymbol y
|^{3}}\,d\boldsymbol{y}-{\rm p}.{\rm v}.\int_{\mathbb{R}^2} 
\frac{({x_1}-{y_1})
f(\boldsymbol y) }{|\boldsymbol x-\boldsymbol y
|^{3}}\,d\boldsymbol{y}\right|=0.$$
Equivalently,
 $$\lim_{\genfrac{}{}{0pt}{}{R^{\rm E}_{\rm sc}
(A_1)\rightarrow-1,\,a_1\rightarrow0}{R^{\rm E}_{\rm sc}
(A_2)\rightarrow-1,\,a_2\rightarrow0}}{\rm p}.{\rm v}.\int_{\mathbb{R}^2} \frac{({x_1}-{y_1})
f(\boldsymbol y) e^{i\sum^{2}_{j=1}\frac{a_jy_j^2}
{R^{\rm E}_{\rm sc}(A_j)c_j}}}{| \boldsymbol x-\boldsymbol y
|^{3}}\,d\boldsymbol{y}={\rm p}.{\rm v}.\int_{\mathbb{R}^2} \frac{({x_1}-{y_1})
f(\boldsymbol y) }{| \boldsymbol x-\boldsymbol y|^{3}}\,d\boldsymbol{y}.$$
This and the continuity of exponential functions further imply that  
$$\lim_{\genfrac{}{}{0pt}{}{R^{\rm E}_{\rm sc}
(A_1)\rightarrow-1,\,a_1\rightarrow0,\,d_1\rightarrow0}{R^{\rm E}_{\rm sc}
(A_2)\rightarrow-1,\,a_2\rightarrow0,\,d_2\rightarrow0}}R_1^{\boldsymbol A}
f(\boldsymbol x)=R_1f(\boldsymbol x).$$
This finishes the proof of (ii).
	
To show (iii),  by the symmetry again, we only consider $R_1^{\boldsymbol{A}}$.
For this purpose, we first assume $f\in\mathscr{S}(\mathbb R^2)$. Then, by the Minkowski norm inequality of $L^p(\mathbb{R}^2)$, for any  $f\in\mathscr{S}(\mathbb R^2)$ we obviously have 
\begin{align*}
& \|R_1^{\boldsymbol{A}}f - R_1f\|_{L^p(\mathbb{R}^2)}\\
&\quad=\left\|R_1^{\boldsymbol{A}}f - e^{-i\sum^{2}_{j=1}
\frac{d_j (\cdot_j)^2}{R^{\rm E}_{\rm sc}(A_j)c_j}}R_1f\right\|_{L^p(\mathbb{R}^2)}
+\left\|e^{-i\sum^{2}_{j=1}	\frac{d_j(\cdot_j)^2}
{R^{\rm E}_{\rm sc}(A_j)c_j}}R_1f-R_1f\right\|_{L^p(\mathbb{R}^2)}
=:\mathrm{I}_1+\mathrm{I}_2.
\end{align*} 
For ${\rm I}_1$, from (i), the  boundedness of the classical
 Riesz transform on  $L^p(\mathbb{R}^n)$ (see, for example, 
 \cite[Corollary 5.2.8]{g20141}) and the Lebesgue dominated 
 convergence theorem, we infer that
 \begin{align*}
 0\leq {\rm I}_1\leq\left\|R_1\left\{
\left[ e^{-i\sum^{2}_{j=1}\frac{a_j (\cdot_j)^2}{R^{\rm E}
_{\rm sc}(A_j)c_j}}-1\right]f\right\}\right\|_{L^p(\mathbb{R}^2)}
\lesssim\left\|
\left[e^{i\sum^{2}_{j=1}\frac{a_j (\cdot_j)^2}{R^{\rm E}
_{\rm sc}(A_j)c_j}}-1\right]f\right\|_{L^p(\mathbb{R}^2)}\to 0
 \end{align*}
as $R^{\rm E}_{\rm sc}(A_1), R^{\rm E}_{\rm sc}(A_2)\to -1$ and
$a_1,a_2\to 0$, which further implies that 
 $$\lim_{\genfrac{}{}{0pt}{}{R^{\rm E}_{\rm sc}
(A_1)\rightarrow-1,\,a_1\rightarrow0}{R^{\rm E}_{\rm sc}
(A_2)\rightarrow-1,\,a_2\rightarrow0}}{\rm I}_1=0.$$
For ${\rm I}_2$, by $R_1f\in L^p(\mathbb{R}^2)$
and the Lebesgue dominated convergence theorem, we obtain
 \begin{align*}
0\leq {\rm I}_2
\leq\left\|
\left[e^{-i\sum^{2}_{j=1}\frac{d_j (\cdot_j)^2}{R^{\rm E}
_{\rm sc}(A_j)c_j}}-1\right]R_1f\right\|_{L^p(\mathbb{R}^2)}\to 0
 \end{align*}
as $R^{\rm E}_{\rm sc}(A_1), R^{\rm E}_{\rm sc}(A_2)\to -1$ and
$d_1,d_2\to 0$, which further implies that 
 $$\lim_{\genfrac{}{}{0pt}{}{R^{\rm E}_{\rm sc}
(A_1)\rightarrow-1,\,d_1\rightarrow0}{R^{\rm E}_{\rm sc}
(A_2)\rightarrow-1,\,d_2\rightarrow0}}{\rm I}_2=0.$$
Combining the estimates for ${\rm I}_1$ and ${\rm I}_2$, we conclude that
\begin{equation}\label{5.4}
\lim_{\genfrac{}{}{0pt}{}{R^{\rm E}_{\rm sc}
(A_1)\rightarrow-1,\,a_1\rightarrow0,\,d_1\rightarrow0}{R^{\rm E}_{\rm sc}
(A_2)\rightarrow-1,\,a_2\rightarrow0,\,d_2\rightarrow0}} \|R_1^{\boldsymbol{A}}
f - R_1f\|_{L^p(\mathbb{R}^2)} = 0.
\end{equation}

Now, let $p\in (1,\infty)$ and $f\in L^p(\mathbb{R}^n)$. Since $\mathscr{S}( \mathbb{R}^n)$ 
is dense in $L^p(\mathbb{R}^n)$  (see, for example, \cite[p. 20]{sw1971}), 
we deduce that, for any $\epsilon\in (0,\infty)$, there exist $g \in \mathscr{S}( \mathbb{R}^n)$  
such that $\|f-g\|_{L^p(\mathbb{R}^2)}<\epsilon$. 
For this $g$, using \eqref{5.4}, we find that there exists $\delta\in (0,\infty)$ such that,
when $|R^{\rm E}_{\rm sc}(A_i)+1|<\delta$, $|a_i|<\delta$,
and $|d_i|<\delta$ for any $i\in\{1,2\}$,    $\|R_1^{\boldsymbol{A}}
g - R_1g\|_{L^p(\mathbb{R}^2)}<\epsilon$. Then, from the 
Minkowski norm inequality of $L^p(\mathbb{R}^2)$,   
the boundedness of the LCRT on  $L^p(\mathbb{R}^n)$ in Theorem \ref{Lp}, 
and the  boundedness of the classical Riesz transform on  $L^p(\mathbb{R}^n)$ 
(see, for example, \cite[Corollary 5.2.8]{g20141}), it follows that, 
when $|R^{\rm E}_{\rm sc}(A_i)+1|<\delta$, $|a_i|<\delta$,
and $|d_i|<\delta$ for any $i\in\{1,2\}$,
\begin{align*}
\left\|R_1^{\boldsymbol{A}}f - R_1f\right\|_{L^p(\mathbb{R}^2)}
&=\left\|R_1^{\boldsymbol{A}}f-R_1^{\boldsymbol{A}}
g+R_1^{\boldsymbol{A}}g+R_1g-R_1g - R_1f\right\|_{L^p(\mathbb{R}^2)}\\
&\leq \left\|R_1^{\boldsymbol{A}}(f-g)\right\|_{L^p(\mathbb{R}^2)}
+\left\|R_1^{\boldsymbol{A}}g-R_1g\right\|_{L^p(\mathbb{R}^2)}
+\left\|R_1(g-f)\right\|_{L^p(\mathbb{R}^2)}\\
&\lesssim \left\|f-g\right\|_{L^p(\mathbb{R}^2)}
+\left\|R_1^{\boldsymbol{A}}g-R_1g\right\|_{L^p(\mathbb{R}^2)}\lesssim \epsilon,
\end{align*} 
which further implies that 
$$\lim_{\genfrac{}{}{0pt}{}{R^{\rm E}_{\rm sc}
(A_1)\rightarrow-1,\,a_1\rightarrow0,\,d_1\rightarrow0}{R^{\rm E}_{\rm sc}
(A_2)\rightarrow-1,\,a_2\rightarrow0,\,d_2\rightarrow0}} \|R_1^{\boldsymbol{A}}
f - R_1f\|_{L^p(\mathbb{R}^2)} = 0,$$
which completes the proof of (iii).

To prove (iv), by symmetry we also only consider 
the case $R_1^{\boldsymbol A}$. For any 
$p\in (1, \infty)$,  $f \in L^p(\mathbb{R}^2)$, 
and $\phi \in \mathscr{S}(\mathbb{R}^2)$, 
we need to show that
$$\lim_{\genfrac{}{}{0pt}{}{R^{\rm E}_{\rm sc}
(A_1)\rightarrow-1,\,a_1\rightarrow0,\,d_1\rightarrow0}{R^{\rm E}_{\rm sc}
(A_2)\rightarrow-1,\,a_2\rightarrow0,\,d_2\rightarrow0}}
\langle R_1^{\boldsymbol A}f - R_1f, \phi \rangle 
= \lim_{\genfrac{}{}{0pt}{}{R^{\rm E}_{\rm sc}
(A_1)\rightarrow-1,\,a_1\rightarrow0,\,d_1\rightarrow0}{R^{\rm E}_{\rm sc}
(A_2)\rightarrow-1,\,a_2\rightarrow0,\,d_2\rightarrow0}}\int_{\mathbb{R}^2} 
\left[R_1^{\boldsymbol A}f(\boldsymbol{x})
- R_1f(\boldsymbol{x}) \right] \phi(\boldsymbol{x}) \, d\boldsymbol{x} = 0.$$
To this end, let $q$ be the conjugate exponent of $p$, that is, $1/p + 1/q = 1$. Since $\phi \in \mathscr{S}(\mathbb{R}^2) 
  \subset L^q(\mathbb{R}^2)$, by H\"{o}lder's inequality, we obtain
\begin{equation}\label{5.5}
\left| \int_{\mathbb{R}^2} \left[R_1^{\boldsymbol A}f(\boldsymbol{x})
- R_1f(\boldsymbol{x}) \right] \phi (\boldsymbol{x})\, 
d\boldsymbol{x} \right| \leq \left\| R_1^{\boldsymbol A}f - 
R_1f \right\|_{L^p(\mathbb{R}^2)} \|\phi\|_{L^q(\mathbb{R}^2)}.
\end{equation}
By (iii), as $ R^{\rm E}_{\rm sc}(A_j) \rightarrow -1$,  $a_j \rightarrow 0$,
and $d_j \rightarrow 0$ for any $ j\in\{1,2\}$, one has 
$\| R_1^{\boldsymbol A}f - R_1f \|_{L^p(\mathbb{R}^2)} \to 0$, 
which, together with \eqref{5.5} further implies the desired conclusion. 
This  finishes the  proof of (iv) and hence Theorem \ref{p-sharp}.
 \end{proof}
\begin{remark}\label{r-chirp}
	\begin{itemize}
		\item[(i)] The convergence in (ii) of Theorem \ref{p-sharp} is pointwise, not uniform. 
To see this, we only need to observe that, for any $\boldsymbol x \in \mathbb{R}^2$, 
as $a \to 0$, $e^{i a x^2}$ converges pointwise to 1. However, this 
convergence is not uniform. Consequently,  in Theorem \ref{p-sharp}(ii), 
when $a_j=0$,  $R^{\rm E}_{\rm sc}(A_j)\rightarrow -1$, and 
$d_j\to 0$ for  $j\in\{1,2\}$,   $R_1^{\boldsymbol A}f$  
does not converge uniformly to $R_1f$.	
\item[(ii)] From (i), we further infer that Theorem \ref{p-sharp}(iii) with $p=\infty$ does not hold. For any  $f\in\mathscr{S}(\mathbb{R}^2)\subset L^\infty(\mathbb{R}^2)$, 
by   \cite[Definition 5.1.13 and Theorem 2.3.20]{g20141}  
and \eqref{eq51}, we conclude that $R_1f,R_1^{\boldsymbol A}f\in C^{\infty}
(\mathbb{R}^2)$, which further implies that 
\begin{equation}\label{eq5.6}
\lim_{\genfrac{}{}{0pt}{}{R^{\rm E}_{\rm sc}
(A_1)\rightarrow-1,\,a_1\rightarrow0,\,d_1\rightarrow0}{R^{\rm E}_{\rm sc}
(A_2)\rightarrow-1,\,a_2\rightarrow0,\,d_2\rightarrow0}}
 \|R_1^{\boldsymbol{A}}f - R_1f\|_{L^\infty(\mathbb{R}^2)}=
 \lim_{\genfrac{}{}{0pt}{}{R^{\rm E}_{\rm sc}
(A_1)\rightarrow-1,\,a_1\rightarrow0,\,d_1\rightarrow0}{R^{\rm E}_{\rm sc}
(A_2)\rightarrow-1,\,a_2\rightarrow0,\,d_2\rightarrow0}}
 \sup_{\boldsymbol x\in \mathbb{R}^2}
 \left|R_1^{\boldsymbol{A}}f - R_1f\right|.
\end{equation}
To make the right-hand side of \eqref{eq5.6} equal  0, it is required that 
$R_1^{\boldsymbol{A}}$ converges uniformly to $R_1$, 
which contradicts (i). Therefore, $p=\infty$ in 
Theorem \ref{p-sharp}(iii) does not hold.

We now show that Theorem \ref{p-sharp}(iii) with $p=1$ also dose not hold.
We only  consider this in the case $n=1$, that is, the LCHT 
$H_A$ does not  converge to the 
Hilbert transform $H$ in $L^1(\mathbb R)$. Let  $A:=\begin{bmatrix}
{a}&{b}\\
{c}&{d}
\end{bmatrix}\in{M_{2\times2}}(\mathbb{R})$
and both $b\ne0$ and $c\ne0$.
When $a=0$,  $R^{\rm E}_{\rm sc}(A)\rightarrow -1$, and 
$d\to 0$,   $H^A$  
does not converge to $H$ in $L^1(\mathbb R)$.  To prove this,
let $f:=\mathbf{1}_{[0,1]}\in L^1(\mathbb R)$.   
For any $x\in\mathbb R$, we have
 $Hf(x)=\frac{1}{\pi}\ln\left|\frac{x}{x-1}\right|$ 
and $\lim_{x\to \infty}x\ln\left|\frac{x}
{x-1}\right|=1$. This implies that there exists 
 $N\in \mathbb{N}$ such that, for any $x>N$,   $x\ln\left|\frac{x}
{x-1}\right|\sim 1$.  Then, by this, we further have
\begin{align*}
\left\|H^Af-Hf\right\|_{L^1(\mathbb R)}&=\left\|e^\frac{{-id(\cdot)^2}}{R^{\rm E}_{\rm sc}(A)c}Hf-Hf\right\|_{L^1(\mathbb R)}=\frac{1}{\pi}\left\|\left[e^\frac{{-id(\cdot)^2}}{R^{\rm E}_{\rm sc}(A)c}-1\right]
\ln\left|\frac{\cdot}{\cdot-1}\right|\right\|_{L^1(\mathbb R)}\\
&\geq\frac{1}{\pi}\int_N^{\infty}|e^\frac{{-idx^2}}{R^{\rm E}_{\rm sc}(A)c}-1|\ln\left|\frac{x}{x-1}\right| 
\,dx\sim\frac{1}{\pi}\int_N^{\infty}\frac{|e^\frac{{-idx^2}}{R^{\rm E}_{\rm sc}(A)c}-1|}{x}
\,dx\\&=\frac{1}{\pi}\int_N^{\infty}\frac{2|\sin(\frac{dx^2}{2R^{\rm E}_{\rm sc}(A)c})
|}{x}\,dx=\frac{1}{\pi}\int_{\frac{dN^2}{2R^{\rm E}_{\rm sc}(A)c}}^{\infty}\frac{|\sin x|}x
\,dx\\&\geq\frac{1}{\pi}\sum_{k=K}^{\infty}\int_{k\pi}^{(k+1)\pi}\frac{|\sin x|}x
\,dx\geq\frac{1}{\pi^2}\sum_{k=K}^{\infty}\frac{1}{k+1}\int_{k\pi}^{(k+1)\pi}|\sin x|
\,dx \\&\geq\frac{1}{\pi^2}\sum_{k=K}^{\infty}\frac{2}{k+1}=\infty,
\end{align*}
where $K:=\lfloor\frac{dN^2}{2R^{\rm E}_{\rm sc}(A)c\pi}\rfloor+1$. 
This implies that Theorem \ref{p-sharp}(iii) fails for $p=1$. 
Thus, the range $p\in (1,\infty)$ of Theorem \ref{p-sharp}(iii) is sharp.
\end{itemize}
\end{remark}
Using the sharpness $R^{\rm E}_{\rm sc}$ of the edge strength and continuity
of  images related to LCRTs introduced  in Definition \ref{is} and also using
Theorem \ref{p-sharp}, we next give a new method of refining image edge detection,
call  the  {\bf LCRT image edge detection method} (for short,
{\bf LCRT-IED method}) as follows. Take Figure \ref{FIG10} as an example.
Graph (a) of Figure \ref{FIG10} is the given original image;
using the classical Riesz transform, we extract the global
edge information from Graph (a) of Figure \ref{FIG10} to
obtain Graph (f) of  Figure \ref{FIG10} which is our target
image [in this case, $R^{\rm E}_{\rm sc}(A_1)=-1=R^{\rm E}
_{\rm sc}(A_2)$]. Via suitably choosing parameter matrices $A_1$ 
or $A_2$ appearing in the LCRTs $\{R_1^{\boldsymbol{A}},R_2^
{\boldsymbol{A}}\}$  with $\boldsymbol{A}:=(A_1,A_2)$ as 
in Definition \ref{Rieszdef1} such that $R^{\rm E}_{\rm sc}(A_1)$ 
or $R^{\rm E}_{\rm sc}(A_2)$ converges increasingly closer to -1, 
we are able to extract the global edge information from Graph (a)
of Figure \ref{FIG10} to obtain Graphs (b) through (e) of Figure 
\ref{FIG10} which gradually approach our target image Graph (f) 
of  Figure \ref{FIG10}. This new method  might have applications in 
large-scale image matching, refinement of image feature extraction, 
and image refinement processing.

Theorem \ref{p-sharp} provides the solid {\bf theoretical  foundation} from harmonic analysis for  the above new ICRT-IED method. To understand this, we only need to observe  that, when $A_1=\begin{bmatrix}{0}\
&{1}\\{-1}\ &{0}\end{bmatrix}=A_2$,  in this case $R^{\rm E}_{\rm sc}
(A_1)=-1=R^{\rm E}_{\rm sc}(A_2)$ and the LCRTs
 $\{R_1^{\boldsymbol{A}},R_2^{\boldsymbol{A}}\}$ reduce back
to classical Riesz transforms.

In Figure \ref{FIG10}, Graph (a) is the
original test image for edge detection;  Graphs (b), (c), 
(d), (e), and (f) correspond to the edge detection images 
via the LCRT-IED method by using,
respectively,  parameter matrices
$M_1:=(m_1,m_2)$ with $m_1:=\begin{bmatrix}{0}\
&{1}\\{-1}\ &{0}\end{bmatrix}$
and $m_2:=\begin{bmatrix}{0}
\ &{1000}\\{-0.001}\ &{0}\end{bmatrix}$,
$M_2:=(m_1,m_3)$ with $m_3:=\begin{bmatrix}{0}
\ &{500}\\{-0.002}\ &{0}\end{bmatrix}$,
$M_3:=(m_1,m_4)$ with $m_4:=\begin{bmatrix}{0}\ &{250}
\\{-0.004}\ &{0}\end{bmatrix}$,  
$M_4:=(m_1,m_5)$ with $m_5:=\begin{bmatrix}{0}\ &{25}\\
{-0.04}\ &{0}\end{bmatrix}$, and
$M_5:=(m_1,m_1)$. In particular, when the parameter matrix is 
$M_5$, in this case the LCRT reduces to the classical Riesz 
transform. Graphs (b), (c), (d), and (e) demonstrate the gradual 
enhancement of  the edge strength 
and continuity in the edge images through the use 
of different parameter matrices of LCRT, which gradually 
approach our target image Graph (f) of Figure  \ref{FIG10}.

\begin{figure}[H]
\centering
\subfigcapskip=-30pt
\subfigure[]{\includegraphics[width=0.365\linewidth]{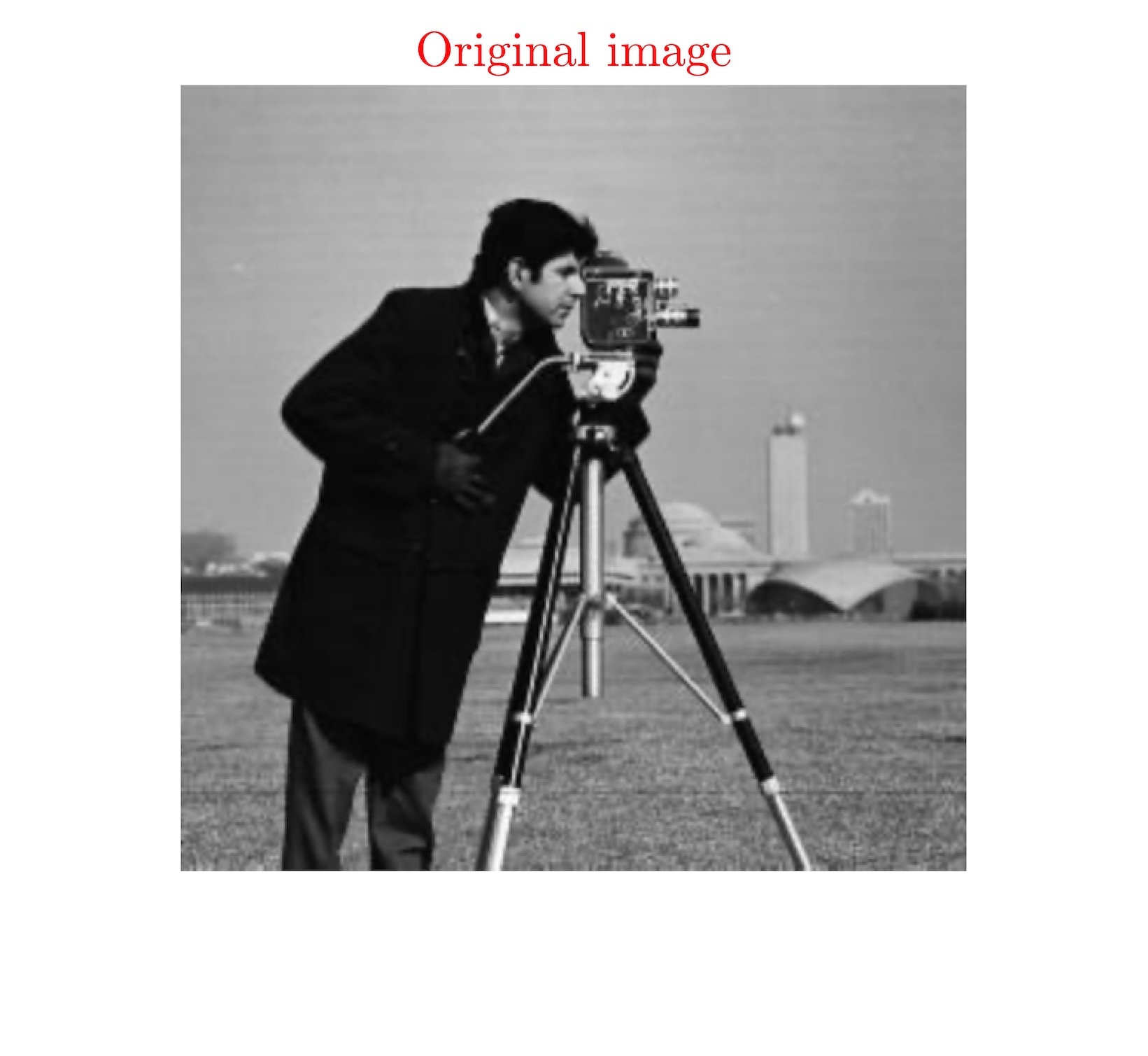}}
\hspace{-1cm}
\subfigure[]{\includegraphics[width=0.365\linewidth]{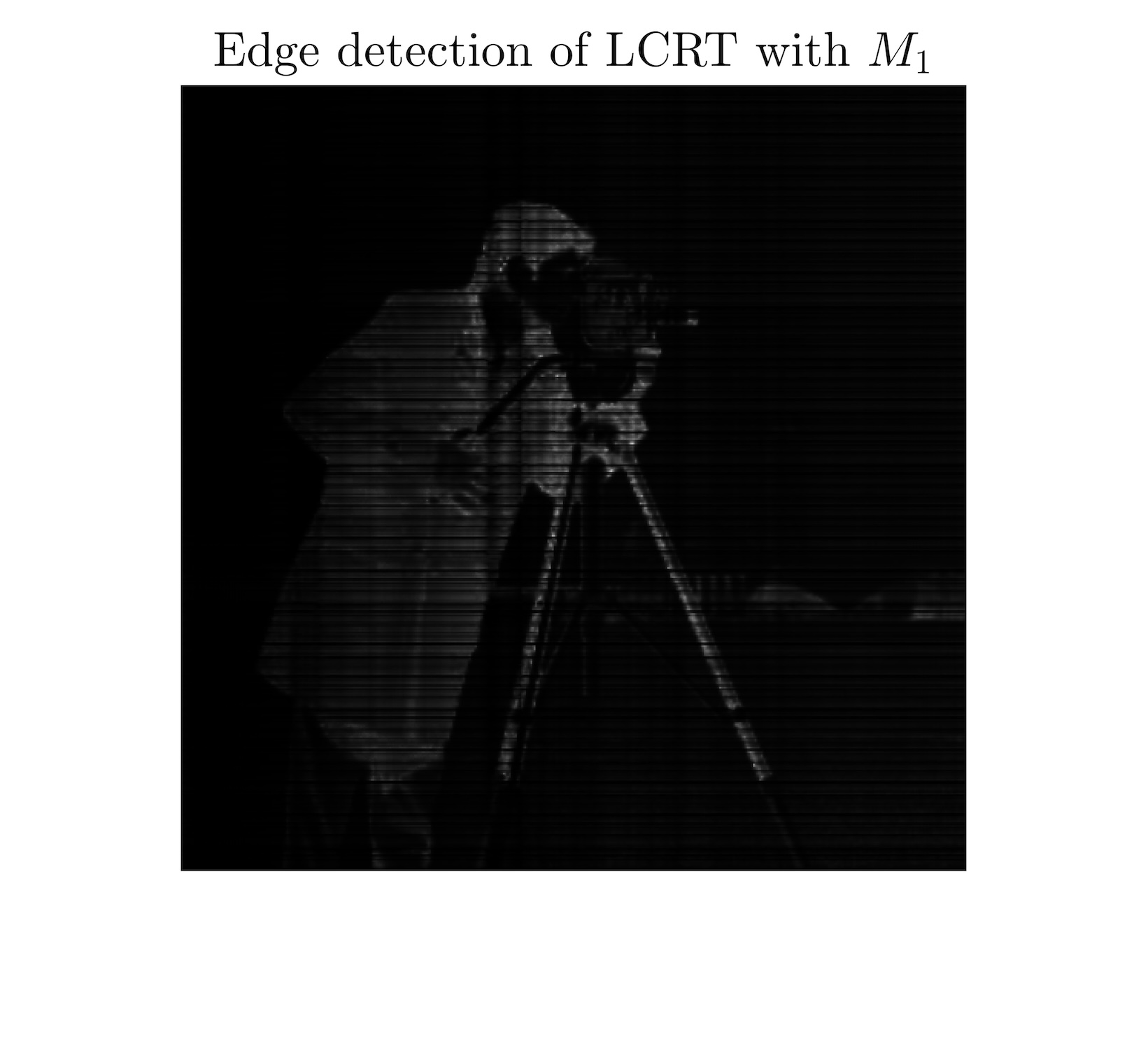}}
\hspace{-1cm}
\subfigure[]{\includegraphics[width=0.365\linewidth]{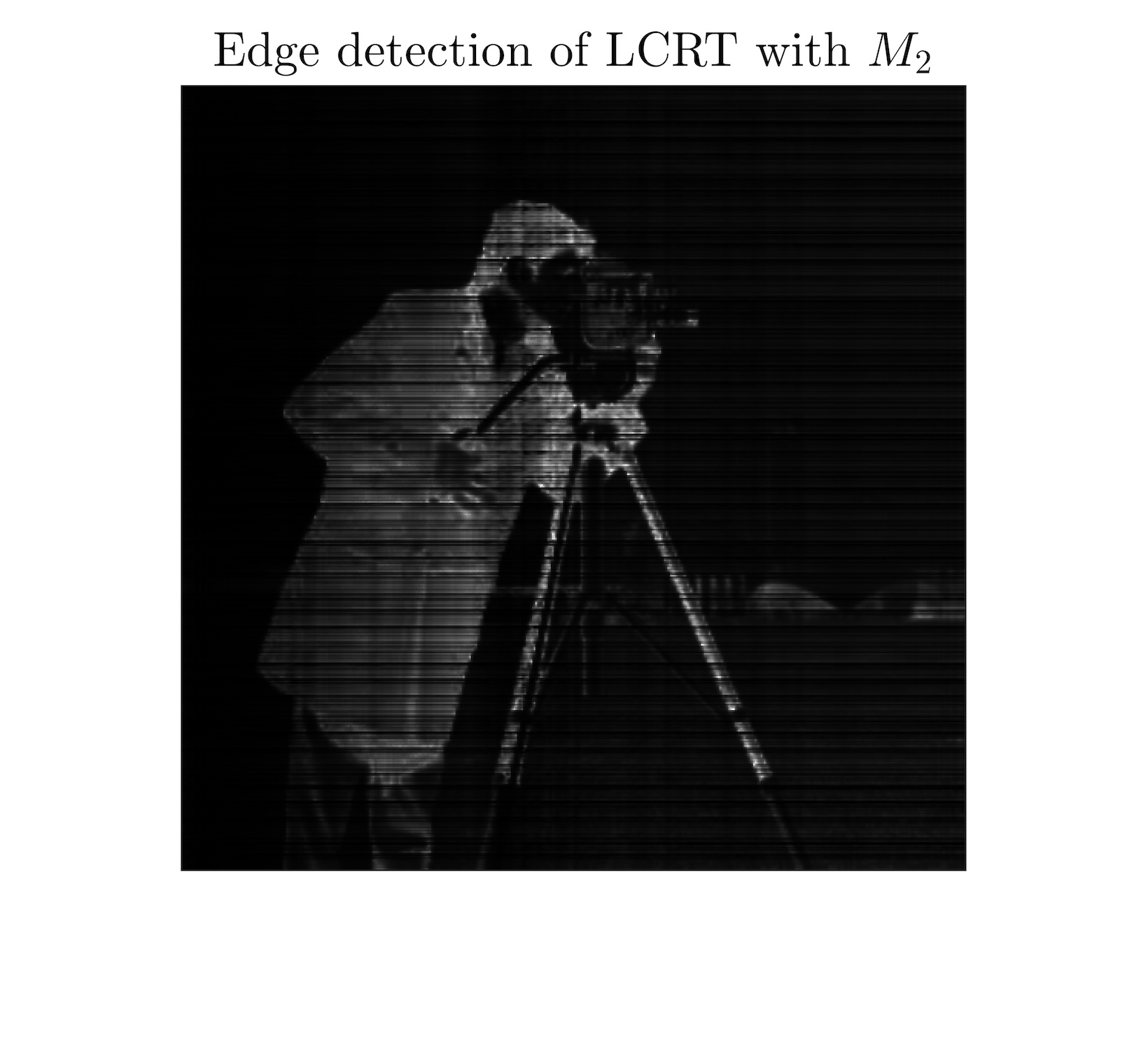}}\\
\vspace{-0.5cm}
\subfigure[]{\includegraphics[width=0.365\linewidth]{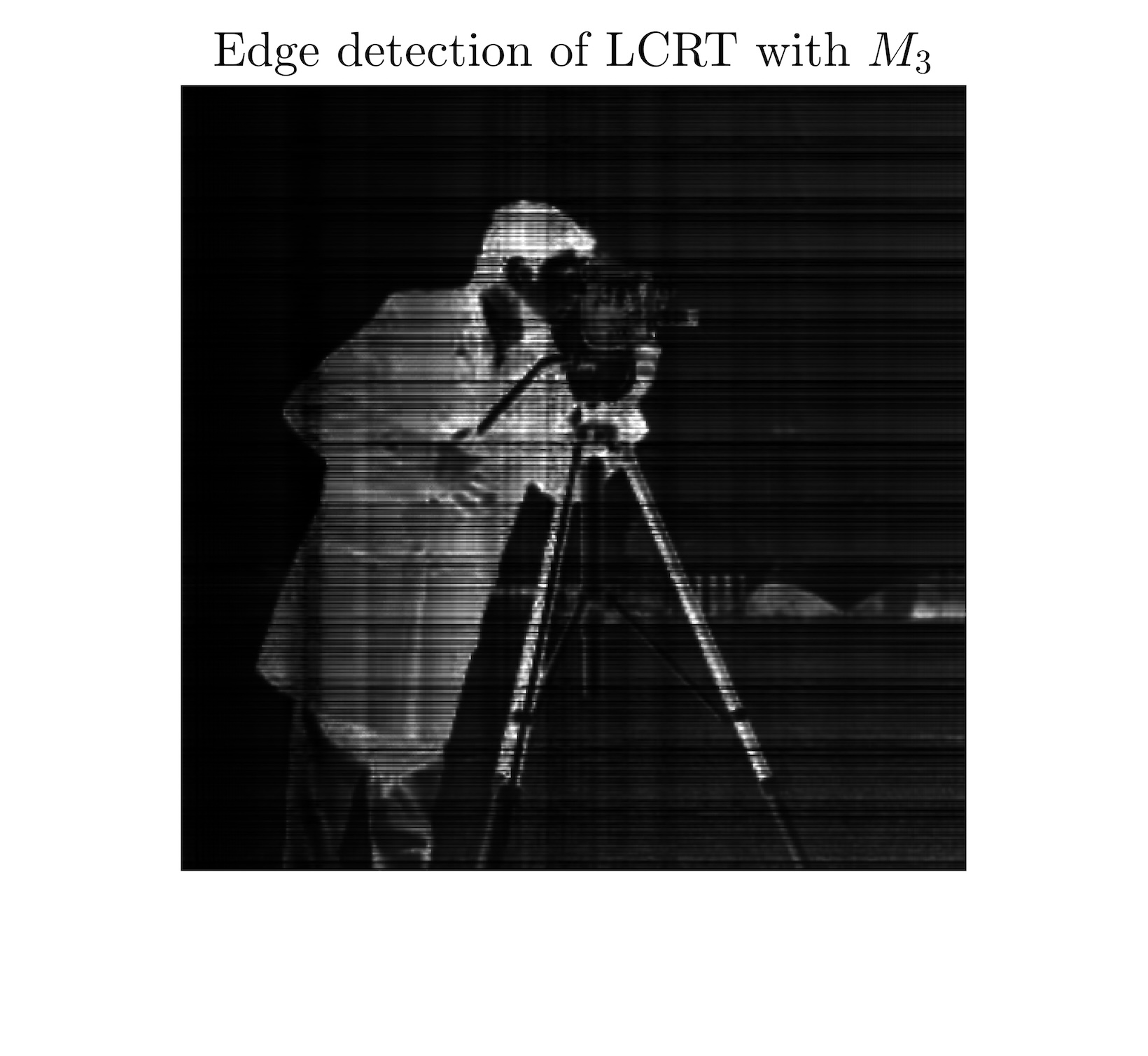}}
\hspace{-1cm}
\subfigure[]{\includegraphics[width=0.365\linewidth]{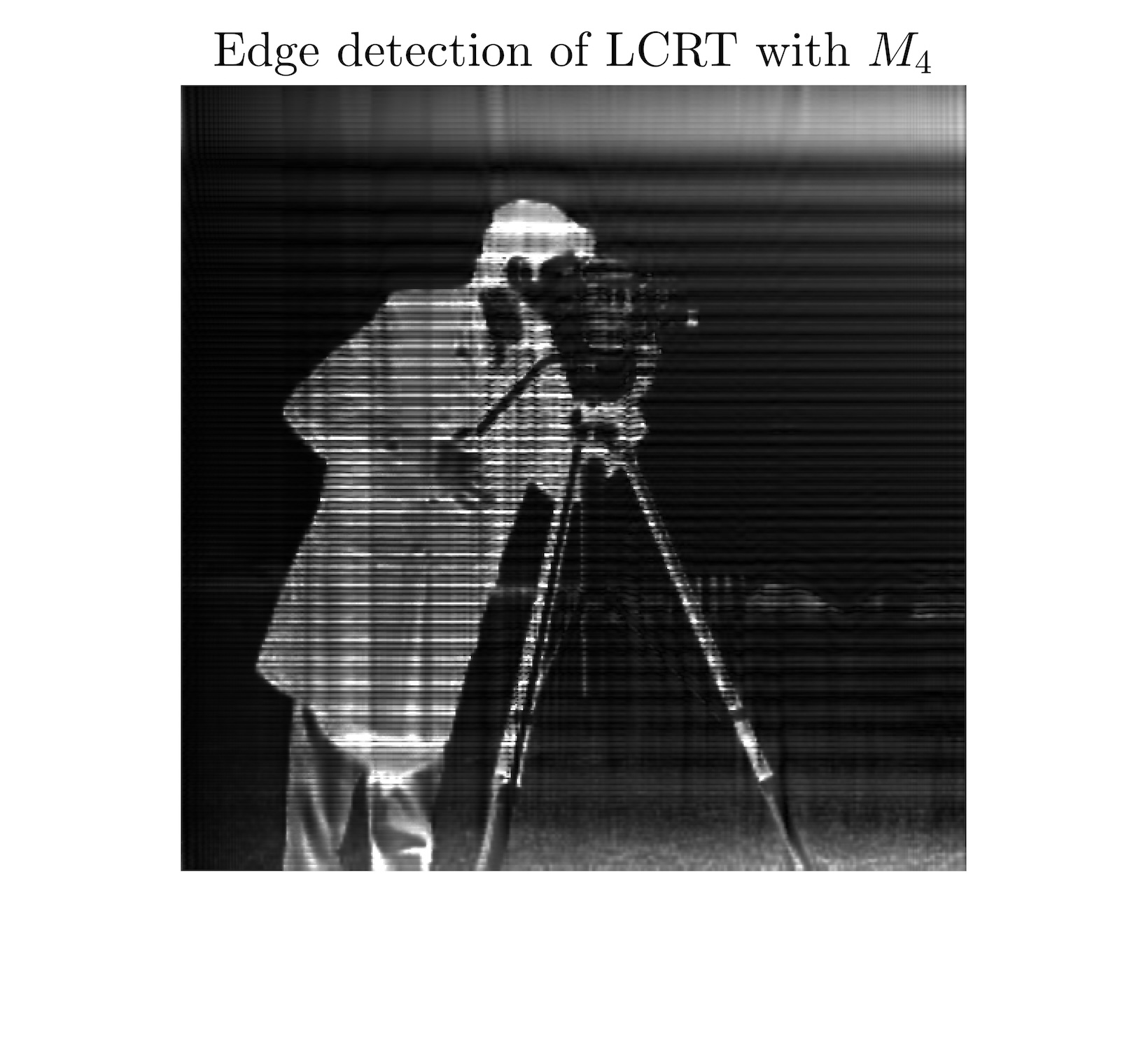}}
\hspace{-1cm}
\subfigure[]{\includegraphics[width=0.365\linewidth]{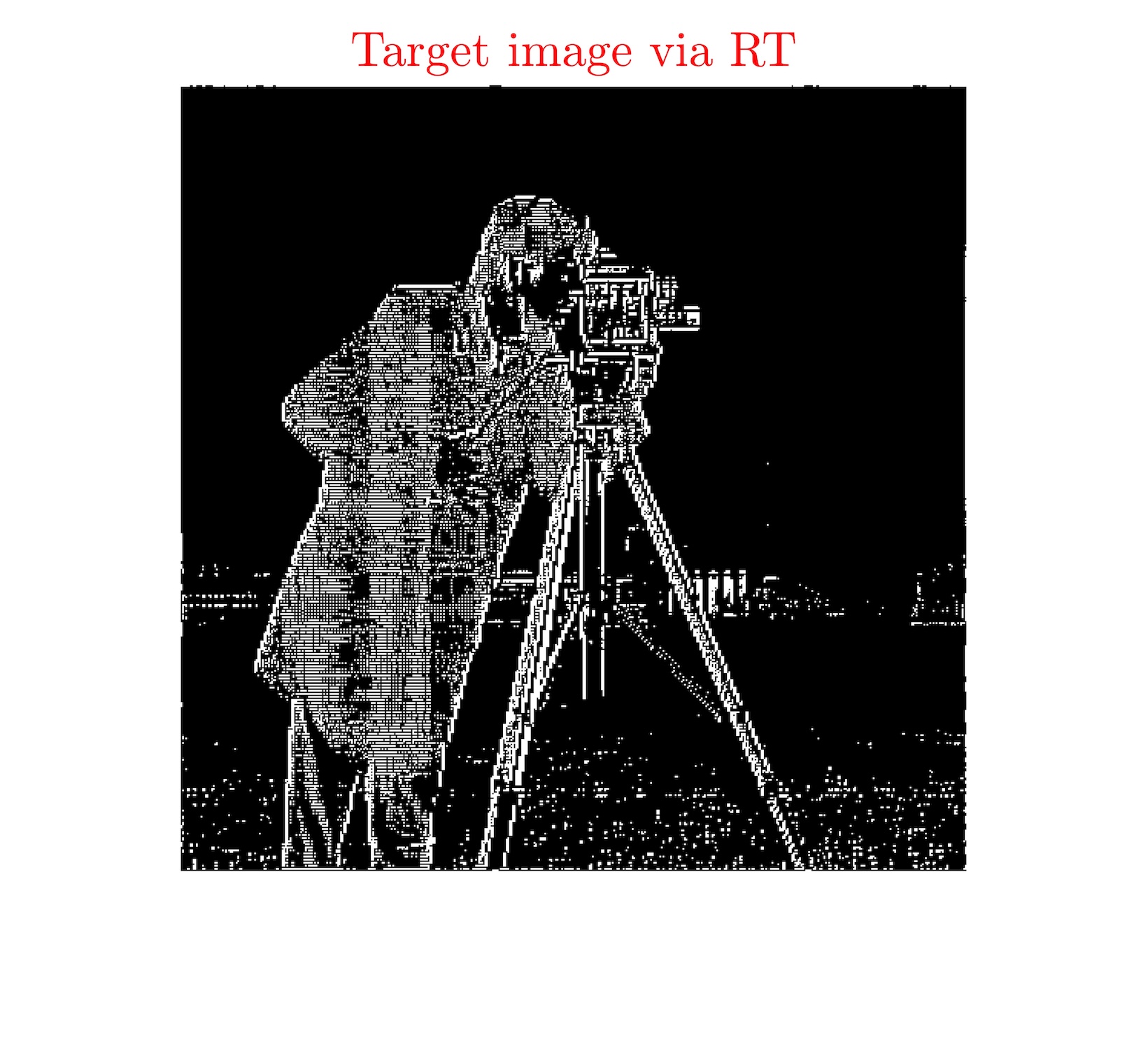}}
\vspace{-0.5cm}
\caption{Edge detection using LCRT-IED method on grayscale 
image Cameraman.}
\label{FIG10}
\end{figure}

\begin{table}[H]
\centering
\begin{tabular}{cccccc}
\toprule[1.5pt]
{\bf Graph} &   (b) &   (c) &   (d) &   (e) &  (f) \\
\midrule[1pt]
$R_{\mathrm{sc}}^{\mathrm{E}}(A_2)$ & -1000000 
& -250000 & -62500 & -625 & -1 \\
\midrule[1pt]
{\bf MSE} & 0.2625 & 0.2568 & 0.2562 & 0.22 & 0.1642 \\
\bottomrule[1.5pt]
\end{tabular}
\caption{$R_{\mathrm{sc}}^{\mathrm{E}}$ and 
MSE corresponding to the edge detection images in Figure 
\ref{FIG10}.}
\label{tab:1}
\end{table}

Table \ref{tab:1} presents $R_{\mathrm{sc}}^
{\mathrm{E}}$ corresponding to the edge detection images in 
Figure \ref{FIG10} and the mean squared error (for short, MSE)  
between the original image and the edge detection images in 
Figure \ref{FIG10}, where $R_{\mathrm{sc}}^{\mathrm{E}}$ 
is the same as in Definition \ref{is} and the MSE is 
the mean squared error of pixel values between two images 
which is used to measure the similarity between them.  
A smaller MSE indicates greater similarity between 
the two images. As shown in Table \ref{tab:1}, when 
$R_{\mathrm{sc}}^{\mathrm{E}}$ approaches $-1$,  the smaller 
MSE values between the original image and the edge detection 
images in Figure \ref{FIG10} imply higher similarity 
between the two images, indicating that the edge strength 
and continuity of the edge image are good.

We now investigate the effectiveness of the aforementioned introduced 
LCRT-IED method for local feature extraction in images.
 In Graph (a) of Figure \ref{FIG10-1}, we divide an image 
 into 9 equally-sized sub-regions. This segmentation aims to assess 
the local similarity between the edge detection images and the original 
image by comparing their MSEs within each sub-region.  
Graphs (b) and (c) in Figure \ref{FIG10-1} display, respectively, 
the MSEs between Graphs (a) and (e) and between Graphs (a) and (f)
of the aforementioned 9 sub-regions in Figure \ref{FIG10}.

\begin{figure}[H]
\centering
\subfigure[Image Segmentation]{\includegraphics[width=0.3\linewidth]{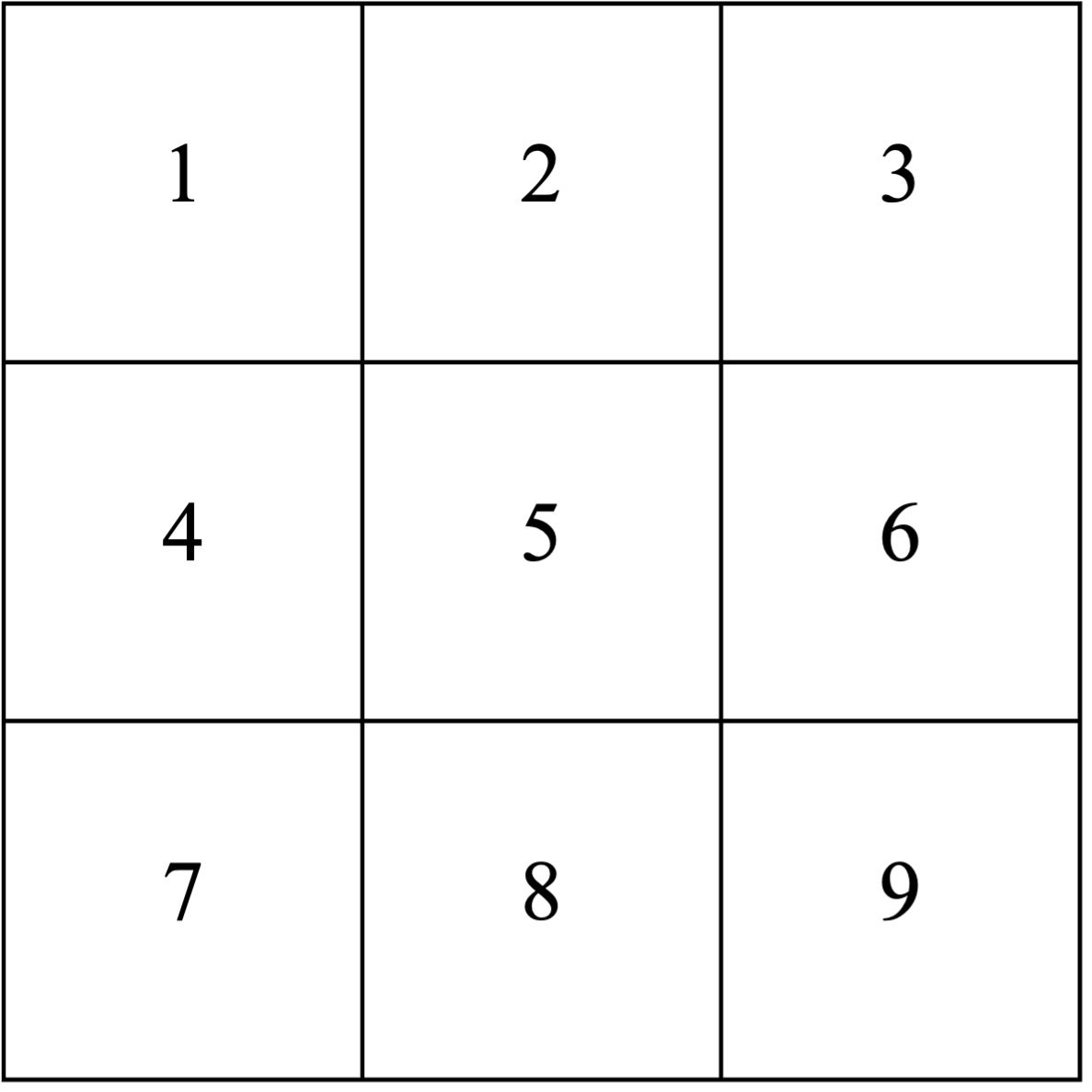}}
\hspace{0.2cm}
\subfigure[MSE of each sub-region between the Graphs (a) and (e) in 
Figure \ref{FIG10}.]{\includegraphics[width=0.3\linewidth]{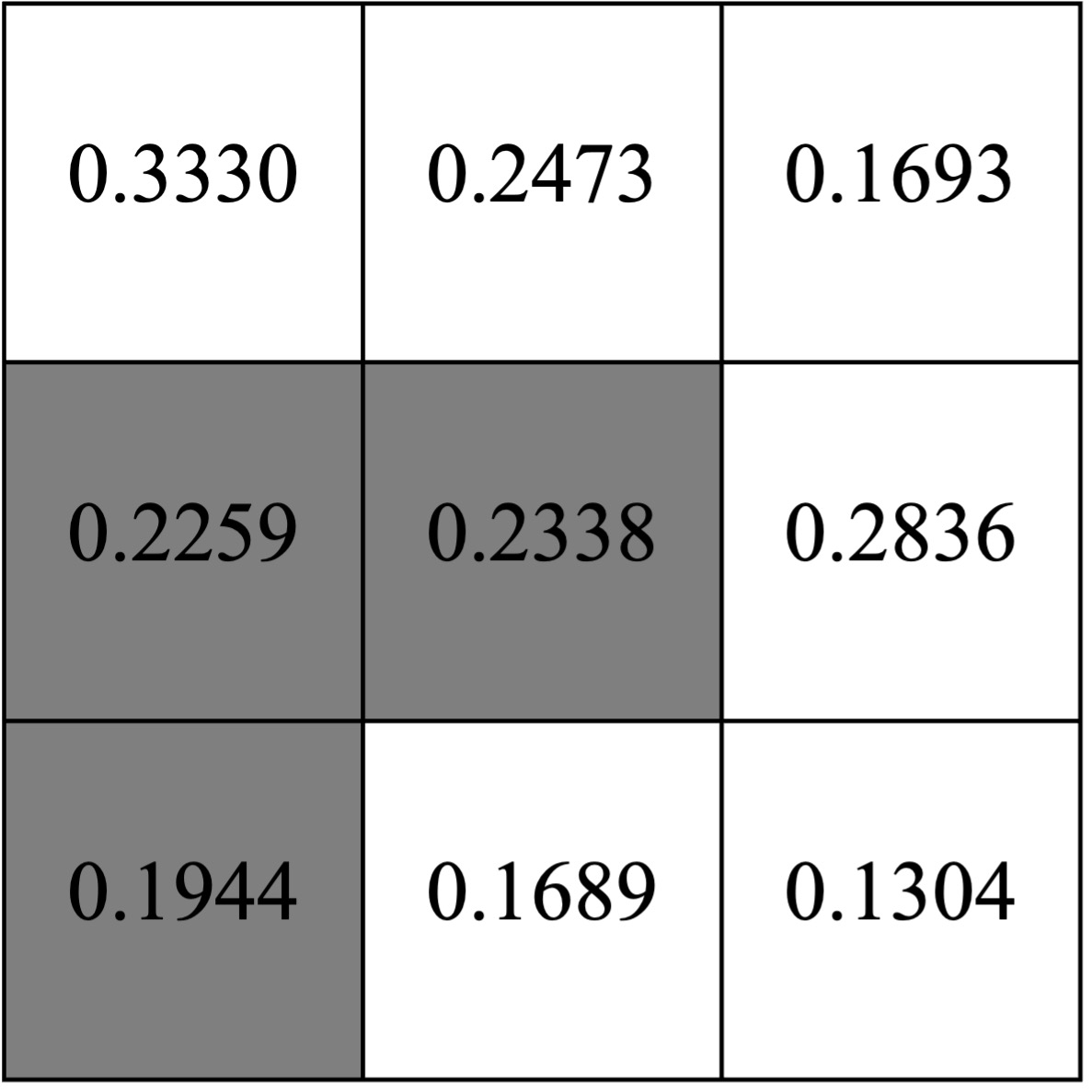}}
\hspace{0.2cm}
\subfigure[MSE of each sub-region between the Graphs (a) and (f) in 
Figure \ref{FIG10}.]{\includegraphics[width=0.3\linewidth]{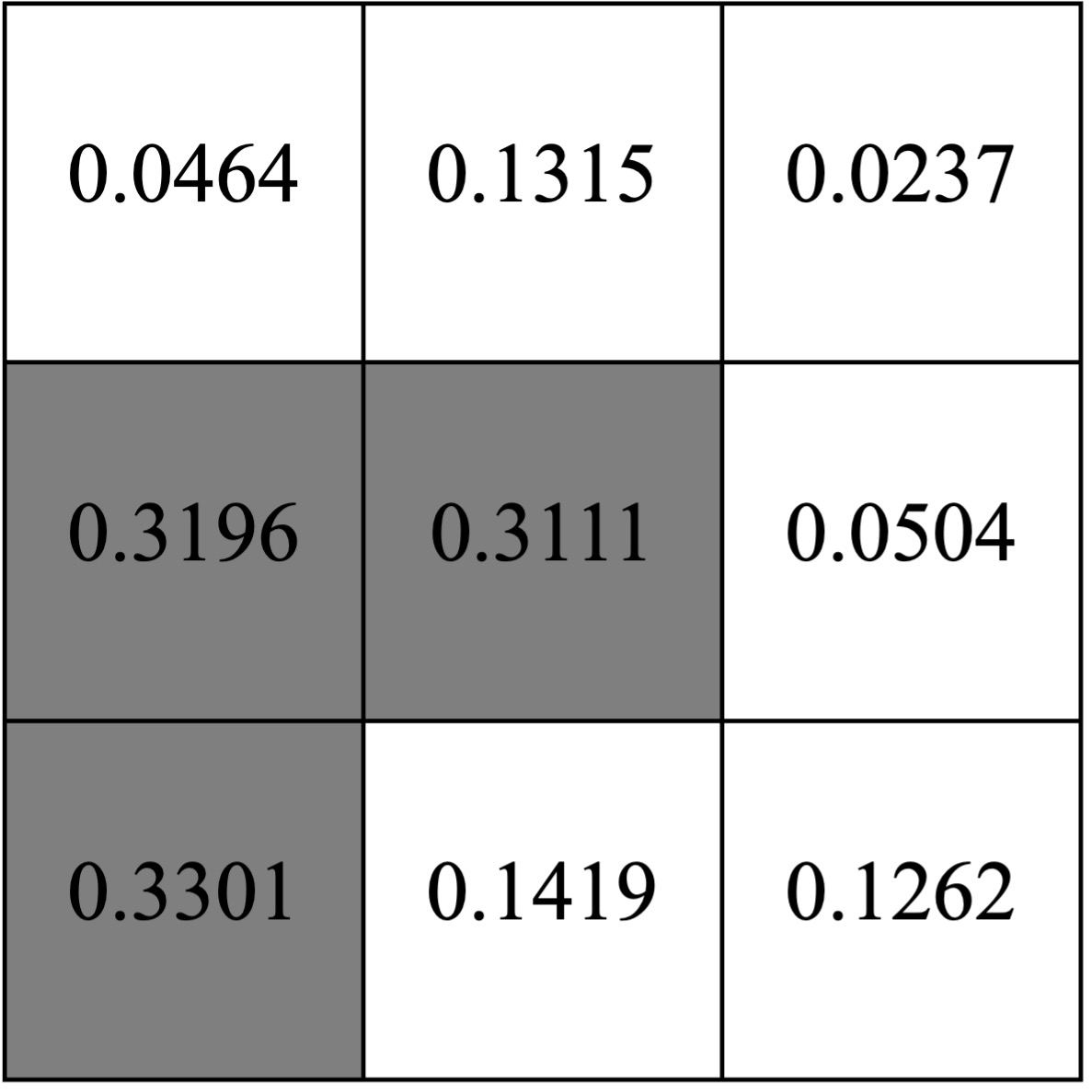}}
\caption{MSE between the original image and the edge detection images
in Figure \ref{FIG10} after image segmentation.}
\label{FIG10-1}
\end{figure}

The analysis of Table \ref{tab:1} reveals that the MSE 
between Graphs (a) and (f) in Figure \ref{FIG10} is significantly 
lower than the one between Graphs (a) and (e) in Figure \ref{FIG10}. 
This suggests that, when the LCRT reduces to the classical Riesz 
transform, it achieves refined global information extraction. Moreover, 
comparing Graphs (b) with (c) in Figure \ref{FIG10-1}, we find
an interesting phenomenon: for some sub-regions of the aforementioned 9 
sub-regions, their MSEs between Graphs (a) and (f) in Figure \ref{FIG10} 
surpasses those corresponding ones between Graphs (a) and (e) in 
Figure \ref{FIG10}. This interesting  phenomenon confirms that the 
LCRT-IED method possesses dual capabilities: while maintaining 
stable global edge detection through parameter matrix refinement, 
it simultaneously excels in local feature extraction, particularly in 
preserving fine-scale image details for some sub-regions.  This is a 
consequence of the convergence in Theorem \ref{p-sharp}(ii), which 
is pointwise rather than uniform 
[see Remark \ref{r-chirp}(i)]. To be precise, while the overall 
edge effect gradually approaches the best edge effect of the 
target image from below,  in some sub-regions the local edge 
effect outperforms the corresponding one of the target image.

To ensure both the reliability of the LCRT-IED method introduced above and 
the generalizability of the conclusions, we now repeat the 
aforementioned experimental procedure on the other grayscale image.

In Figure \ref{FIG11}, Graph (a) is the original test image 
for edge detection;   Graphs (b), (c), (d),  (e), and (f) correspond to the 
edge detection images via the LCRT-IED method by using, respectively, 
parameter matrices
$T_1:=(t_1,t_2)$ with $t_1:=\begin{bmatrix}{0}\
&{1}\\{-1}\ &{0}\end{bmatrix}$
and $t_2:=\begin{bmatrix}{100}
\ &{45}\\{22.2}\ &{100}\end{bmatrix}$,
$T_2:=(t_1,t_3)$ with $t_3:=\begin{bmatrix}{100}
\ &{70}\\{\frac{9999}{70}}\ &{100}\end{bmatrix}$,
$T_3:=(t_1,t_4)$ with $t_4:=\begin{bmatrix}{100}\ &{300}
\\{33.33}\ &{100}\end{bmatrix}$,  
$T_4:=(t_1,t_5)$ with $t_5:=\begin{bmatrix}{100}\ &{730}\\
{14.28}\ &{100}\end{bmatrix}$, and $T_5:=(t_1,t_1)$.
In particular, when the parameter matrix is $T_5$, 
in this case the LCRT reduces to the classical Riesz 
transform. Graphs (b), (c), (d), and (e) demonstrate the 
gradual enhancement of  the edge  strength 
and continuity in the edge images through using
different parameter matrices in the LCRT, which gradually 
approach our target image Graph (f) from below.

\begin{figure}[H]
\centering
\subfigcapskip=-30pt
\subfigure[]{\includegraphics[width=0.365\linewidth]{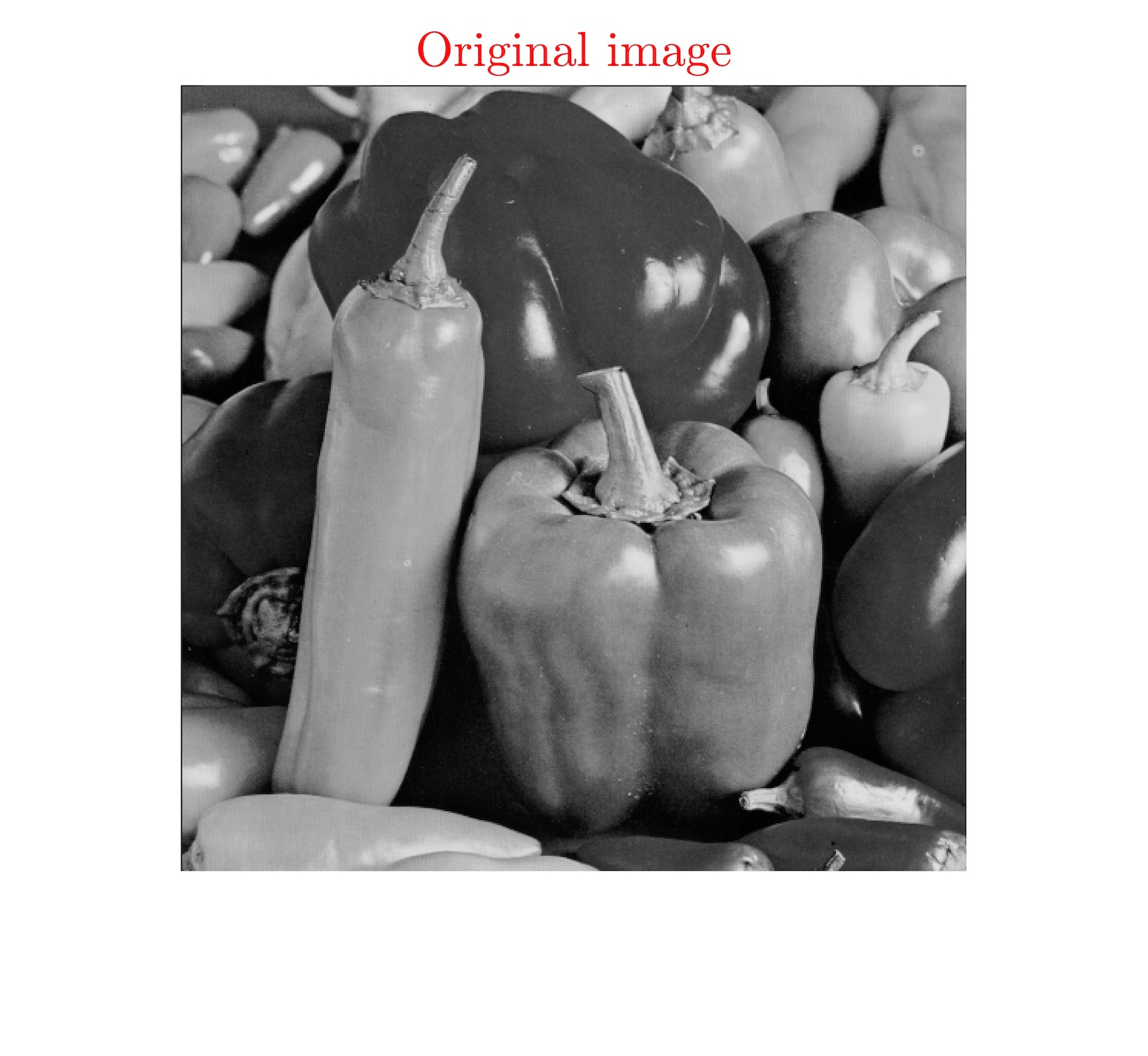}}
\hspace{-1cm}
\subfigure[]{\includegraphics[width=0.365\linewidth]{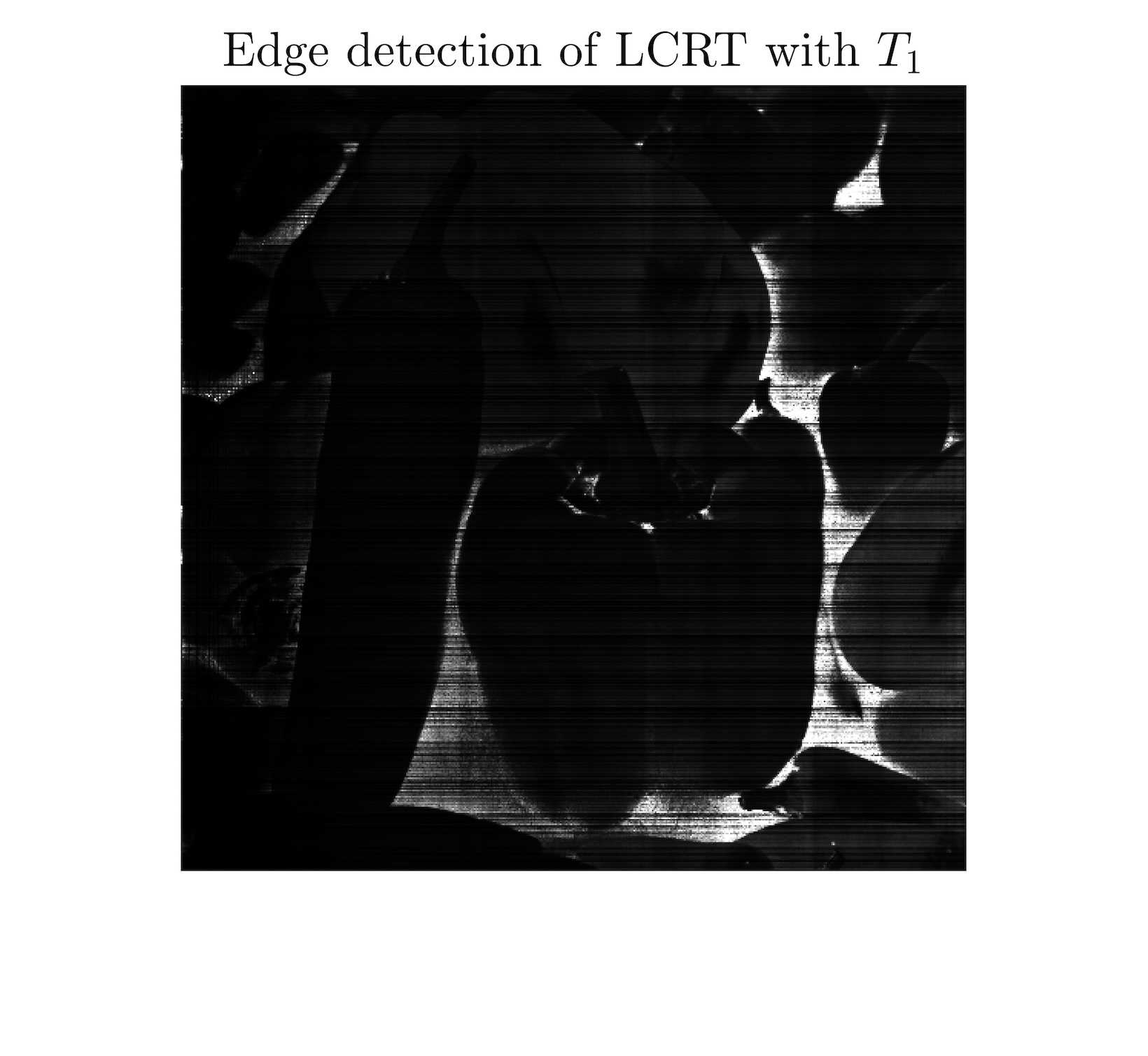}}
\hspace{-1cm}
\subfigure[]{\includegraphics[width=0.365\linewidth]{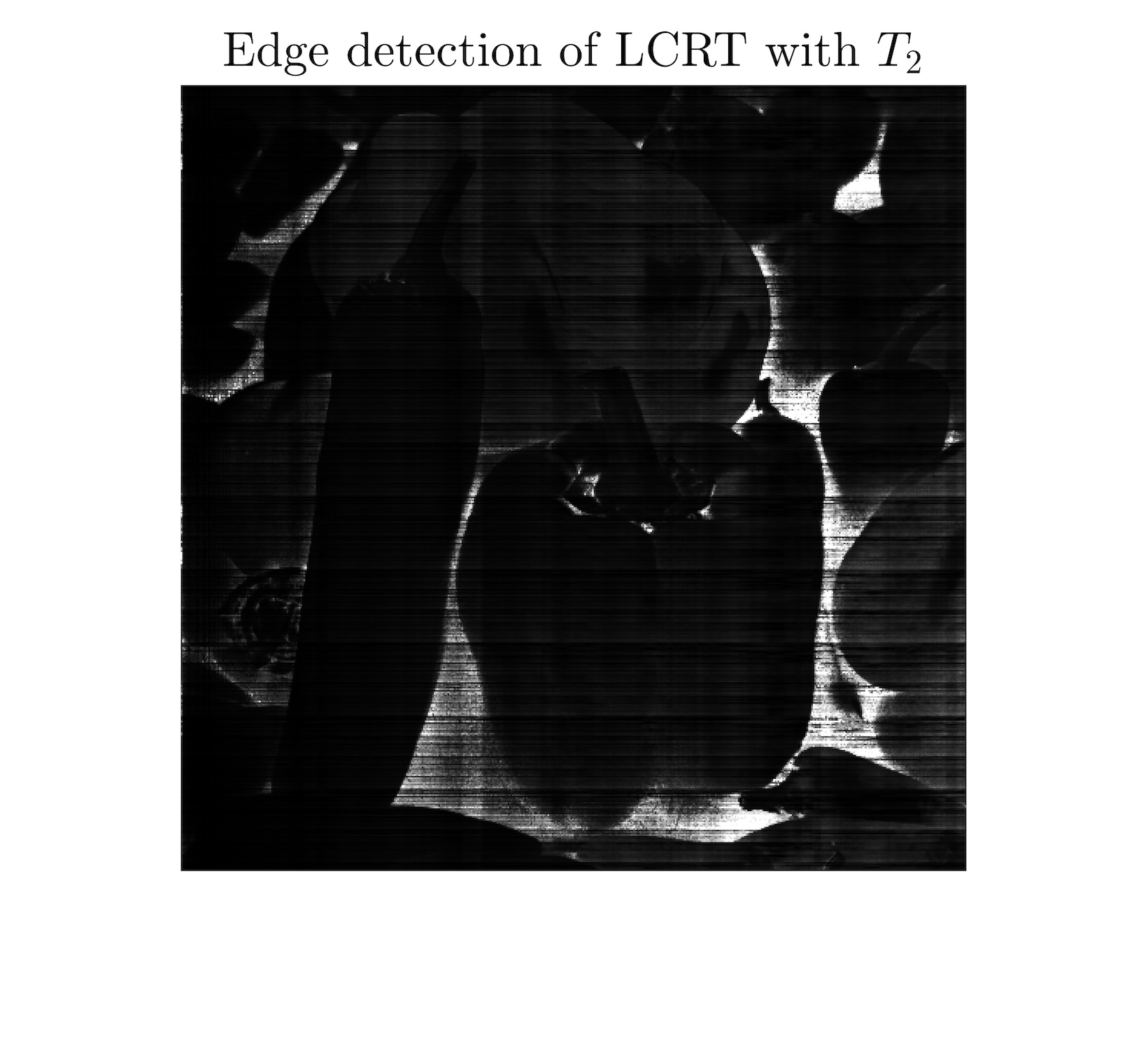}}\\
\vspace{-0.5cm}
\subfigure[]{\includegraphics[width=0.365\linewidth]{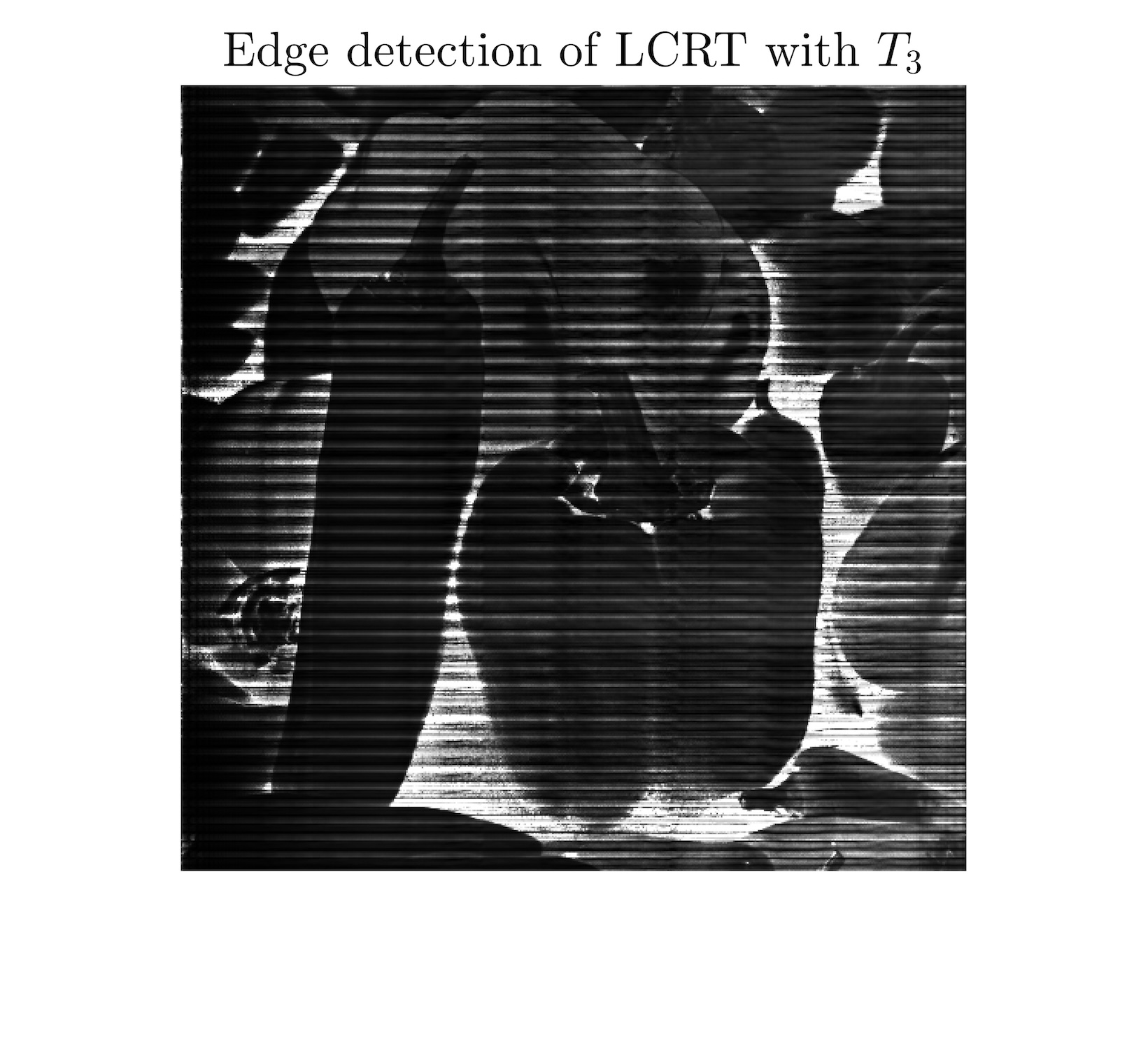}}
\hspace{-1cm}
\subfigure[]{\includegraphics[width=0.365\linewidth]{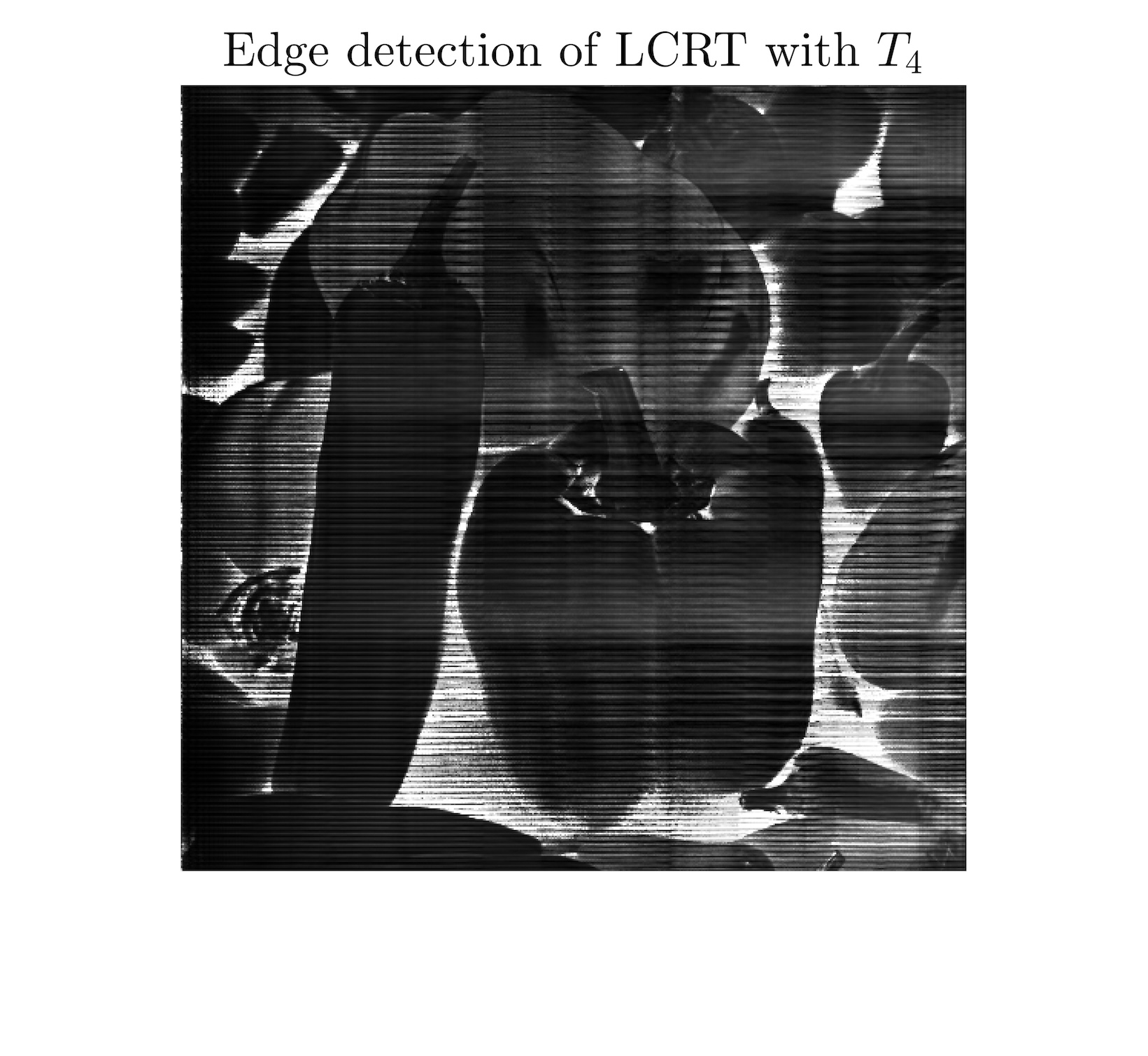}}
\hspace{-1cm}
\subfigure[]{\includegraphics[width=0.365\linewidth]{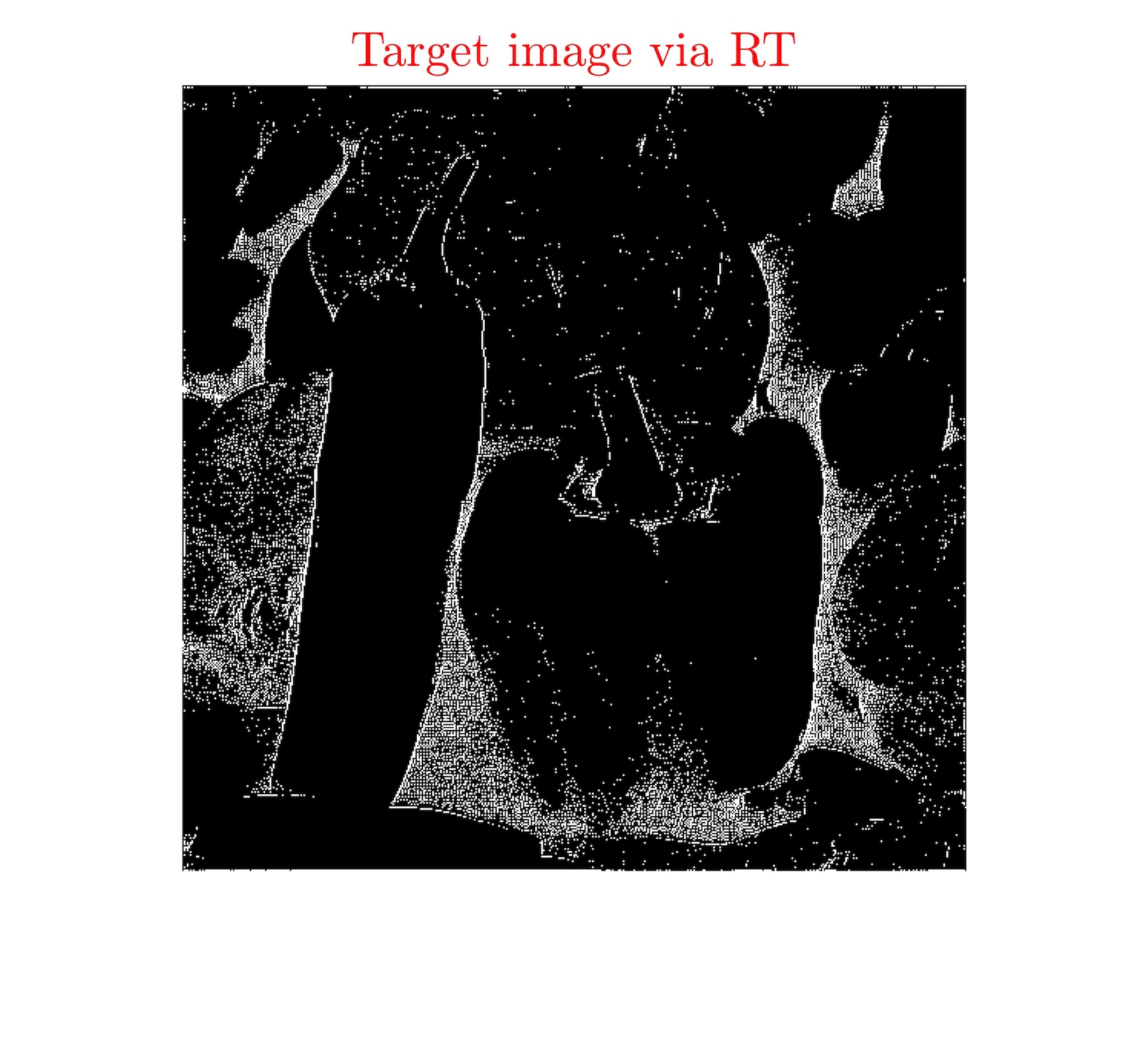}}
\vspace{-0.5cm}
\caption{Edge detection using LCRT-IED method on grayscale 
image Peppers.}
\label{FIG11}
\end{figure}

\begin{table}[H]
\centering
\begin{tabular}{cccccc}
\toprule[1.5pt]
{\bf Graph} &   (b) &   (c) &   (d) &   (e) &  (f) \\
\midrule[1pt]
{\bf  $R_{\mathrm{sc}}^{\mathrm{E}}(A_2)$} & 51.1024
& 9.0010 & 0.4901 & 0.2025 & -1 \\
\midrule[1pt]
{\bf MSE} & 0.2197 & 0.2163 & 0.2143 & 0.2009 & 0.1739 \\
\bottomrule[1.5pt]
\end{tabular}
\caption{$R_{\mathrm{sc}}^{\mathrm{E}}$ and 
MSE corresponding to the edge detection images in Figure 
\ref{FIG11}.}
\label{tab:2}
\end{table}

Table \ref{tab:2} presents $R_{\mathrm{sc}}^
{\mathrm{E}}$ corresponding to the edge detection images in 
Figure \ref{FIG11} and the MSE  
between the original image and the edge detection images in 
Figure \ref{FIG11}. As shown in Table \ref{tab:2}, as
$R_{\mathrm{sc}}^{\mathrm{E}}$ approaches $-1$,  
the MSE between the original image and the edge detection 
images in Figure \ref{FIG11} becomes increasingly smaller.

Again, we next investigate the effectiveness of the aforementioned 
introduced LCRT-IED method for local feature extraction in images. 
Graph (a) of Figure \ref{FIG11-1} demonstrates dividing an image 
into 9 equally-sized sub-regions. Graphs (b) and (c)
in Figure \ref{FIG11-1} display, respectively, the MSEs, corresponding 
to the 9 sub-regions of  Graph (a) of Figure \ref{FIG11-1}, between 
Graphs (a) and (e) of Figure \ref{FIG11} and between Graphs (a) 
and (f) of Figure \ref{FIG11}.

As demonstrated by the experimental results 
from Figures \ref{FIG11} and \ref{FIG11-1} and Table 
\ref{tab:2}, the findings of this study align consistently 
with the results presented in Figures \ref{FIG10} and 
\ref{FIG10-1} and Table \ref{tab:1}, thereby confirming 
that the LCRT-IED method exhibits  reliability and 
generalizability.

In Figures \ref{FIG10} and \ref{FIG11}, we perform
LCRT-IED method on two different grayscale images
using different parameter matrices. Indeed, our LCRT-IED  
method can also be directly applied to RGB images,
achieving similar experimental results. However, the
processing approach slightly differs when handling RGB
images. To be precise, we first divide the RGB image into
the red, the green, and the blue channels. Each color channel is
treated as a grayscale image, and we then apply the LCRT-IED method 
with the same parameter matrices to each channel. Finally, we 
recombine these three channels to form the RGB image.

\begin{figure}[H]
\centering
\subfigure[Image segmentation.]{\includegraphics[width=0.3\linewidth]{MSE.jpg}}
\hspace{0.2cm}
\subfigure[MSE of each sub-region between the Graphs (a) and (e) in 
Figure \ref{FIG11}.]{\includegraphics[width=0.3\linewidth]{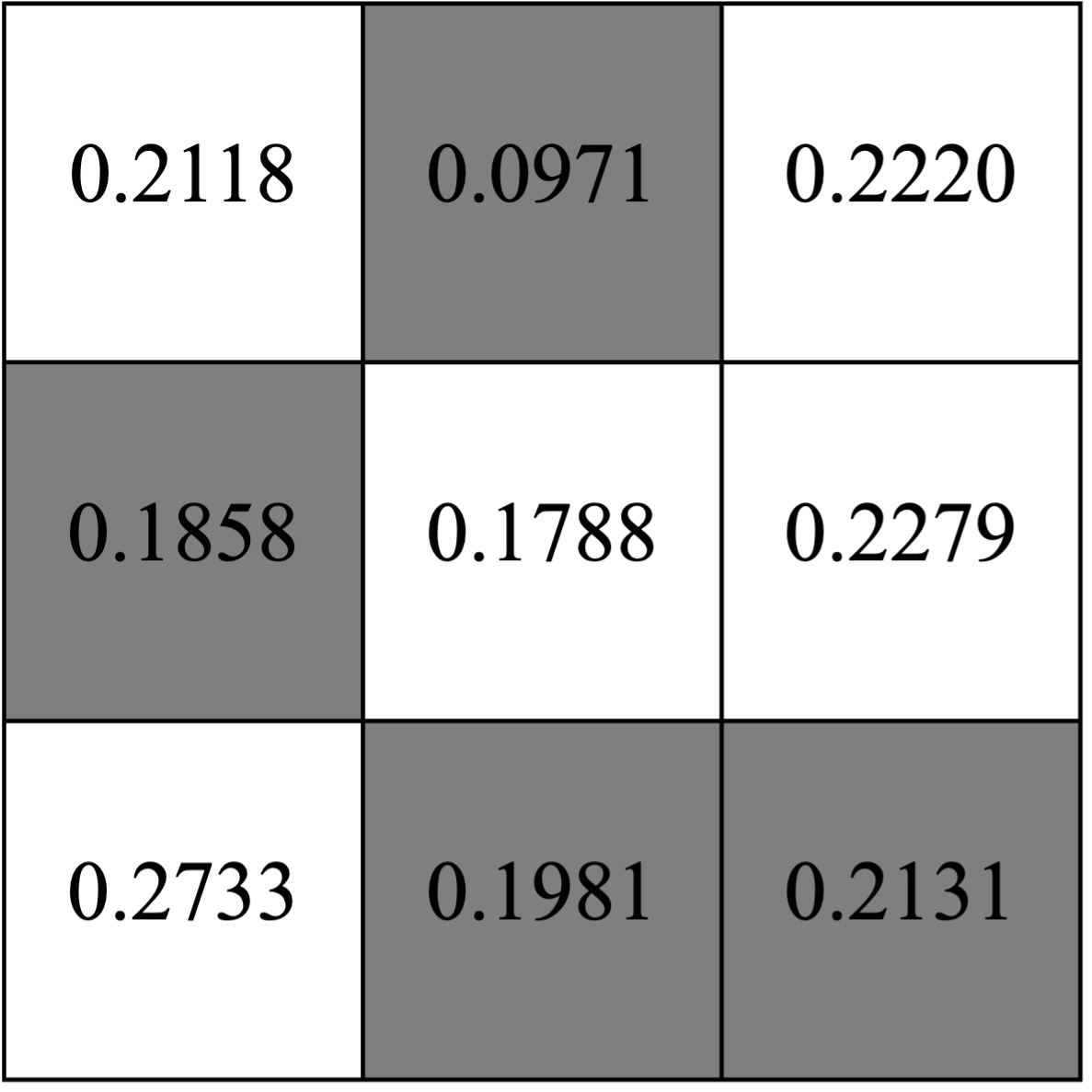}}
\hspace{0.2cm}
\subfigure[MSE of each sub-region between the Graphs (a) and (f) in 
Figure \ref{FIG11}.]{\includegraphics[width=0.3\linewidth]{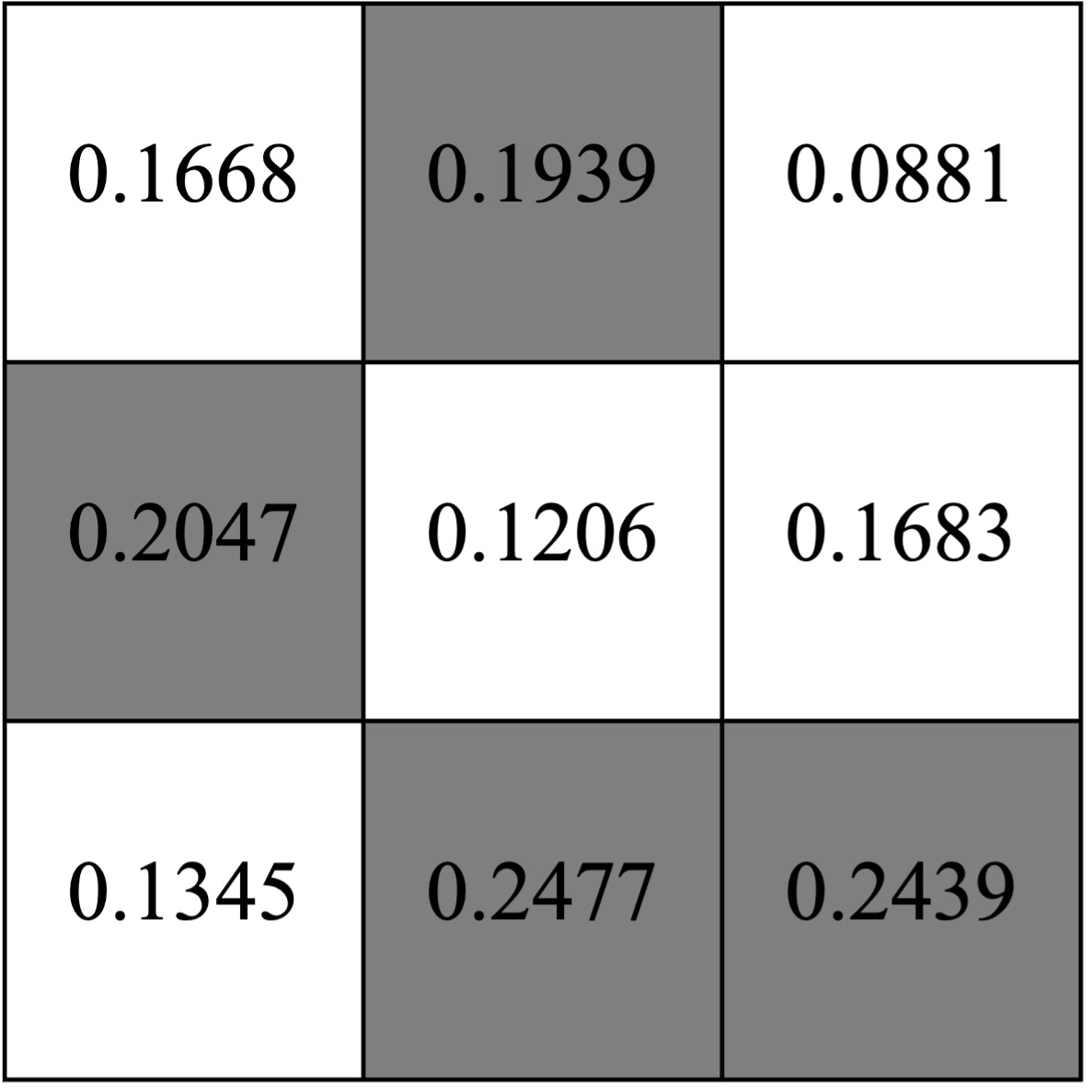}}
\caption{MSE between the original image and the edge detection images 
in Figure \ref{FIG11} after image segmentation.}
\label{FIG11-1}
\end{figure}

\begin{figure}[H]
\centering
\subfigcapskip=-30pt
\subfigure[]{\includegraphics[width=0.365\linewidth]{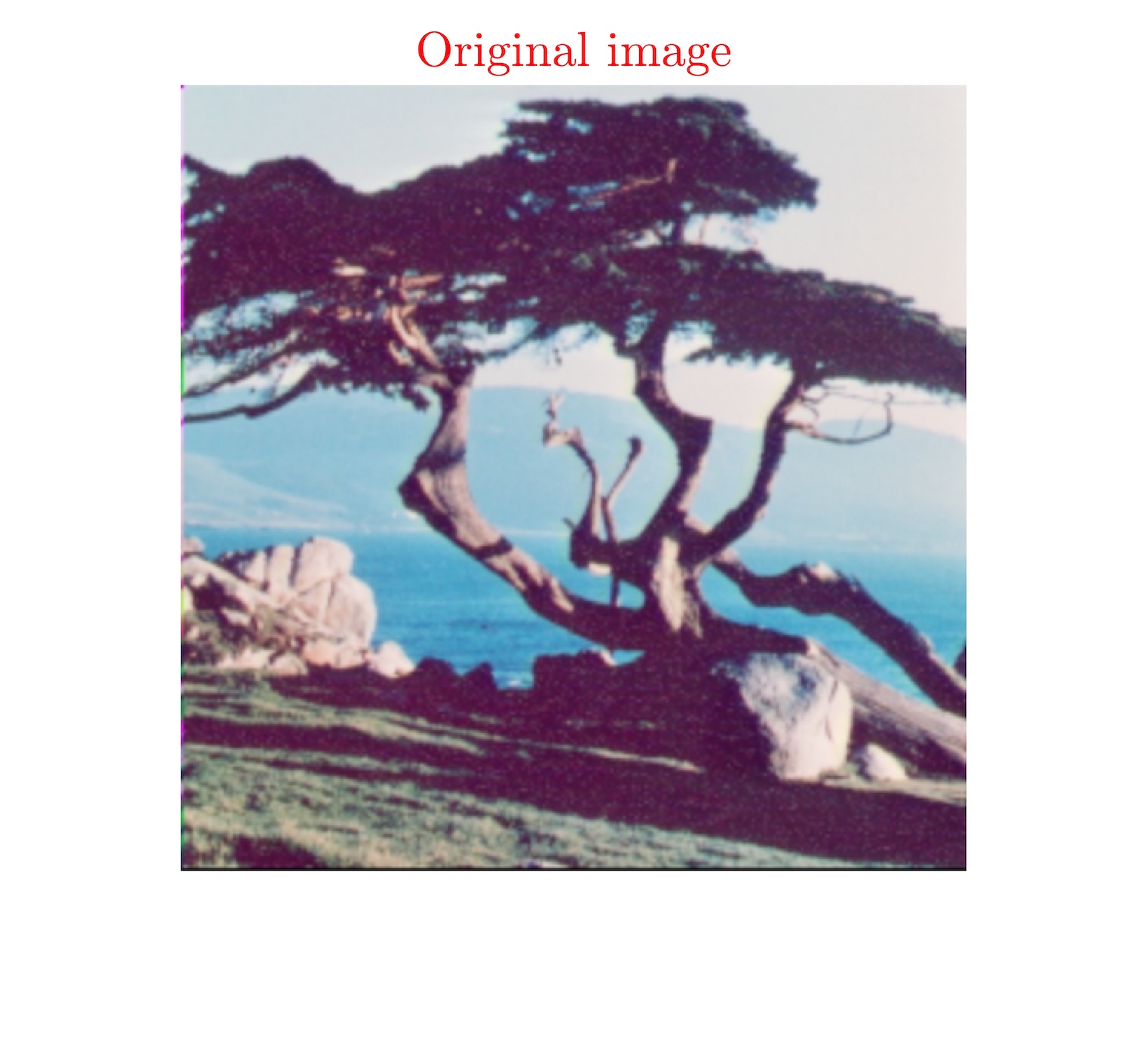}}
\hspace{-1cm}
\subfigure[]{\includegraphics[width=0.365\linewidth]{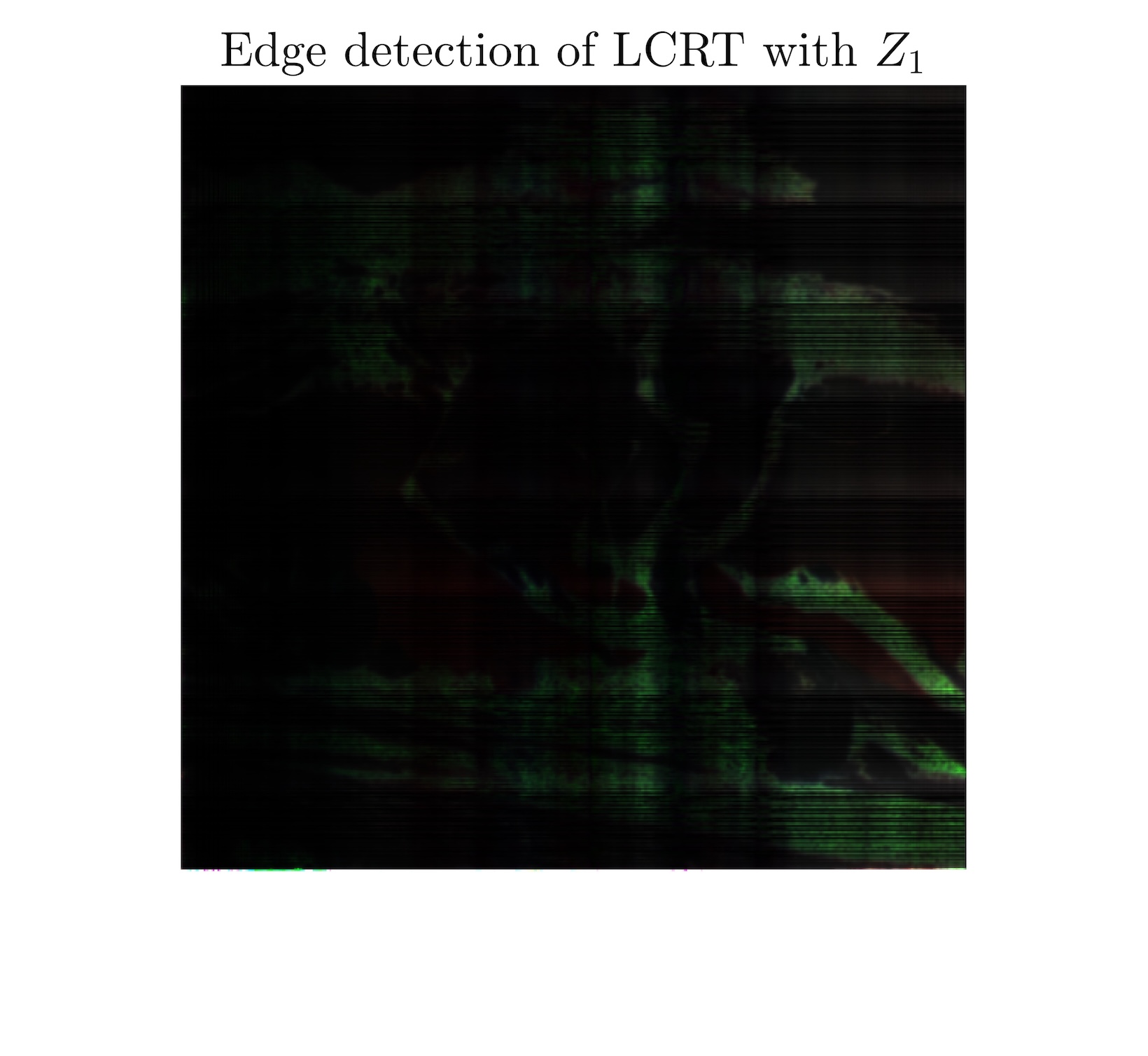}}
\hspace{-1cm}
\subfigure[]{\includegraphics[width=0.365\linewidth]{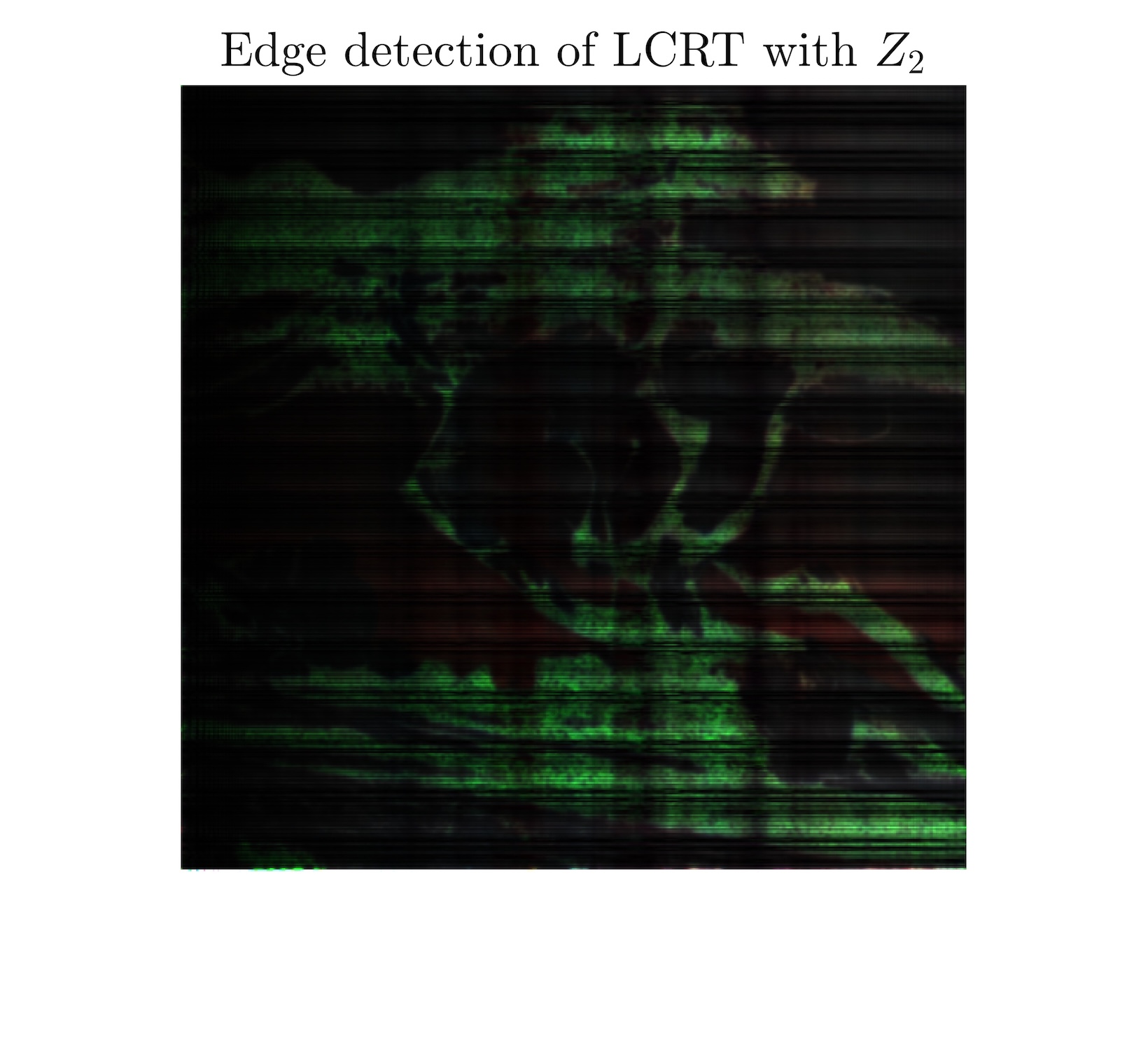}}\\
\subfigure[]{\includegraphics[width=0.365\linewidth]{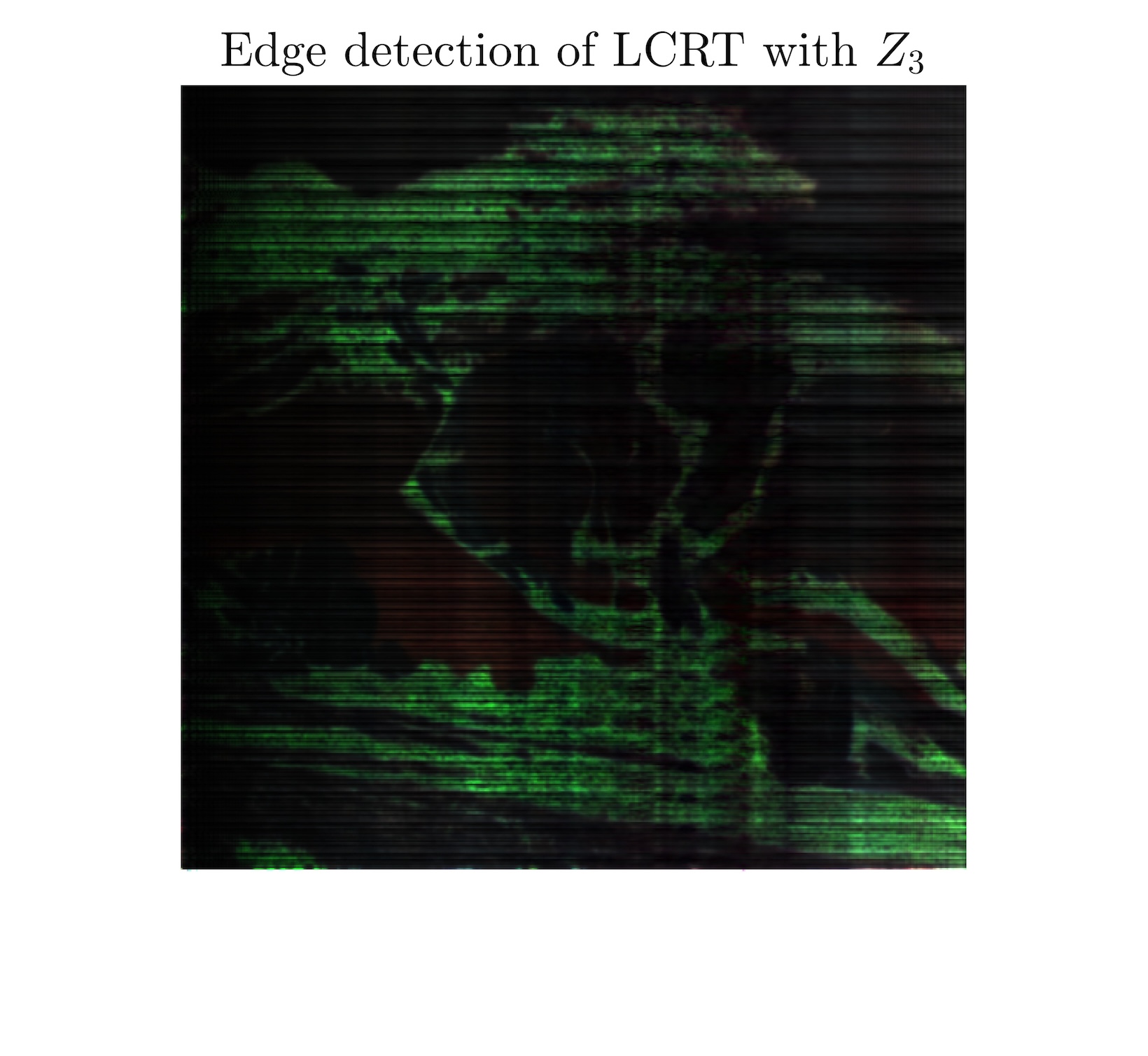}}
\hspace{-1cm}
\subfigure[]{\includegraphics[width=0.365\linewidth]{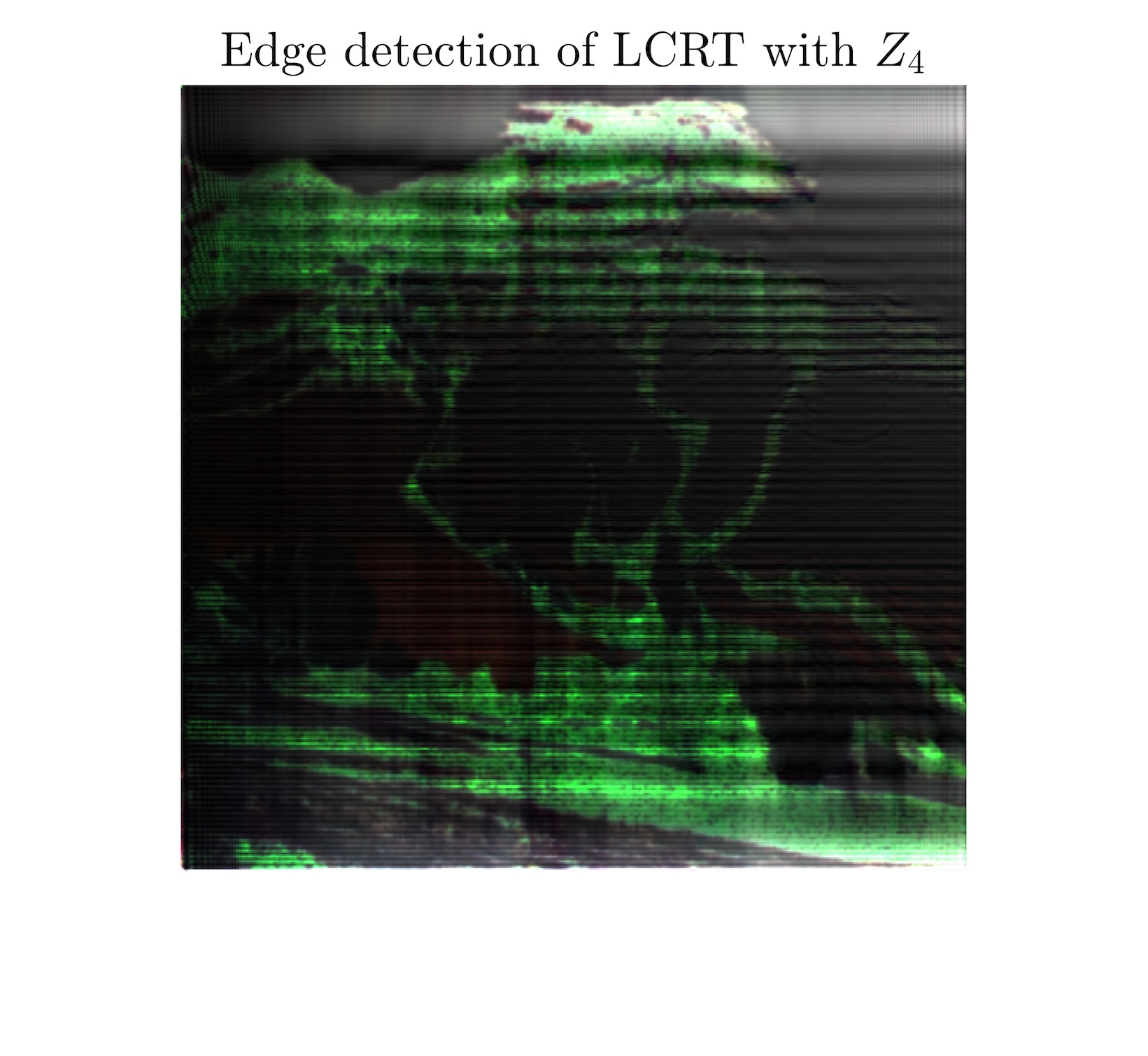}}
\hspace{-1cm}
\subfigure[]{\includegraphics[width=0.365\linewidth]{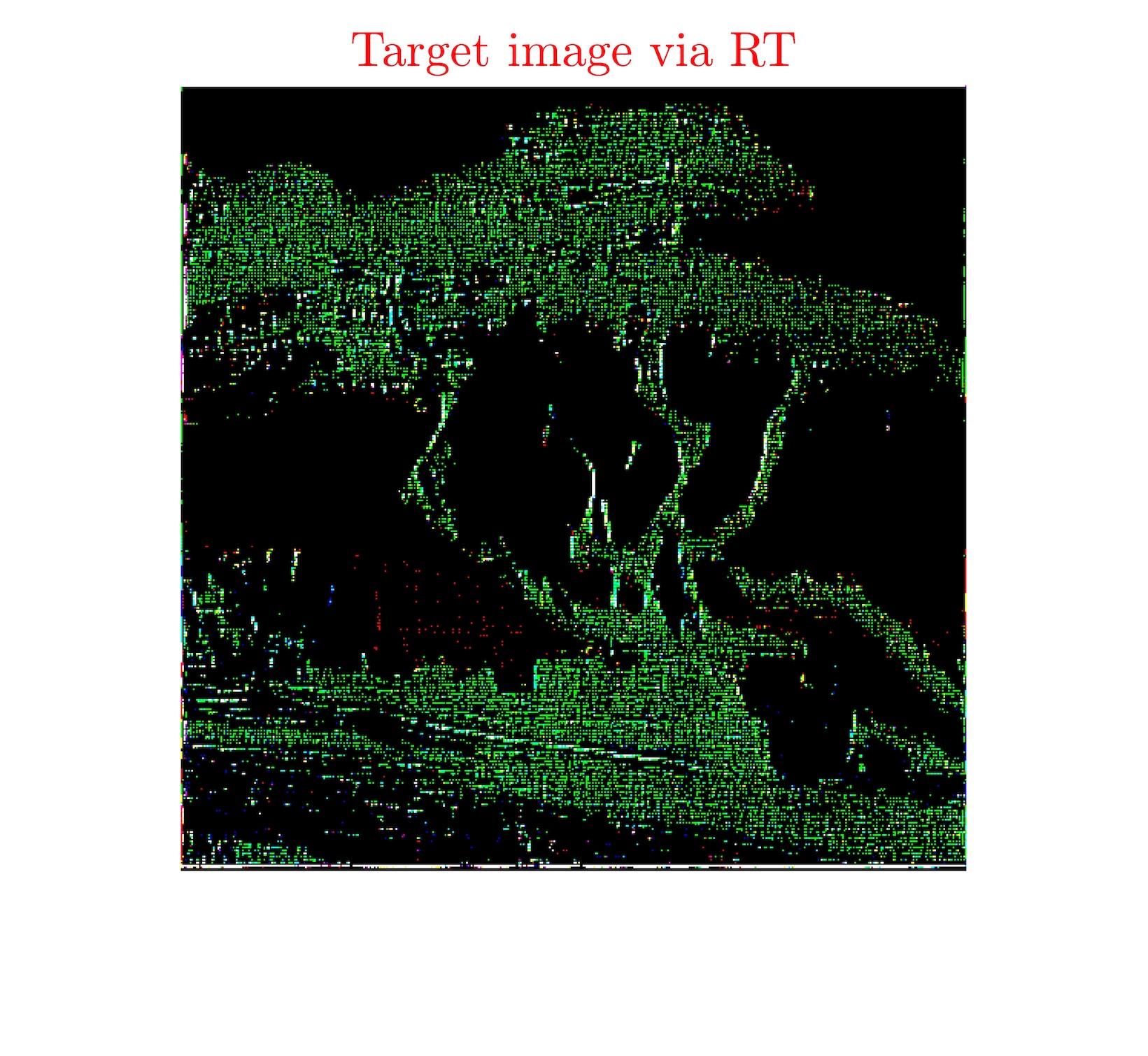}}
\vspace{-0.5cm}
\caption{Edge detection using LCRT-IED method on RGB image.}
\label{FIG1111}
\end{figure}

In Figure \ref{FIG1111}, we use the classic RGB 
image in Graph (a) as the original test image; Graphs (b), (c), 
(d), (e), and (f) correspond to the edge detection images 
via the LCRT-IED method by using,
respectively,  parameter matrices
$Z_1:=(z_1,z_2)$ with $z_1:=\begin{bmatrix}{0}\ &
{1}\\{-1}\ &{0}\end{bmatrix}$ and $z_2:=\begin
{bmatrix}{0}\ &{400}\\{-\frac{1}{400}}\ &{0}\end{bmatrix}$,
$Z_2:=(z_1,z_3)$ with  $z_3:=\begin
{bmatrix}{0}\ &{200}\\{-0.005}\ &{0}\end{bmatrix}$,
$Z_3:=(z_1,z_4)$ with  $z_4:=\begin{bmatrix}
{0}\ &{125}\\{-0.008}\ &{0}\end{bmatrix}$,
$Z_4:=(z_1,z_5)$ with  $z_5:=\begin{bmatrix}
{0}\ &{22}\\{-\frac{1}{22}}\ &{0}\end{bmatrix}$, and
$Z_5:=(z_1,z_1)$.
In particular, when the parameter matrix is $Z_5$, 
in this case the LCRT reduces to the classical Riesz 
transform. Graphs (b), (c), (d), and (e) demonstrate the 
gradual enhancement of  the edge  strength and continuity 
in the edge images through using  different parameter matrices 
in the LCRT, which gradually approach our target image Graph (f).

\begin{table}[H]
\centering
\begin{tabular}{cccccc}
\toprule[1.5pt]
{\bf Graph} &   (b) &   (c) &   (d) &   (e) &  (f) \\
\midrule[1pt]
$R_{\mathrm{sc}}^{\mathrm{E}}(A_2)$ & -160000 
& -40000 & -15625 & -484 & -1 \\
\midrule[1pt]
{\bf MSE(red)} & 0.2933 & 0.2731& 0.2659 & 0.2148 & 0.1193 \\
\midrule[1pt]
{\bf MSE(green)} & 0.3095 & 0.2975 & 0.2914 & 0.2688 & 0.1729 \\
\midrule[1pt]
{\bf MSE(blue)} & 0.3528 & 0.3319 & 0.3215 & 0.2697 & 0.1642 \\
\bottomrule[1.5pt]
\end{tabular}
\caption{$R_{\mathrm{sc}}^{\mathrm{E}}$ and 
MSE corresponding to the edge detection images in Figure 
\ref{FIG1111}.}
\label{tab:3}
\end{table}

\begin{figure}[H]
\centering
\subfigure[MSE of the Red Channel between (a) and (e) in Figure  \ref{FIG1111}.]
{\includegraphics[width=0.3\linewidth]{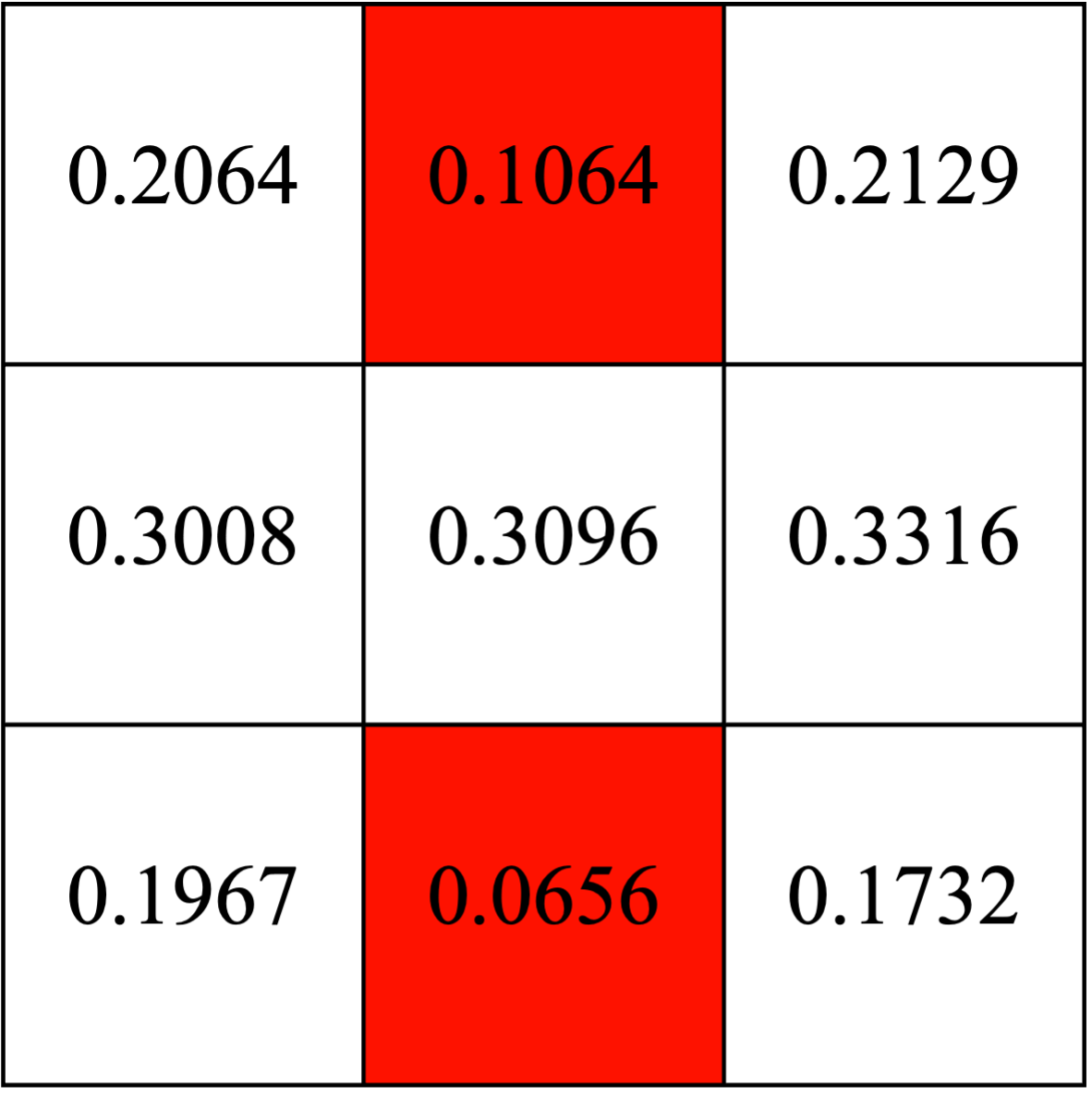}}
\hspace{0.2cm}
\subfigure[MSE of the green channel between (a) and (e) in Figure  \ref{FIG1111}.]
{\includegraphics[width=0.3\linewidth]{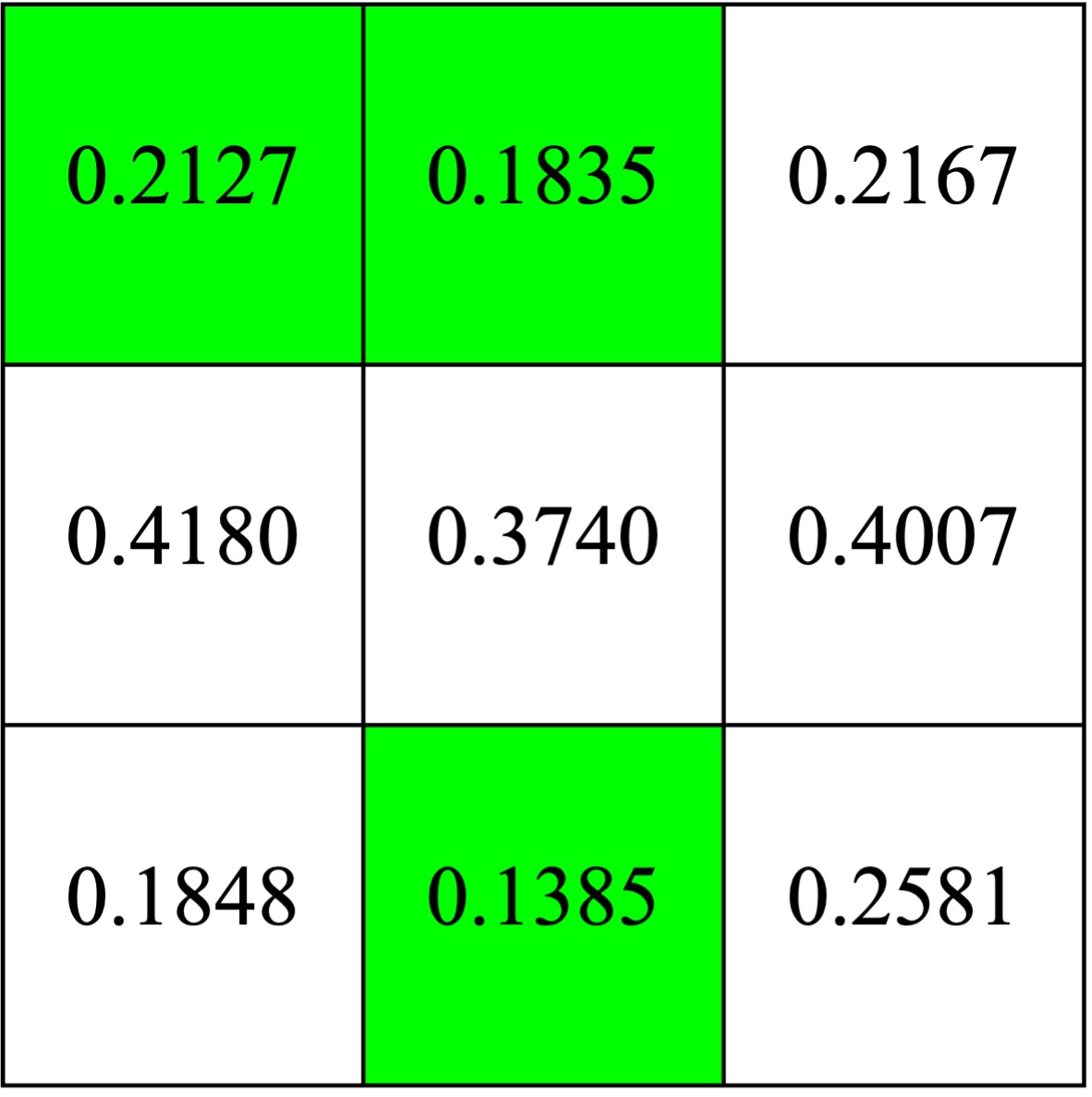}}
\hspace{0.2cm}
\subfigure[MSE of the blue channel between (a) and (e) in Figure  \ref{FIG1111}.]
{\includegraphics[width=0.3\linewidth]{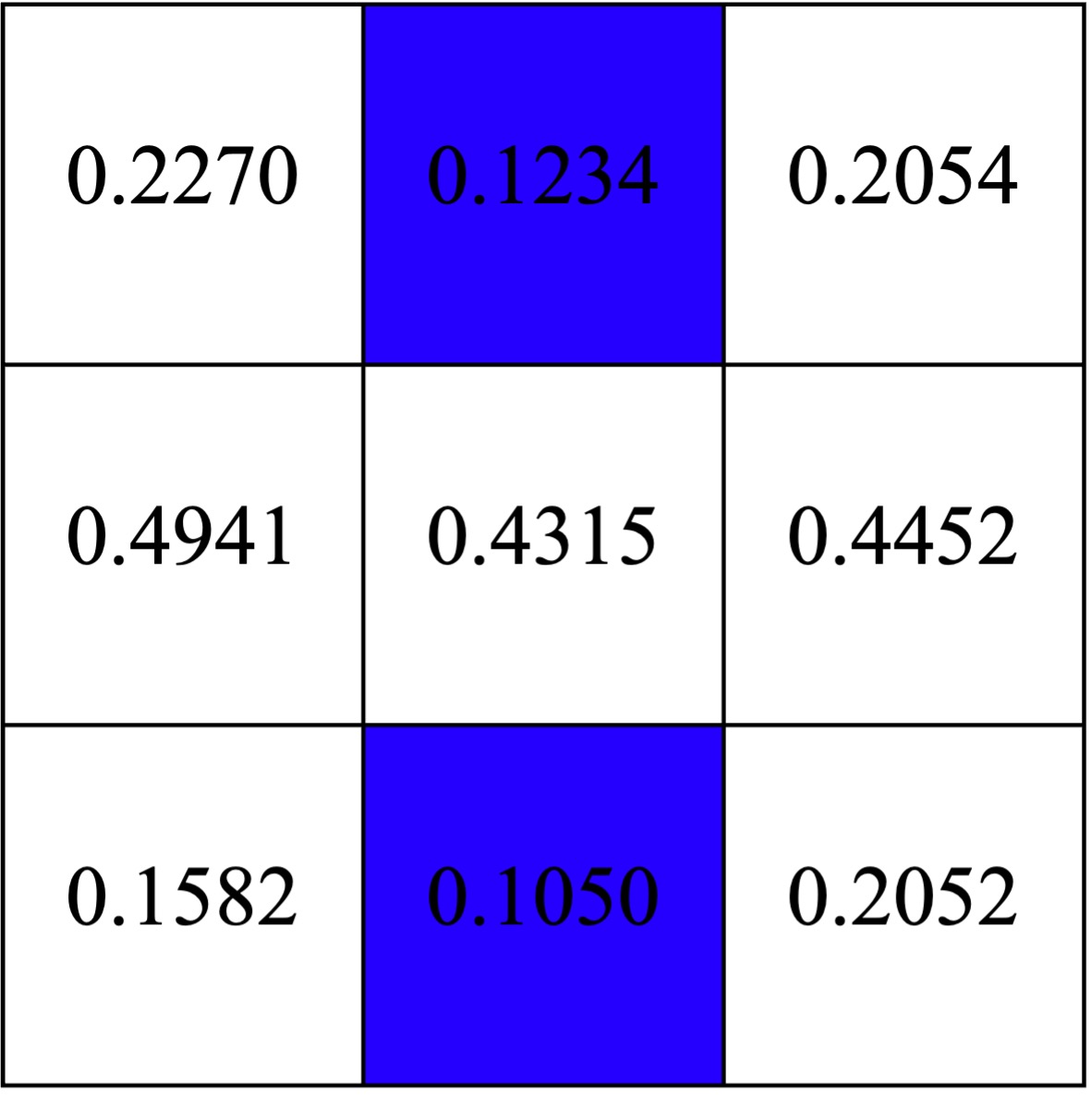}}\\
\subfigure[MSE of the red channel between (a) and (f) in Figure \ref{FIG1111}.]
{\includegraphics[width=0.3\linewidth]{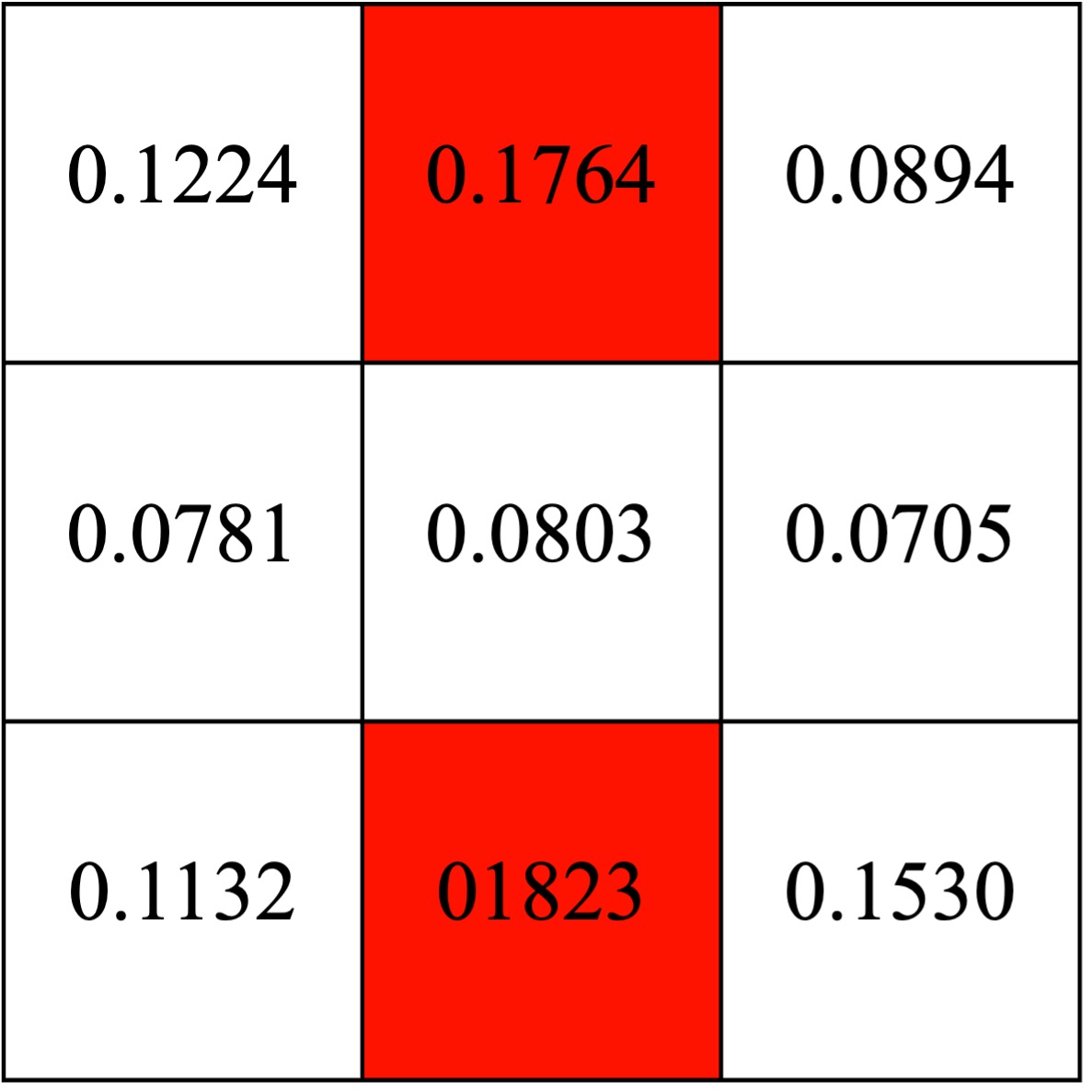}}
\hspace{0.2cm}
\subfigure[MSE of the green channel between (a) and (f) in Figure  \ref{FIG1111}.]
{\includegraphics[width=0.3\linewidth]{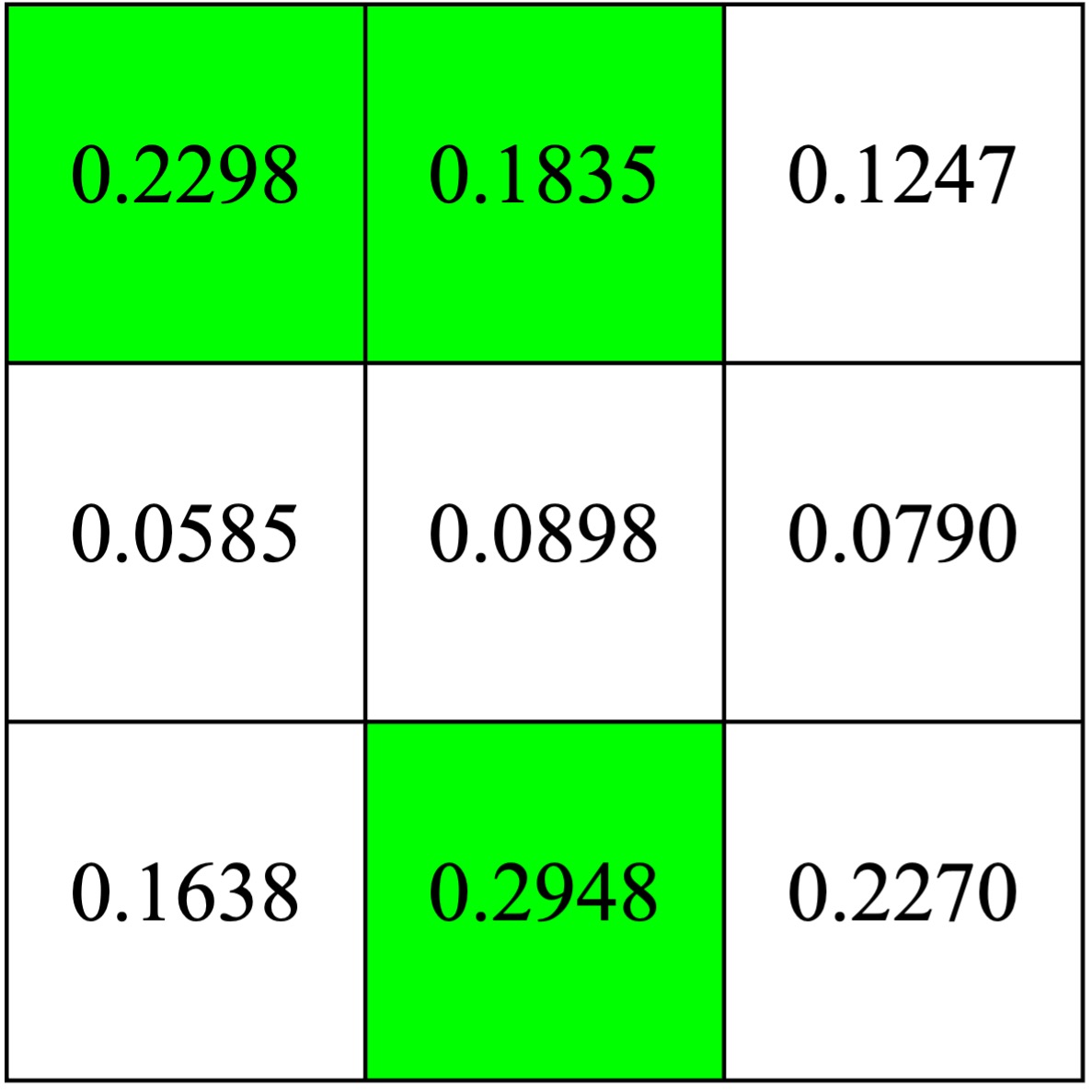}}
\hspace{0.2cm}
\subfigure[MSE of the blue channel between (a) and (f) in Figure  \ref{FIG1111}.]
{\includegraphics[width=0.3\linewidth]{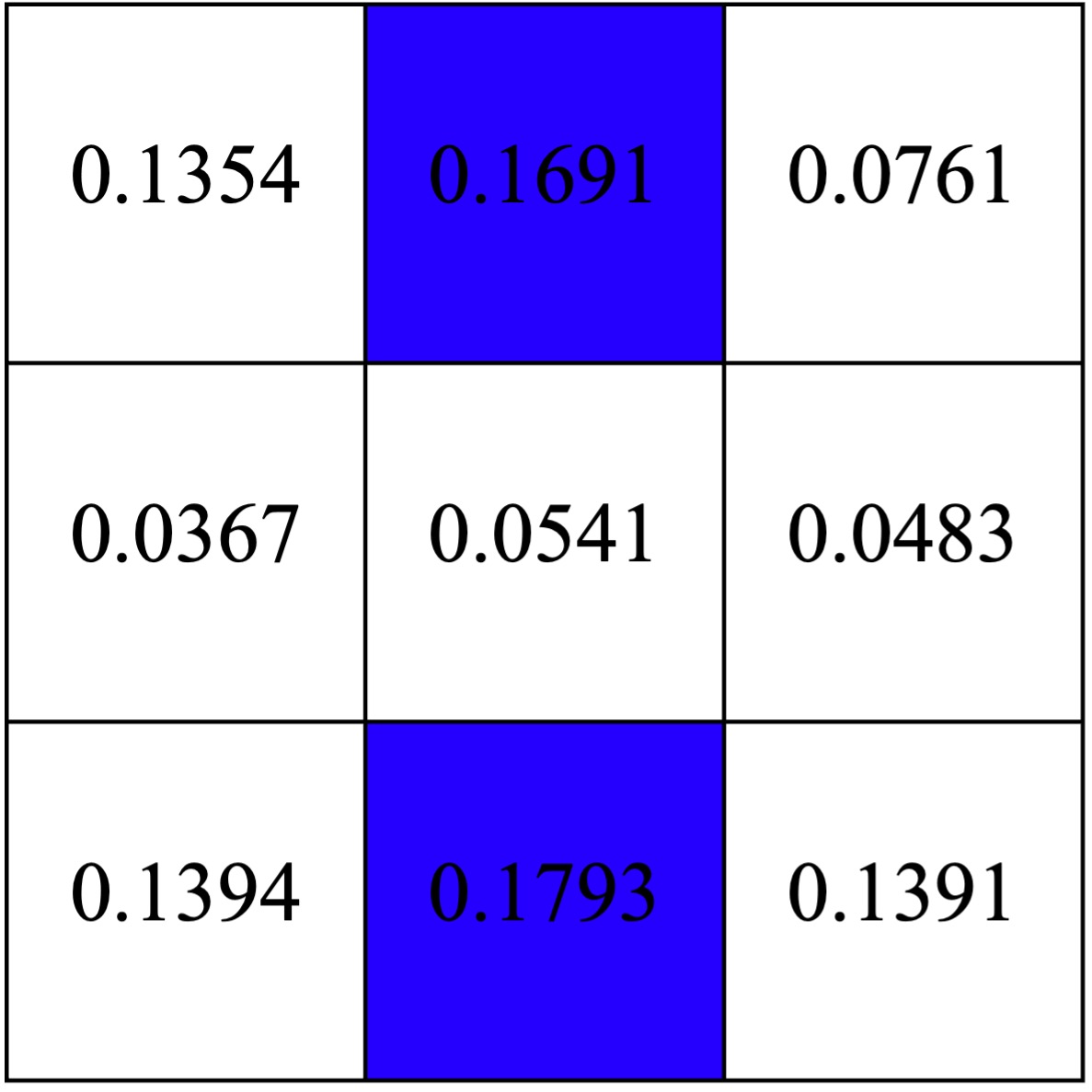}}
\caption{MSE between the original image and the edge detection images
 in Figure \ref{FIG1111} in the red, the green, and the blue 
channels after image segmentation.}
\label{FIG1111-1}
\end{figure}

Since the original test image in Figure \ref{FIG1111} is an
RGB image, Table \ref{tab:3} presents the $R_{\mathrm{sc}}
^{\mathrm{E}}$ corresponding to the edge detection images 
in Figure \ref{FIG1111}, as well as the MSE between the 
original image and the edge detection images in Figure 
\ref{FIG1111}, respectively, on the red, the green, and the blue channels. 
As shown in Table \ref{tab:3}, when $R_{\mathrm{sc}}^
{\mathrm{E}}$ approaches $-1$, the MSE between the 
original image and the edge detection images in Figure 
\ref{FIG1111} on all three channels increasingly decreases, 
indicating the higher degree of similarity between them.

Figure \ref{FIG1111-1} shows the MSEs, respectrvely,
 between Graphs (a) in Figure \ref{FIG1111} and Graphs (e) 
 and (f)  in Figure \ref{FIG1111} about each sub-region 
in the red, the green, and the blue channels after the same 
image segmentation as in Figure \ref{FIG10-1}. 
Particularly, Graphs (a), (b), and (c) in Figure \ref{FIG1111-1} 
display the MSEs, respectively, between Graphs (a) and (f) 
in Figure \ref{FIG1111}  the aforementioned 9
sub-regions in the red, the green, and the blue channels. Graphs (d), (e), and (r) in 
Figure \ref{FIG1111-1} display the MSEs, respectively, between 
images (a) and (e) in Figure \ref{FIG1111} about the aforementioned 9
sub-regions in the red, the green, and the blue channels.

As demonstrated by the experimental data in 
Figures \ref{FIG1111} and \ref{FIG1111-1}, as well as 
Table \ref{tab:3}, the results for RGB images and the two 
types of grayscale images show the high degree of consistency. 
This validates that the LCRT-IED method also possesses excellent 
feature extraction capabilities in RGB image processing.
 
Based on the results of the three sets of experiments 
mentioned above, we find that this new LCRT-IED method
introduced above has two key  advantages:
\begin{enumerate}[(i)]
\item
{\bf Effective control of edge properties}:
The  LCRT-IED method is able to gradually control the image 
edge strength and continuity.
\item
{\bf Local feature extraction}:
The  LCRT-IED method demonstrates the significant capability in 
feature extraction about some local regions, enabling more precise 
preservation of image detail information.
\end{enumerate}
These advantages highlight the practical value of the LCRT-IED method 
for image processing tasks in complex scenarios.
When we reduce from LCRTs to fractional Riesz transforms,
selecting a specific fractional order can effectively extract
local information in specific directions of the image.
It is worth mentioning that the experimental results
presented in this article do \emph{not require final binarization}
like the fractional Riesz transform edge detection
method in \cite{fglwy23} and, moreover, the results are equally striking.
This can be attributed to the fact that the LCT is able
to perform not only rotations but also scaling operations
in the time-frequency plane. Furthermore, the LCRT-IED method 
also performs exceptionally well when processing RGB images.

\begin{remark}
By observing the results of the three sets of experiments above, 
we notice that, in the above experiments, we always fix the parameter matrix 
$A_1$ in LCRTs $\{R_1^{\boldsymbol{A}},R_2^
{\boldsymbol{A}}\}$  with $\boldsymbol{A}:=(A_1,A_2)$ as 
in Definition \ref{Rieszdef1} with $n=2$ and only allow $A_2$ to vary. This naturally raises a question:
 if we instead fix $A_2$ and let $A_1$ vary, would the same effect be observed? 
 However, experimental validation shows that this is not the case, and only some
 weak edge information is available in this case.  The deep reason for this lies in the fact that the roles of $R_1^{\boldsymbol A}$ and $R_2^{\boldsymbol A}$ are 
  not symmetric in the local direction $\theta(\boldsymbol{x})$  in  \eqref{eqlo}, leading to an 
  essential difference  in their effectiveness within the model under consideration.
\end{remark}

\subsection{Advantages and Features \label{subsec5.2}}

In image processing, finely tuning the strength and the continuity 
of image edges can significantly enhance processing effects. The 
advantages and features of  LCRT-IED methods can be summarized 
as follows:
\begin{enumerate}[(i)]
\item {\bf Improvement of the efficiency of image matching}:
In image matching, we are confronted with a challenging 
task: how to efficiently identify images that match the target image 
from a vast database containing millions of images. To address this 
challenge, we can employ a multi-stage processing method based on 
edge strength and continuity. This method adopts a hierarchical 
strategy to refine the matching process step by step. Initially, 
the LCRT-IED method is applied to extract edge features with 
varying strengths and continuities. Subsequently, these features
 are used to train and evaluate models, thereby selecting the 
refined edge-detected images for subsequent matching stages. 
 By combining coarse matching and fine matching, dissimilar 
 images are filtered out, and the focus is narrowed down to 
 candidate images with high similarity. Finally, the best-matching 
 image is determined through pairwise comparison. This multi-stage 
 approach not only simplifies the complexity of the matching 
 process but also demonstrates significant advantages when dealing 
 with large-scale image databases.

\item
{\bf Refinement of feature extraction}:
Edges, as important indicators of image brightness changes,
are often closely related to the contours or the boundaries
of objects under consideration.
Edge detection techniques can effectively capture this critical
information, laying a solid foundation for deeper image
analysis and interpretation. For example, in medical image
analysis, edge detection  help to identify the boundaries of
tumors, and, in facial recognition systems, edge detection can
extract facial features. By adjusting the strength and  the continuity
of the edges, these features can be more accurately identified
and emphasized, thus improving the accuracy of image analysis.
\item
{\bf Image refinement processing}:
By employing advanced edge detection techniques, key
information in the image can be accurately retained while
cleverly removing unnecessary pixel details. For example,
in the manufacturing industry, image refinement processing
techniques can be used to inspect product surfaces. By
accurately preserving edge information, surface defects such
as scratches or blemishes can be quickly identified. This fine-tuning
not only significantly reduces computational burden and
speeds up processing, but also greatly enhances the intuitiveness
and the efficiency of image analysis.
\item
{\bf Catalyst for subsequent processing}:
Finely tuning the strength and the continuity of image
edges provides a solid and precise foundation for
subsequent tasks such as image segmentation,
object recognition, and tracking. The  refined
edge information simplifies the accurate identification
and the localization of target objects, significantly
improving the efficiency and the accuracy of follow-up
processing tasks. This ensures the smoothness and the
efficiency of the entire image analysis process and
provides strong support for subsequent stages of
image processing.
\end{enumerate}

\section{Conclusions}\label{sec5}
This article introduces LCRTs on the
basis of LCTs and Riesz transforms.  We also further perform
a series of numerical simulations on images
using both LCRTs and HLCHTs. By comparing these simulation
experiments, we validate their respective effectiveness and highlight
the differences among them. 
In addition, we apply LCRTs to refine image 
edge detection and propose a new LCRT image edge detection method, namely the
LCRT-IED method, by first introducing the concept of 
the sharpness $R^{\mathrm E}_{\mathrm{sc}}$ of the edge 
strength and continuity of images 
associated with the LCRT.  This method is able to gradually control the edge strength 
and continuity by subtly tuning the parameter matrices of LCRTs  based on the sharpness $R^{\mathrm E}_{\mathrm{sc}}$
and performs well in feature extraction on some local regions.
Moreover, this method is not only applicable to grayscale 
images but also works well with RGB images. This new method can be
regarded as the refinement of the fractional Riesz transform edge detection method
in \cite{fglwy23}. Serving as a 
catalyst for subsequent image processing, it is of great 
significance for image matching, image feature extraction, 
and image refinement processing.

\section*{Data availability}

No data was used for the research described in the article.

\bigskip

\noindent  Shuhui Yang  and Dachun Yang (Corresponding author)

\smallskip

\noindent  Laboratory of Mathematics and Complex
Systems (Ministry of Education of China),
School of Mathematical Sciences, Beijing Normal
University, Beijing 100875, The People's Republic of China

\smallskip

\noindent{\it E-mails:} \texttt{shuhuiyang@bnu.edu.cn} (S. Yang)

\noindent\phantom{{\it E-mails:}} \texttt{dcyang@bnu.edu.cn} (D. Yang)

\bigskip

\noindent Zunwei Fu

\smallskip

\noindent School of Mathematics and Statistics, Linyi University,
Linyi 276000, The People's Republic of China; College of Information
Technology, The University of Suwon, Hwaseong-si 18323,
South Korea

\smallskip

\noindent{\it E-mail:} \texttt{zwfu@suwon.ac.kr}

\bigskip

\noindent  Yan Lin and Zhen Li

\smallskip

\noindent School of Science, China University of Mining and
Technology, Beijing 100083, The People's Republic of China

\smallskip

\noindent{\it E-mails:} \texttt{linyan@cumtb.edu.cn} (Y. Lin)

\noindent\phantom{{\it E-mails:}} \texttt{lizhen@student.cumtb.edu.cn} (Z. Li) 

\begin{thebibliography}{99}
\bibitem{asf2017}
S. A. Abbas,   Q. Sun and   H. Foroosh,
An exact and fast computation of discrete Fourier transform for
polar and spherical grid,
IEEE Trans. Signal Process. 65 (2017), 2033--2048.

\vspace{-.3cm}

\bibitem{b1961}
V. Bargmann,
On a Hilbert space of analytic functions and an
associated integral transform,
Comm. Pure Appl. Math. 14 (1961), 187--214.

\vspace{-.3cm}

\bibitem{bko1997}
B. Barshan, M. A. Kutay and H. M. Ozaktas,
Optimal filtering with linear canonical transformations,
Opt. Commun.135 (1997), 32--36.

\vspace{-.3cm}

\bibitem{b1978}
M. J. Bastiaans,
The Wigner distribution function applied
to optical signals and systems,
Opt. Commun. 25 (1978), 26--30.

\vspace{-.3cm}

\bibitem{b1996}
L. M. Bernardo,
ABCD matrix formalism of fractional Fourier optics,
Opt. Eng. 35 (1996), 732--740.

\vspace{-.3cm}

\bibitem{bshyh2007}
N. Bi, Q. Sun, D. Huang, Z. Yang and J. Huang,
Robust image watermarking based on multiband wavelets
and empirical mode decomposition,
IEEE Trans. Image Process. 16 (2007), 1956--1966.

\vspace{-.3cm}

\bibitem{b19921}
B. Boashash,
Estimating and interpreting the instantaneous
frequency of a signal. I. Fundamentals,
Proc. IEEE 80 (1992), 520--538.

\vspace{-.3cm}

\bibitem{b19922}
B. Boashash,
Estimating and interpreting the instantaneous frequency of
a signal. II. Algorithms and applications,
Proc. IEEE 80 (1992), 540--568.

\vspace{-.3cm}

\bibitem{cd14}
G. Chaple and R. D. Daruwala,
Design of Sobel operator based image edge detection
algorithm on FPGA,
2014 International Conference on Communication
and Signal Processing, 2014, pp. 788--792, IEEE.

\vspace{-.3cm}

\bibitem{cfgw21}
W. Chen, Z. Fu, L. Grafakos and Y. Wu,
Fractional Fourier transforms on $L^p$ and applications,
Appl. Comput. Harmon. Anal. 55 (2021), 71--96.

\vspace{-.3cm}

\bibitem{ccs23}
Y. Chen,  C. Cheng and  Q. Sun,
Graph Fourier transform based on singular
value decomposition of the directed Laplacian,
Sampl. Theory Signal Process. Data Anal. 21 (2023),
Paper No. 24, 28 pp.

\vspace{-.3cm}

\bibitem{ccls23}
 C. Cheng, Y. Chen, Y. J. Lee and Q. Sun,
 SVD-based graph Fourier transforms on directed
 product graphs,
  IEEE Trans. Signal Inform. Process. Netw. 9
  (2023), 531--541.

\vspace{-.3cm}

\bibitem{c1970}
S. A. Collins,
Lens-system diffraction integral written
in terms of matrix optics,
J. Opt. Soc. Amer. 60 (1970), 1168--1177.

\vspace{-.3cm}

\bibitem{dwy13}
C. Deng, G. Wang and X. Yang,
Image edge detection algorithm based on improved canny operator,
2013 International Conference on Wavelet Analysis
and Pattern Recognition, 2013, pp. 168--172, IEEE.

\vspace{-.3cm}

\bibitem{ddp24} N. C. Dias,  M. de Gosson and J. N. Prata, A metaplectic perspective of uncertainty principles in the linear canonical transform domain, J. Funct. Anal. 287 (2024),
 Paper No. 110494, 54 pp.

\vspace{-.3cm}

\bibitem{dp1999}
J. J. Ding and S. C. Pei,
2-D affine generalized fractional Fourier transform,
1999 IEEE International Conference on Acoustics,
Speech, and Signal Processing. Proceedings, 1999,
pp. 3181--3184, IEEE.

\vspace{-.3cm}

\bibitem{fs01}
M. Felsberg and G. Sommer,
The monogenic signal,
IEEE Trans. Signal Process. 49 (2001), 3136--3144.


\vspace{-.3cm}

\bibitem{fl08}
Y. Fu and L. Li,
Generalized analytic signal associated with linear canonical
transform, Opt. Commun. 281 (2008), 1468--1472.

\vspace{-.3cm}

\bibitem{fglwy23}
Z. Fu, L. Grafakos, Y. Lin, Y. Wu and S. Yang,
Riesz transform associated with the fractional
Fourier transform and applications in image edge detection,
Appl. Comput. Harmon. Anal. 66 (2023), 211--235.

\vspace{-.3cm}

\bibitem{flyy24}
Z. Fu, Y. Lin, D. Yang and S. Yang
Fractional Fourier Transforms Meet
Riesz Potentials and Image Processing,
SIAM J. Imaging Sci.17 (2024), 476--500.

\vspace{-.3cm}

\bibitem{g1946}
D. Gabor,
Theory of communication. Part 1: The analysis of information,
J. Inst. Elec. Engrs. Part III 93 (1946), 429--441.

\vspace{-.3cm}

\bibitem{gcn21}
H. Ge,  W. Chen and M. K. Ng,
New restricted isometry property analysis for $\ell_1-\ell_2$
minimization methods,
SIAM J. Imaging Sci. 14 (2021),   530--557.

\vspace{-.3cm}

\bibitem{gcn21-2}
H. Ge,  W. Chen and M. K. Ng,
On recovery of sparse signals with prior support
information via weighted $\ell_p$-minimization,
IEEE Trans. Inform. Theory 67 (2021),   7579--7595.

\vspace{-.3cm}

\bibitem{g20141}
L. Grafakos,
Classical Fourier Analysis, Grad. Texts in Math., 249,
Springer, New York, 2014.

\vspace{-.3cm}

\bibitem{hs22}
Y. Han and  W. Sun,
Inversion of the windowed linear canonical
transform with Riemann sums,
Math. Methods Appl. Sci. 45 (2022), 6717--6738.

\vspace{-.3cm}

\bibitem{hs222}
Y. Han and  W. Sun,
Inversion formula for the windowed linear
canonical transform,
Appl. Anal. 101 (2022),  5156--5170.

\vspace{-.3cm}

\bibitem{hkos15}
J. J. Healy, M. A. Kutay, H. M. Ozaktas and J. T. Sheridan,
Linear canonical transforms, Theory and applications,
Springer, New York, 2015.

\vspace{-.3cm}

\bibitem{hll1997}
J. Hua, L. Liu and G. Li,
Extended fractional Fourier transforms,
J. Opt. Soc. Amer. A 14 (1997), 3316--3322.

\vspace{-.3cm}

\bibitem{hcg24}
L. Huo,  W. Chen and H. Ge,
Image restoration based on transformed total
variation and deep image prior,
Appl. Math. Model. 130 (2024), 191--207.

\vspace{-.3cm}

\bibitem{ja1996}
D. F. V. James and G. S. Agarwal,
The generalized Fresnel transform and
its application to optics,
Opt. Commun. 126 (1996), 207--212.

\vspace{-.3cm}

\bibitem{krz21}
R. Kamalakkannan,  R. Roopkumar and  A. Zayed,
Short time coupled fractional Fourier transform
and the uncertainty principle,
Fract. Calc. Appl. Anal. 24 (2021),   667--688.

\vspace{-.3cm}

\bibitem{krz22}
R. Kamalakkannan,  R. Roopkumar and  A. Zayed,
On the extension of the coupled fractional Fourier
transform and its properties,
Integral Transforms Spec. Funct. 33 (2022),   65--80.

\vspace{-.3cm}

\bibitem{krz23}
R. Kamalakkannan,  R. Roopkumar and  A. Zayed,
Quaternionic coupled fractional Fourier transform on
Boehmians, in: Sampling, Approximation, and Signal
Analysis--Harmonic Analysis in the Spirit of J. Rowland Higgins,
pp. 453--468, Appl. Numer. Harmon. Anal.,
Birkh\"{a}user/Springer, Cham, 2023.

\vspace{-.3cm}

\bibitem{kock08}
A. Koc, H. M. Ozaktas, C. Candan and M. A. Kutay,
Digital computation of linear canonical transforms,
IEEE Trans. Signal Process. 56 (2008), 2383--2394.	
\vspace{-.3cm}

\bibitem{ko1998}
M. A. Kutay and H. M. Ozaktas,
Optimal image restoration with the
fractional Fourier transform,
J. Opt. Soc. Amer. A 15 (1998), 825--833.

\vspace{-.3cm}

\bibitem{la10}
K. Langley and S. J. Anderson,
The Riesz transform and simultaneous representations of
phase energy and orientation in spatial vision,
Vision Res.  50  (2010), 1748--1765.

\vspace{-.3cm}

\bibitem{lbo01}
K. G. Larkin, D. J. Bone and M. A. Oldfield,
Natural demodulation of two-dimensional fringe
patterns. I. General background of the spiral phase
quadrature transform,
J. Opt. Soc. Amer. A 18 (2001), 1862--1870.

\vspace{-.3cm}

\bibitem{ltw06}
B. Li, R. Tao and Y, Wang,
Hilbert transform associated with the linear canonical
transform, Acta Armamentarii 27 (2006), 827--830.

\vspace{-.3cm}

\bibitem{mac22}
K. Manab, P. Akhilesh and V. R. Kumar,
Multidimensional linear canonical transform and convolution,
J. Ramanujan Math. Soc. 37 (2022), 159--171.

\vspace{-.3cm}

\bibitem{mk87}
A. C. McBride and F. H. Kerr,
On Namias's fractional Fourier transforms,
IMA J. Appl. Math. 39 (1987), 159--175.

\vspace{-.3cm}

\bibitem{mq71}
M. Moshinsky and C. Quesne,
Linear canonical transformations and
their unitary representations,
J. Math. Phys. 12 (1971), 1772--1780.

\vspace{-.3cm}


\bibitem{mun02}
T. Musha, H. Uchida and M. Nagashima,
Self-monitoring sonar transducer array
with internal accelerometers,
IEEE J. Ocean. Eng. 27 (2002), 28--34.

\vspace{-.3cm}

\bibitem{n1980}
V.  Namias,
The fractional order Fourier transform and its application
to quantum mechanics,
IMA J. Appl. Math. 25 (1980), 241--265.

\vspace{-.3cm}

\bibitem{ozk01}
H. M. Ozaktas, Z. Zalevsky and M. A. Kutay,
The Fractional Fourier Transform: with Applications in
Optics and Signal Processing,
Wiley, New York, 2001.

\vspace{-.3cm}

\bibitem{pd2000}
S. C. Pei and J. J. Ding,
Closed-form discrete fractional and affine Fourier transforms,
IEEE Trans. Signal Process. 48 (2000), 1338--1353.

\vspace{-.3cm}

\bibitem{pd2002}
S. C. Pei and J. J. Ding,
Eigenfunctions of linear canonical transform,
IEEE Trans. Signal Process. 50 (2002),  11--26.

\vspace{-.3cm}

\bibitem{ppst23} G. Plonka, D. Potts, G. Steidl and M. Tasche, Numerical Fourier Analysis, Second edition, Applied and Numerical Harmonic Analysis, Birkh\"auser/Springer, Cham, 2023.

\vspace{-.3cm}

\bibitem{sko98}
A. Sahin, M. A. Kutay and H. M. Ozaktas,
Nonseparable two-dimensional fractional Fourier transform,
Appl. Optics 37 (1998), 5444--5453.

\vspace{-.3cm}

\bibitem{sdj11}
E. Sejdi\'{c}, I. Djurovi\'{c} and L. Stankovi\'{c},
Fractional Fourier transform as a signal processing
tool: An overview of recent developments,
Signal Process. 91 (2011), 1351--1369.

\vspace{-.3cm}

\bibitem{ss04}
L. Shen and Q. Sun,
Biorthogonal wavelet system for high-resolution
image reconstruction,
IEEE Trans. Signal Process. 52 (2004), 1997--2011.

\vspace{-.3cm}

\bibitem{sw1971}E. M. Stein and G. Weiss, 
Introduction to Fourier analysis on Euclidean spaces, 
Princeton University Press, Princeton, NJ, 1971.  

\vspace{-.3cm}

\bibitem{tlw10}
R. Tao, Y. Li and Y. Wang, Short-time fractional
Fourier transform and its applications,
IEEE Trans. Signal Process. 58 (2010), 2568--2580.

\vspace{-.3cm}

\bibitem{xwx09}
G. Xu, X. Wang and X. Xu,
Generalized Hilbert transform and its properties in 2D LCT domain,
Signal Process. 89 (2009), 1395--1402.

\vspace{-.3cm}

\bibitem{Yetik}
I. S. Yetik and A. Nehorai,
Beamforming using the fractional Fourier transform,
IEEE Trans. Signal Process. 51 (2003), 1663--1668.

\vspace{-.3cm}

\bibitem{z1998}
A. Zayed,
Hilbert transform associated with the fractional Fourier transform,
IEEE Signal Process. Lett. 5 (1998), 206--208.

\vspace{-.3cm}

\bibitem{z24}
A. Zayed,
Fractional Integral Transforms: Theory and Applications,
CRC Press, Abingdon, 2024.

\vspace{-.3cm}

\bibitem{z2018}
A. Zayed,
Two-dimensional fractional Fourier transform and some of
its properties,
Integral Transforms Spec. Funct.  29 (2018), 553--570.

\vspace{-.3cm}

\bibitem{z2019}
A. Zayed,
A new perspective on the two-dimensional fractional Fourier
transform and its relation with the Wigner distribution,
J. Fourier Anal. Appl. 25 (2019), 460--487.

\vspace{-.3cm}

\bibitem{zl18}
Y. Zhang and B. Li,  $\phi$-linear canonical analytic signals,
Signal Process. 143 (2018), 181--190.

\vspace{-.3cm}

\bibitem{yz24} Y. Zhou, Generalizations of the fractional Fourier transform and their analytic properties, arXiv: 2409.11201.

\end{thebibliography}
\end{document}